\documentclass[reqno,10pt,centertags]{amsart}
\usepackage{amsmath,amsthm}
\usepackage{enumerate}
\usepackage{mathrsfs, amssymb}
\usepackage{bbm}

\usepackage{color}

\newtheorem{thm}{Theorem}
\newtheorem{lem}[thm]{Lemma}
\newtheorem{cor}[thm]{Corollary}
\newtheorem{defi}[thm]{Definition}
\newtheorem{rem}[thm]{Remark}

\newtheorem{prop}[thm]{Proposition}
\newtheorem{quest}[thm]{Question}

\newtheorem{rema}[thm]{Remark}

\begin{document}

\title[Embedding of operator ideals into $\mathcal{L}_p-$spaces]{Embeddings of operator ideals into $\mathcal{L}_p-$spaces on finite von Neumann algebras}

\author{M.\ Junge}
\address{1409 West Green Street, Urbana, Illinois, 61801,USA}
\email{mjunge@illinois.edu}

\author{F.\ Sukochev}
\address{School of Mathematics and Statistics, UNSW, Kensington, NSW 2052,
Australia} 
\email{f.sukochev@unsw.edu.au}
%\email{f.sukochev@unsw.edu.au}

\author{D.\ Zanin} 
\address{School of Mathematics and Statistics, UNSW, Kensington, NSW 2052,
Australia} 
\email{d.zanin@unsw.edu.au}

\thanks{The first author is partially supported by NSF grant DMS-1201886. The second and third authors are supported by ARC}

\begin{abstract}
Let  $\mathcal{L}(H)$ be the $*$-algebra of all bounded operators on an infinite dimensional Hilbert space $H$ and let $(\mathcal{I}, \|\cdot\|_{\mathcal{I}})$ be an ideal in $\mathcal{L}(H)$ equipped with a Banach norm which is distinct from the Schatten-von Neumann ideal $\mathcal{L}_p(\mathcal{H})$, $1\leq p<2$. We prove that $\mathcal{I}$  isomorphically embeds into an $L_p$-space $\mathcal{L}_p(\mathcal{R}),$ $1\leq p<2,$ (here, $\mathcal{R}$ is the hyperfinite II$_1$-factor) if its commutative core (that is, Calkin space for $\mathcal{I}$) isomorphically embeds into $L_p(0,1).$ Furthermore, we prove that an Orlicz ideal $\mathcal{L}_M(H)\neq\mathcal{L}_p(H)$ isomorphically embeds into $\mathcal{L}_p(\mathcal{R}),$ $1\leq p<2,$ if and only if it is an interpolation space for the Banach couple $(\mathcal{L}_p(H),\mathcal{L}_2(H)).$
Finally, we consider isomorphic embeddings of $(\mathcal{I}, \|\cdot\|_{\mathcal{I}})$ into $L_p$-spaces associated with arbitrary finite von Neumann algebras.
\end{abstract}



\maketitle

\section{Introduction}

Classifying subspaces of Banach spaces is in general a very difficult task, and so far can only be achieved for very special examples of classical function spaces. A beautiful example is Dacuhna-Castelle characterization of subspaces of $L_1(0,1)$ with a symmetric basis, as average of Orlicz spaces. This characterization is genuinely based on the analysis of independent copies and, hence, is probabilistic in nature. Noncommutative function spaces appear naturally in operator theory and noncommutative geometry, and hence it is reasonable to expect a similar striking result for symmetric subspaces of noncommutative $L_1-$spaces or even noncommutative $L_p-$spaces in the interesting range for $1\leq p<2.$ Again M. Kadec discovery of $q-$stable random variables resulting in embedding $l_q$ into $L_p(0,1)$ serves as inspiration for our results. However, the noncommutative notion of independence is much more ambiguous in the noncommutative setting. Therefore
we prefer to use more functional analytic tools, in particular (several versions) of the Kruglov operator introduced in \cite{AS} (see \cite{LT2} for an alternative approach using Poisson processes going back to \cite{JMST}).

% The main theme of this article is best explained through references to the classical themes. It is well-known that the class of all subspaces of the Lebesgue space $L_1=L_1(0,1)$ is very rich and yet does not have any reasonable description. If we consider only symmetric subspaces of $L_1$, that is subspaces with a symmetric basis (symmetric sequence spaces),
% %or isomorphs of some r.i. function spaces,
% then these subspaces are known to be the averages of Orlicz spaces \cite{B3}. Far more information is available on subspaces of $L_1$ isomorphic to Orlicz sequence  spaces. It is important to note that an isomorph of an Orlicz sequence space $l_M\neq l_p$  in  $L_p=L_p(0,1)$ always can be  given \cite{ASorlicz} by the span of a sequence of independent identically distributed random variables. In the special case $l_M=l_q,$ $1\leq p<q\leq 2,$ this was discovered by  M.I. Kadec \cite{Kad} in 1958.

A detailed study of symmetric subspaces of $L_1$ was done by J. Bretagnolle and  D. Dacunha-Castelle (see \cite{BD, BD2, B3}). They proved that, for every given mean zero $f\in L_p$, the sequence $\{f_k\}_{k=1}^\infty$ of its independent copies is equivalent, in $L_p$, to the standard basis of some Orlicz sequence space $l_M$ \cite[Theorem 1, p.X.8]{B3}. %Moreover, J. Bretagnolle and  D. Dacunha-Castelle proved also that an Orlicz function space $L_M=L_M[0,1]$ can be isomorphically embedded into the space $L_p$, $1\leq p<2$, if and only if $M$ is equivalent to a $p$-convex and $2$-concave Orlicz function on $[0,\infty)$ \cite[Theorem IV.3]{BD2}.
Later some of these results were independently rediscovered by M. Braverman \cite{Br2,Br3}. Note that the methods used in \cite{BD, BD2, B3, Br2, Br3} depend heavily on the techniques related to the theory of random processes. In a recent paper \cite{ASorlicz}, a different approach based on methods and ideas from the interpolation theory of operators and the usage of so-called Kruglov operator \cite{as_zam,AS,as1,AS_UMN} is suggested.

The main topic of this paper is the study of (symmetric) subspaces of the noncommutative analogue of the space $L_p(0,1),$ i.e. the space $\mathcal{L}_p(\mathcal{R},\tau)$ of $\tau$-measurable affiliated with the hyperfinite II$_1$-factor $\mathcal{R}$ of finite $p$-norm. Indeed, the hyperfinite and finite von Neumann algebra $\mathcal{R}$ is equipped with a unique tracial state $\tau$ and the couple $(\mathcal{R},\tau)$ may be considered as a natural noncommutative analogue of the pair  $(L_{\infty}(0,1),dm)$ (here, $dm$ is the Lebesgue measure). In particular, the trace $\tau$  is normal and faithful (see next section for details). More generally, closed $\tau$-measurable operators affiliated to $\mathcal{R}$ form a $*$-algebra $\mathcal{S}(\mathcal{R},\tau)$ which is analogous to the algebra $L_0(0,1)$ of all (unbounded) Lebesgue measurable functions on $(0,1)$. The precise definition of the $L_p$-spaces associated with $\mathcal{R}$ mimics the classical one
$$\mathcal{L}_p(\mathcal{R}):=\{A\in \mathcal{S}(\mathcal{R},\tau):\ \tau(|A|^p)<\infty\}.$$
The separable space  $\mathcal{L}_p(\mathcal{R})$ plays a role in the class of noncommutative $L_p$-spaces associated with semifinite von Neumann algebras (see e.g. \cite{PX})  similar to that of the role of the space $L_p(0,1)$ in the class of all $L_p$-spaces on $\sigma$-finite measure spaces.

The class of symmetric subspaces of the space $\mathcal{L}_p(\mathcal{R})$ is naturally understood as subspaces isomorphic to Banach ideals in $\mathcal{L}(H),$ the only von Neumann factor of the type $I_\infty.$ Indeed, the latter class corresponds to that of symmetric sequence spaces \cite{KScrelle, LSZ}. For example, when symmetric sequence space in question is $l_q$, then the corresponding Banach ideal is the classical Schatten-von Neumann class $\mathcal{L}_q(H)$. Thus, the precise form of the main question studied in this article is as follows.

\begin{quest} Which Banach ideals in $\mathcal{L}(H)$ admit an isomorphic embedding into $\mathcal{L}_p(\mathcal{R})?$
\end{quest}

The following result taken from  \cite[Theorem 3.1]{sukp2}, \cite[Theorem 1.1 and Corollary 1.2]{hrs}, \cite{junge} encapsulates all available information to date.

\begin{thm}\label{cont} Let $\mathcal{R}$ be the hyperfinite II$_1$-factor. We have
\begin{enumerate}[{\rm (a)}]
\item\label{conta}  $\mathcal{L}_2(H)$ is the only Banach ideal in $\mathcal{L}(H)$ isomorphically embeddable into $\mathcal{L}_p(\mathcal{R})$ for $p\geq2.$
\item\label{contb} $\mathcal{L}_p(H)$ does not isomorphically embed into $\mathcal{L}_p(\mathcal{R})$ for $1\leq p<2.$
\item\label{contc} $\mathcal{L}_q(H)$ isomorphically embeds into $\mathcal{L}_p(\mathcal{R})$ for $1\leq p<q<2.$
\end{enumerate}
\end{thm}

In this paper, we investigate this question in the setting $1\leq p<2$. To state our main results, we need a notion of commutative core of an ideal, which is also frequently called a Calkin space (see e.g. \cite{LSZ}). Fix an orthonormal basis in $H$ (the particular choice of a basis is inessential) and identify the algebra $l_{\infty}$ of all bounded sequences with the subalgebra of all diagonal operators on $H$ with respect to this basis. Given an ideal $\mathcal{I}\in\mathcal{L}(H),$ we define its commutative core by $$I:=\mathcal{I}\cap l_{\infty}$$  (that is, the commutative core $I$ corresponding to $\mathcal{I}$ is the collection of all diagonal operators from $\mathcal{L}(H)$ belonging to $\mathcal{I}$). If $(\mathcal{I}, \|\cdot\|_{\mathcal{I}})$ is a Banach ideal, then its commutative core $I$ is equipped with the induced norm $\|x\|_{{I}}:=\|x\|_{\mathcal{I}}$, $x\in {I}.$ It can be shown that $(I, \|\cdot\|_I)$ is a symmetric sequence space. We are now able to state the main results of the paper.

\begin{thm}\label{main theorem} A separable Banach ideal $\mathcal{I}\neq\mathcal{L}_p(H)$ in $\mathcal{L}(H)$ admits an isomorphic embedding into $\mathcal{L}_p(\mathcal{R}),$ $1\leq p<2,$ provided that its commutative core $I$ admits an isomorphic embedding into $L_p(0,1).$
\end{thm}

The separability condition, as well as the condition $\mathcal{I}\neq\mathcal{L}_p(H)$ are obviously essential (see Theorem \ref{cont} \eqref{contb} above).

The result of Theorem \ref{main theorem} heavily relies on the classification of Orlicz ideals which admit an isomorphic embedding into $\mathcal{L}_p(\mathcal{R})$ and which is a perfect noncommutative analogue of results from \cite{BD, BD2, B3, ASorlicz}.

\begin{thm}\label{orlicz embedding theorem} Let $\mathcal{R}$ be the hyperfinite II$_1$-factor, let $\mathcal{L}_M(H)\neq\mathcal{L}_p(H)$ be an Orlicz ideal and let $1\leq p<2.$ The following conditions are equivalent.
\begin{enumerate}[{\rm (i)}]
\item\label{Oeti} The space $\mathcal{L}_M(H)$ isomorphically embeds into $\mathcal{L}_p(\mathcal{R}).$
\item\label{Oetii} The space $\mathcal{L}_M(H)$ is $p$-convex and $2$-concave.
\item\label{Oetiii} The Orlicz function $M$ on $(0,\infty)$ is equivalent on $[0,1]$ to some $p$-convex and $2$-concave Orlicz function.
\item\label{Oetiv} The space $\mathcal{L}_M(H)$ is an interpolation space for the couple $(\mathcal{L}_p(H),\mathcal{L}_2(H)).$
\end{enumerate}
\end{thm}

\begin{rema}
The condition $\mathcal{L}_M(H)\neq\mathcal{L}_p(H)$ guarantees that $M(t)\not\sim t^p$ as $t\to0.$
\end{rema}

% {\cre \begin{cor}\label{as theorem} Let $1\leq p<2$ and let $M$ be an Orlicz function on $(0,\infty)$ such that and $M(t)\not\sim t^p$ as $t\to0$. The following conditions are equivalent.
% \begin{enumerate}[{\rm (i)}]
% \item $M$ is $p$-convex, $2$-concave.
% \item {\cre $l_M$} is isomorphic to a subspace of $L_p(0,1).$
% \end{enumerate}
% \end{cor}}

%{\cre Another consequence of Theorem \ref{main theorem} yields a similar characterization of Lorents ideals $\Lambda_{\psi}^p(H)$ admitting isomorphic copies in  $\mathcal{L}_p(\mathcal{R})$. This may be seen as a noncommutative extension of Sch\"utt's result from \cite{Schutt} established there for isomorphs of Lorentz sequence spaces in $L_p(0,1)$.
%
%\begin{thm}\label{schutt} Let $\Lambda_{\psi}^p(H)\neq\mathcal{L}_p(H)$ be a Lorentz ideal and let $1\leq p<2.$ The following conditions are equivalent.
%\begin{enumerate}[{\rm (i)}]
%\item The space $\Lambda_{\psi}^p(H)$ isomorphically embeds into $\mathcal{L}_p(\mathcal{R}).$
%\item The space $\Lambda_{\psi}^p(H)$ is an interpolation space for the couple $(\mathcal{L}_p(H),\mathcal{L}_r(H))$ for some $r\in(p,2).$
%\end{enumerate}
%\end{thm}}

Our second main result may be viewed as a converse to Theorem \ref{main theorem} under an additional assumption that the hyperfinite factor $\mathcal{R}$ is replaced with an arbitrary finite von Neumann algebra $\mathcal{N}.$

\begin{thm}\label{second main theorem} For a separable Banach ideal $\mathcal{I}\neq\mathcal{L}_p(H)$ in $\mathcal{L}(H),$ the following conditions are equivalent
\begin{enumerate}[{\rm (i)}]
\item\label{ji} $\mathcal{I}$ admits an isomorphic embedding into $\mathcal{L}_p(\mathcal{M},\sigma),$ $1\leq p<2,$ for some finite von Neumann algebra $(\mathcal{M},\sigma).$
\item\label{jii} The commutative core $I$ of the ideal $\mathcal{I}$ admits an isomorphic embedding into $\mathcal{L}_p(\mathcal{N},\tau),$ $1\leq p<2,$ for some finite von Neumann algebra $(\mathcal{N},\tau).$
\end{enumerate}
\end{thm}

\begin{rema}
We do not claim that the von Neumann algebras in \eqref{ji} and \eqref{jii} are the same.
\end{rema}

The following question remains open.

\begin{quest} Suppose $\mathcal{I}$ admits an isomorphic embedding into $\mathcal{L}_p(\mathcal{M},\sigma),$ $1\leq p<2,$ for some finite von Neumann algebra $(\mathcal{M},\sigma).$ Is it true that the commutative core $I$ of the ideal $\mathcal{I}$ admits an isomorphic embedding into $L_p(0,1)?$
\end{quest}

Finally, we are compelled to say a few words about our techniques and methods, which we believe are of wider applicability and of interest in their own right. Our main technical instrument is the Kruglov operator (used, e.g. in the proof of the most difficult implication \eqref{Oetiii}$\Rightarrow$\eqref{Oeti} in Theorem \ref{orlicz embedding theorem}) which is a replacement of integration with respect to the Poisson process (see e.g. \cite{LT2}). In the special setting of the commutative von Neumann algebra $(L_{\infty}(0,1),dm)$, such an operator was firstly introduced in \cite{as_zam,AS} (see also the survey \cite{AS_UMN}). In the
noncommutative setting, this operator is defined and studied in this paper. In Section \ref{krug fin}, we introduce the Kruglov operator affiliated with finite von Neumann algebras, basically mimicking the approach from \cite{AS}. In Section \ref{junge section}, we introduce such an operator affiliated already with a semifinite non-finite von Neumann algebra and this is an important extension from our viewpoint. We emphasize that the germ of our approach to this definition can be found already in \cite{junge}, although our present presentation and approach are quite different from those given in \cite{junge}. Theorem \eqref{second main theorem} makes crucial, and extends, the work of Kwapien and Sch\"utt \cite{KwS}.

\section{Preliminaries}

Let $H$ be a complex infinite-dimensional separable Hilbert space equipped with an inner product $\langle\cdot,\cdot\rangle_H$ and $\mathcal{L}(H)$ be the $C^*$-algebra of all bounded linear operators on $H$ equipped with the operator norm $\|\cdot\|_\infty$. We denote by ${\rm Tr}$ the classical trace on $\mathcal{L}(H)$.
For every compact operator $A\in\mathcal{L}(H),$ let $\mu(A)=\{\mu(k,A)\}_{k\geq0}$ be the decreasing sequence of singular values of the operator $A,$ counted with multiplicities.

 For $A\in\mathcal{L}(H)$ we use the following standard notations
\begin{equation*}
|A|:=\sqrt{A^*A},\quad\Re A:=\frac{A+A^*}{2},\quad\Im A:=\frac{A-A^*}{2i},
\end{equation*}
and for $A=A^*\in\mathcal{L}(H)$ we set
\begin{equation*}
A_+:=\frac{|A|+A}2,\quad A_-:=\frac{|A|-A}{2}.
\end{equation*}
The notations $\mathcal M$ and $\mathcal N$ usually stand for finite or semifinite von Neumann algebras. By $\mathbbm 1_{\mathcal M}$ (resp. $\mathbbm 1_{\mathcal N}$) we denote the unit element of the algebra $\mathcal M$ (resp. $\mathcal N$). For brevity we use simply $\mathbbm 1$ instead of both $\mathbbm 1_{\mathcal M}$ or $\mathbbm 1_{\mathcal N}.$ This should not lead to any confusion.

For positive real numbers $a,b$ and some constant $C>0$ we write $a {\sim} b$, if the inequality $Ca\le b\le a/C$ holds. When we need to specify the constant $C$, we will use the notation $a \stackrel{C}{\sim} b.$

\subsection{Calkin correspondence}

Fix an orthonormal basis in $H$ (the particular choice of a basis is inessential). We identify the algebra $l_{\infty}$ of bounded complex sequences with the subalgebra of all diagonal operators in $\mathcal{L}(H)$ with respect to the chosen basis. Given an ideal $\mathcal{I}\in\mathcal{L}(H),$ we define its \textit{commutative core} $I=\mathcal{I}\cap l_{\infty}$ (that is, the commutative core is the collection of all diagonal operators in $\mathcal{I}$). A famous result due to Calkin \cite{Calkin} asserts that one can recover the original ideal $\mathcal{I}$ by the formula
$$\mathcal{I}=\{A\in\mathcal{L}(H):\ \mu(A)\in I\}.$$

The crucial property of the commutative core $I$ is that,
\begin{equation}\label{eq_ideal}x\in I,\ \ \mu(y)=\mu(x)\Longrightarrow y\in I.\end{equation}
In fact, Calkin theorem \cite{Calkin} states that every ideal in $l_{\infty}$ with property \eqref{eq_ideal} is a commutative core of some ideal in $\mathcal{L}(H).$

Whenever an ideal $\mathcal{I}$ in $\mathcal{L}(H)$ is equipped with a Banach norm\footnote{For definiteness, we always assume that $\|A\|_{\infty}\leq\|A\|_{\mathcal{I}}$ for every $A\in\mathcal{I}.$} $\|\cdot\|$ such that the multiplication $A\to AB$ is continuous mapping on $\mathcal{I}$ for every $B\in\mathcal{I},$ we call $(\mathcal{I},\|\cdot\|)$ a \textit{Banach ideal}.
For every Banach ideal $\mathcal{I}$ in $\mathcal{L}(H),$ there exists an equivalent unitarily invariant norm on $\mathcal{I}$ given by the formula
$$\|A\|_{\mathcal{I}}=\sup_{\substack{B,C\in\mathcal{L}(H)\\ \|B\|_{\infty},\|C\|_{\infty}\leq1}}\|BAC\|.$$
That is, we have $\|AU\|_{\mathcal{I}}=\|UA\|_{\mathcal{I}}=\|A\|_{\mathcal{I}}$ for all $A\in\mathcal{I}$ and for all unitary $U\in\mathcal{L}(H).$ In particular, we have
$$A\in\mathcal{I},\ \mu(B)=\mu(A)\Longrightarrow B\in\mathcal{I},\ \|B\|_{\mathcal{I}}=\|A\|_{\mathcal{I}}.$$
It follows immediately that
\begin{equation}\label{eq_ideal1}
x\in I,\ \ \mu(y)=\mu(x)\Longrightarrow y\in I,\ \ \|y\|_I=\|x\|_I,
\end{equation}
where  $\|\cdot\|_I$  is the norm on $I$ induced by $\|\cdot\|_{\mathcal{I}}.$ An arbitrary Banach ideal $I$ in $l_{\infty}$ equipped with a norm $\|\cdot\|_I$ satisfying \eqref{eq_ideal1} is called \textit{symmetric sequence space} (see \cite{LT1,LSZ}).

The question whether Calkin correspondence specializes to symmetric sequence spaces and Banach ideals in $\mathcal{L}(H)$ was fully resolved in \cite{KScrelle} (see also \cite{LSZ}).

\begin{thm}\cite[Theorem 8.11]{KScrelle}\label{KStheorem} If $(I,\|\cdot\|_I)$ is a symmetric sequence space, then the corresponding ideal $\mathcal{I}$ in $\mathcal{L}(H)$ equipped with a functional
$$\|\cdot\|_{\mathcal{I}}:A\to\|\mu(A)\|_I,\quad A\in\mathcal{I}$$
is a Banach ideal in $\mathcal{L}(H).$
\end{thm}

A variant of Calkin correspondence exists for arbitrary semifinite von Neumann algebra and symmetric function spaces on $(0,\infty)$ (see \cite{LT2,LSZ}). Theorem \ref{KStheorem} also works in this generality (see \cite{KScrelle,LSZ}). However, in this paper, we deal only with $\mathcal{L}_p-$spaces on semifinite von Neumann algebras, which can be defined without resorting  to the general machinery of \cite{KScrelle,LSZ} (see Subsection \ref{type subsect} below).

Let $m\ge 1$ and let $\alpha\in\mathbb{C}^m$. We shall frequently view a finite sequence $\alpha=(\alpha(0),\dots, \alpha({m-1}))$ as an element of symmetric space $I$ by identifying $\alpha$ with  an element $(\alpha(0),\dots, \alpha({m-1}),0,\dots)$.

\subsection{Banach ideals in $\mathcal{L}(H):$ examples}

The best known example of a Banach ideal in $\mathcal{L}(H)$ is a
Schatten-von Neumann ideal $\mathcal{L}_p(H),$ $1\leq p<\infty,$
defined by setting
$$\mathcal{L}_p(H)=\{A\in\mathcal{L}(H):\ \mu(A)\in l_p\},\quad \|A\|_p=\|\mu(A)\|_p,$$
where $l_p$ is a Banach space of all $p$-summable sequences and $\|\cdot\|_p$ is the standard norm on $l_p.$

Given an Orlicz function $M$ (that is, an even convex function on $\mathbb{R}$ such that $M(0)=0$), we define an Orlicz sequence ideal $l_M$ in $l_{\infty}$ as the space of all bounded sequences $x$ such that
$$M\left(\frac{x}{\lambda}\right)\in l_1$$
for some $\lambda>0.$ An Orlicz sequence ideal equipped with the norm
$$\|x\|_{l_M}=\inf\left\{\lambda:\ \left\|M\left(\frac{x}{\lambda}\right)\right\|_1\leq 1\right\}$$
is a symmetric sequence space. An Orlicz ideal $\mathcal{L}_M(H)$ in $\mathcal{L}(H)$ is now defined by setting
$$\mathcal{L}_M(H)=\{A\in\mathcal{L}(H):\ \mu(A)\in l_M\},\quad \|A\|_{\mathcal{L}_M}=\|\mu(A)\|_{l_M}.$$

%Let $\psi:(0,\infty)\to(0,\infty)$ be a concave increasing function such that $\psi(+0)=0.$ A Lorentz sequence ideal $\Lambda_{\psi}^p$ in $l_{\infty}$ is defined by setting
%$$\Lambda_{\psi}^p=\{x\in l_{\infty}:\ \sum_{k=0}^{\infty}\mu^p(k,x)(\psi(k+1)-\psi(k))<\infty\}.$$
%A Lorentz sequence ideal $\Lambda_{\psi}^p$ equipped with a norm
%$$\|x\|_{\Lambda_{\psi}^{\cb p}}=\Big(\sum_{k=0}^{\infty}\mu^p(k,x)(\psi(k+1)-\psi(k))\Big)^{1/p}$$
%is a symmetric sequence space. A Lorentz ideal $\Lambda_{\psi}^{\cb p}(H)$ is defined as follows.
%$$\Lambda_{\psi}^{\cb p}(H)=\{A\in\mathcal{L}(H):\ \mu(A)\in\Lambda_{\psi}^{\cb p}\},\quad \|A\|_{\Lambda_{\psi}^p(H)}=\|\mu(A)\|_{\Lambda_{\psi}^p}.$$

\subsection{Dilation operator}

let us recall that  for $x=\{x(k)\}_{k\ge 0}\in l_\infty,$ the discrete \textit{
dilation operators} is defined as follows
\begin{equation}\label{eq_D_discrete}D_n(x):=\{\underbrace{x(0),\ldots,x(0)}_{\mbox{$n$ times}},\underbrace{x(1),\ldots,x(1)}_{\mbox{$n$ times}},\ldots\},\quad D_{1/n}(x):=\Big\{\frac1n\sum_{l=kn}^{(k+1)n-1}x(l)\Big\}_{k\geq0}\end{equation}
and satisfies
\begin{equation}\label{eq_D1}\|D_n\|_{l_p\to l_p}=n^{1/p},\ \  \|D_{1/n}\|_{l_p\to l_p}=n^{-1/p}, \ \ \ 1\le p<\infty.\end{equation}

The continuous version of dilation operator is defined below.
For every $u\in(0,\infty),$ the dilation operator $D_u$ on $L_1(0,\infty)$  is defined by
\begin{equation}\label{eq_D_continuous}(D_ux)(s):=x\Big(\frac{s}{u}\Big),\quad s\in(0,\infty),\quad x\in L_1(0,\infty).\end{equation}

For $u\in(0,1),$ the dilation operator $D_u$ on $L_1(0,1)$  is defined by
$$(D_ux)(s):=\left\{\begin{array}{cl}x(\frac{s}{u}),& s\in(0,u)\\ 0,& s\in [u,1)\end{array}\right., \quad x\in L_1(0,1).$$
For the general properties of dilation operator see \cite{LT2}.

\subsection{Singular value function}

In this paper, $\mathcal{M}$ always stands for a semifinite von Neumann algebra equipped with a faithful normal semifinite trace $\tau.$ Denote by $\mathcal{P(M)}$ the lattice of all projections from $\mathcal{M}$ and by $\mathcal{P}_{fin}(\mathcal{M})$ the set of all $\tau$-finite projections. An unbounded densely defined operator is said to be affiliated with $\mathcal{M}$ if it commutes with all elements in the commutant of $\mathcal{M}.$ A closed densely defined operator $A$ affiliated with $\mathcal{M}$ is called \textit{$\tau$-measurable} if, for every $n\geq0,$ we have $\tau(E_{|A|}(n,\infty))\to 0$ as $n\to\infty$, where by $E_{|A|}(n,\infty)$ we denote the spectral projection of the operator $|A|.$ If, in particular, $\tau(\mathbbm{1})<\infty,$ then every closed operator affiliated with $\mathcal{M}$ is $\tau$-measurable. We denote the $^*$-algebra of all $\tau$-measurable operators by $\mathcal{S}(\mathcal{M},\tau)$ with the strong sum and the strong product (see \cite{FackKosaki,LSZ}). By $t_\tau$ we denote the measure topology on $\mathcal{S}(\mathcal{M},\tau).$ Observe that the algebra $\mathcal N$ is dense in $\mathcal{S}(\mathcal{M},\tau)$ with respect to the measure topology $t_\tau$ (see~e.g.~\cite{ddp} and references therein). For $x=x^*\in\mathcal{S}(\mathcal{M},\tau)$ by
$\exp(ix)$ we denote the series $\sum_{j=0}^\infty\frac{(ix)^j}{j!},$ which converges in the measure topology.

For every $A\in\mathcal{S}(\mathcal{M},\tau),$ its singular value function $\mu(A):t\to\mu(t,A),$ $t>0,$ is defined by the formula (see e.g. \cite{FackKosaki})
$$\mu(t,A):=\inf\{\|A(\mathbbm{1}-p)\|_{\infty}:\ \tau(p)\leq t\}.$$
Equivalently, the function $\mu(A)$ can be defined in terms of the distribution function $d_{|A|}$ of $|A|.$ For an arbitrary self-adjoint $B\in\mathcal{S}(\mathcal{M},\tau),$ its distribution function $d_B$ is defined by setting
$$d_B(s):=\tau(E_B(s,\infty)),\quad s\in\mathbb{R},$$
where $E_B$ denotes the spectral measure of the operator $B.$ One can define $\mu(A)$ as the right inverse of the function $d_{|A|},$ i.e.
$$\mu(t;A)=\inf\{s: d_{|A|}(s)\leq t\},\quad t\ge0.$$
In addition, if $A_0,\dots,A_{n-1}\in S(\mathcal{M},\tau)$, then (see e.g. \cite{LSZ})
\begin{equation}\label{kps dilate ineq}
\mu\Big(\sum_{i=0}^{n-1}\mu(A_i)\Big)\leq D_n\sum_{i=0}^{n-1}\mu(A_i).
\end{equation}
When $\mathcal{M}=\mathcal{L}(H)$ and $\tau$ is the standard trace ${\rm Tr},$ and when the operator $A$ is compact, the function $\mu(A)$ is a step function whose values are the singular values of the operator $A.$

We call the operators $A=A^*,B=B^*\in\mathcal{S}(\mathcal{M},\tau)$ \textit{equimeasurable} if $d_{A_+}=d_{B_+}$ and $d_{A_-}=d_{B_-}.$ Observe that if $\tau(\mathbbm{1})<\infty,$ then operators $A=A^*,B=B^*\in\mathcal{S}(\mathcal{M},\tau)$ are equimeasurable if and only if $d_A=d_B.$ For equimeasurable operators $A=A^*,B=B^*\in\mathcal{S}(\mathcal{M},\tau),$ we also have $\mu(A)=\mu(B).$
The notion of equimeasurability for the operators from different algebras is defined similarly.
Finally, we call the operator $A=A^*\in\mathcal{S}(\mathcal{M},\tau)$ \textit{symmetrically distributed} if it is equimeasurable with the operator $-A.$

For every $A\in\mathcal{S}(\mathcal{M},\tau),$ the spectral projection $E_{|A|}(0,\infty)$ is called the \textit{support projection of $A$} (written ${\rm supp}(A)$). If $\tau({\rm supp}(A))<\infty,$ then we call $A$ \textit{finitely supported}.

In the special case when $\mathcal{M}=L_{\infty}(0,1)$ and $\tau$ is a Lebesgue integral, the definition of distribution function (and related definition of equimeasurability) coincides with the classical one. In this case the singular valued function $\mu(A)$ defined above coincides with a classic notion of decreasing rearrangement (see e.g. \cite{LT1}).

The following notion is a fundamental
concept in classical symmetric function space theory (we refer the interested reader to the book \cite[Chapter II]{LSZ} for detailed discussion of this important notion).
\begin{defi} \label{hlp majorization def first} Let $\mathcal{M}$ be a semifinite von Neumann algebra and let $A,B\in \mathcal{S}(\mathcal{M},\tau).$ The operator $B$ is said to be submajorized (in the sense of Hardy, Littlewood and Polya) by $A$ and written $B\prec\prec A$ if
$$\int_0^t\mu(s,B)ds\leq\int_0^t\mu(s,A)ds,\quad t\geq0.$$
\end{defi}
When $\mathcal{M}$ is the
algebra $l_\infty$ ($L_\infty(0,1),$ $L_\infty(0,\infty)$) the  definition above reduces to the
classical notion of submajorization of sequences (functions).

The following property of submajorization (see
e.g.~\cite[II.6.1]{KPS}) is used below
\begin{equation}\label{sum maj}
\sum_{k=1}^{\infty}x_k\prec\prec\sum_{k=1}^{\infty}\mu(x_k),\quad x_k\in L_{\infty}(0,\infty) \ (x_k\in L_{\infty}(0,1)), \ \ k\ge 1.
\end{equation}
Here, it is assumed that the series on the right hand side converges in $L_{\infty}(0,\infty)$ (in $L_{\infty}(0,1)$). This guarantees that the series in the left hand side converges in $L_{\infty}(0,\infty)$ (in $L_{\infty}(0,1)$).

%The connection of this definition with classical submajorization theory (that is, for functions or sequences) is immediate: we have
%$$B\prec\prec A\Leftrightarrow\mu(B)\prec\prec\mu(A).$$

\subsection{Properties of the Orlicz norm}

For the general theory of Orlicz spaces we refer to \cite{R-R}. Let us present the following simple property that is used frequently in the text. If two Orlicz functions $M_1$ and $M_2$ satisfies
$$M_1(t)\le M_2(t),\quad t\in(0,\infty),$$
then for any $x\in l_{M_2}$, we have that $x\in l_{M_1}$ and $\|x\|_{l_{M_1}}\le \|x\|_{l_{M_2}}.$

Two Orlicz functions $M_1$ and $M_2$ on $(0,\infty)$ are said to
be equivalent on $[0,1]$ ($M_1\sim M_2$), if there exists a
positive finite constant $C$ such that
$$C^{-1}M_1(t)\le M_2(t)\le CM_1(t), \ \ \ t\in[0,1].$$

An Orlicz function $M$ on $(0,\infty)$ is said to be
\begin{enumerate}[{\rm (i)}]
\item $p$-convex if the function $t\to M(t^{1/p}),$ $t>0,$ is convex.
\item $q$-concave if the function $t\to M(t^{1/q}),$ $t>0,$ is concave.
\end{enumerate}

%The following assertion is proved in \cite{ASorlicz}. For convenience of the reader we provide a short explanation.
The proof of lemma below follows immediately from \cite[Lemma
5]{ASorlicz} by substituting $st=s_1$ and $t=s_2$ in that lemma.

\begin{lem}\label{as lemma} Let $1\le p<q<\infty.$ An Orlicz function $M$ on $(0,\infty)$ is equivalent on  $[0,1]$ to a $p$-convex and $q$-concave Orlicz function  if and only if it satisfies the conditions
\begin{equation}\label{pconv crit}
\frac{M(s_1)}{s_1^p}\leq{\rm const}\cdot\frac{M(s_2)}{s_2^p},\quad
0< s_1\leq s_2\leq 1.
\end{equation}
\begin{equation}\label{2conc crit}
\frac{M(s_1)}{s_1^q}\geq{\rm const}\cdot\frac{M(s_2)}{s_2^q},\quad
0< s_1\leq s_2\leq 1.
\end{equation}
If $M$ itself is $p$-convex and $q$-concave, then the constants in
\eqref{pconv crit} and \eqref{2conc crit} are equal to $1.$
\end{lem}

\begin{rem}\label{rem_Orl_mon_submaj}
Every Orlicz space has the Fatou property (see p.64 in
\cite{KPS}). Every symmetric sequence space with the Fatou
property coincides with its second K\"othe dual (see p.65 in
\cite{KPS}). In particular, the norm in such a space is monotone
with respect to the submajorization.
\end{rem}

\subsection{$p$-concavity, $q$-convexity, upper $p$-estimate and lower $q$-estimate of Banach ideals}

 Let $1\leq p,q<\infty$.  The Banach ideal
$\mathcal I\subseteq \mathcal L(H)$ is said to be {\it $p$-convex}
(respectively, {\it $q$-concave}) if there exists a constant $M>0$
such that, for any finite sequence $\{x_k\}_{k=0}^n\subseteq
\mathcal I$,
\begin{equation}
\label{eqnpconvex} \Big \Vert \Big(\sum_{k=0}^n\vert x_k\vert
^p\Big)^{1/p}\Big \Vert_{\mathcal I} \leq M\Big(\sum_{k=0}^n\Vert
x_k\Vert _{\mathcal I}^p\Big)^{1/p},
\end{equation}
respectively,
\begin{equation}
\label{eqnqconcave} \Big(\sum_{k=0}^n\Vert x_k\Vert _{\mathcal
I}^q\Big)^{1/q} \leq M\Big \Vert \Big(\sum_{k=0}^n\vert x_k\vert
^q\Big)^{1/q}\Big \Vert_{\mathcal I}.
\end{equation}
These definitions are similar to
 the usual notions of $p$-convexity and $q$-concavity for Banach lattices as given in ~\cite{LT2}.

 %However, the definitions given there
% do not make sense in the non-commutative setting in the case that $p=q=\infty$.

 If $x\in \mathcal L(H)$, then the {\it
right} and {\it left support projections} of $x$, denoted by
$r(x)$ and $l(x)$ respectively, are the projections onto the
orthogonal complements of the kernels of $x,x^*$ respectively. The
operators $x,y\in \mathcal L(H)$ are said  to be {\it left}
(respectively, {\it right}) {\it disjointly supported} if
$r(x)r(y)=0$ (respectively, $l(x)l(y)=0$).

We now make the following definition. Let $1\leq p,q<\infty$.  The
Banach ideal $\mathcal I\subseteq \mathcal L(H)$ is said to
satisfy an {\it upper} $p$-estimate (respectively, {\it lower}
$q$-estimate) if there exists a constant $K>0$ such that, for
every finite sequence $x_0,\dots,x_n$ of pairwise left disjointly
supported elements of $\mathcal I$, it follows that
\begin{equation}\label{eq__up_p_est}
\Big\Vert \sum_{k=0}^nx_i\Big\Vert _{\mathcal I} \leq
K\Big(\sum_{k=0}^n \Vert x_k\Vert _{\mathcal I}^p\Big)^{1/p}
\end{equation}
respectively,
\begin{equation}\label{eq_low_q_est}
K\Big\Vert \sum_{k=0}^nx_i\Big\Vert _{\mathcal I} \geq
\Big(\sum_{k=0}^n \Vert x_k\Vert _{\mathcal I}^q\Big)^{1/q}.
\end{equation}
We remark that \lq\lq  left\rq\rq\ can be replaced throughout by
\lq\lq right\rq\rq\  since $x,y$ are left disjointly supported if
and only if $x^*,y^*$ are right disjointly supported. For
classical (commutative) counterparts of these notions for Banach
lattices we again refer to \cite{LT2}. For the study of
connections between just defined properties of $\mathcal I$ and

corresponding classical properties of its commutative core $I$ we
refer to  \cite{AL} and \cite{dds}.

\subsection{Type and cotype inequalities in the noncommutative $\mathcal{L}_p-$spaces}\label{type subsect}

A noncommutative $L_p$-space associated with an algebra $(\mathcal{M},\tau)$ (written $\mathcal{L}_p(\mathcal{M})$) is defined by setting
$$\mathcal{L}_p(\mathcal{M})=\{A\in\mathcal{S}(\mathcal{M},\tau):\ \tau(|A|^p)<\infty\},\quad \|A\|_p=\tau(|A|^p)^{1/p}.$$
Observe that $A\in\mathcal{L}_p(\mathcal{M})$ if and only if $\mu(A)$ belongs to the Lebesgue space $L_{p}(0,\infty).$ We also have $\|A\|_p=\|\mu(A)\|_p$ for every $A\in\mathcal{L}_p(\mathcal{M}).$

For reader's convenience, we show the elementary right hand side estimate of \eqref{eq_Th8} below. The (substantially harder) proof of the left hand side estimate is due to Tomczak-Jaegermann and can be found in \cite{tj} (see also \cite{Fack}). For brevity, we state this result using the language and notations from the theory of tensor products of von Neumann algebras. Recall that the Rademacher functions are defined as follows
\begin{equation}\label{eq_Rad}r_k(t):=\mathrm{sign} \sin (2^{k+1}\pi t), \ \ \ t\in [0,1], \ \ \ k=0,1,2, \ldots\end{equation}
(see also \cite{LT1,LT2}).

\begin{thm}\label{lp theorem} Let $n\ge 0,$ let $\mathcal{M}$ be a semifinite von Neumann algebra and let $A_k\in\mathcal{L}_p(\mathcal{M}),$ $0\leq k\leq n.$ If $1\leq p<2,$ then
\begin{equation}\label{eq_Th8}{\rm const}\cdot\Big(\sum_{k=0}^{n}\|A_k\|_p^2\Big)^{1/2}\leq\Big\|\sum_{k=0}^{n}A_k\otimes r_k\Big\|_p\leq\Big(\sum_{k=0}^{n}\|A_k\|_p^p\Big)^{1/p}.\end{equation}
Here $r_k,$ $k\ge 0,$ are Rademacher functions and the norm in the
middle is the usual $L_p$-norm in the $L_p$-space
$\mathcal{L}_p(\mathcal{M}\bar{\otimes}L_{\infty}(0,1),\tau\otimes
dm).$
\end{thm}
\begin{proof}[Proof of the right hand side of \eqref{eq_Th8}]  For $p=1,$ it follows from the triangle inequality in $\mathcal{L}_1(\mathcal{M}\bar{\otimes}L_{\infty}(0,1))$ that
$$\Big\|\sum_{k=0}^nA_k\otimes r_k\Big\|_1\leq\sum_{k=0}^n\|A_k\|_1,\quad n\geq 0.$$
For $p=2,$ it is immediate that
$$\Big\|\sum_{k=0}^nA_k\otimes r_k\Big\|_2=\Big(\sum_{k=0}^n\|A_k\|_2^2\Big)^{1/2},\quad n\geq 0.$$
The right hand side inequality of \eqref{eq_Th8} now follows by applying complex interpolation method (see e.g. \cite{Kosaki}).% We refer the reader to \cite{tj,Fack} for the left hand side inequality.
\end{proof}

\subsection{Independent (noncommutative) random variables}\label{indep subsect} Let $\mathcal{M}$ be a finite von Neumann algebra equipped with a faithful normal trace such that $\tau(\mathbbm{1})=1,$ where we denote the unit element of $\mathcal{M}$ by $\mathbbm{1}.$ Self-adjoint operators affiliated with $\mathcal{M}$ are called (noncommutative) \textit{random variables}.

Commuting unital subalgebras $\mathcal{A}_k,$ $1\leq k\leq n,$ of
$\mathcal{M}$ are called \footnote{This definition is taken from
\cite{NSp} (see Definition 5.1 there).} \textit{independent} if
$$\tau(A_1A_2\cdots A_n)=\tau(A_1)\tau(A_2)\cdots\tau(A_n)$$
whenever $A_k\in\mathcal{A}_k,$ $1\leq k\leq n.$ Random variables $A_k\in\mathcal{S}(\mathcal{M},\tau),$ $1\leq k\leq n,$ are called \textit{independent} if the von Neumann subalgebras $\mathcal{A}_k,$ $1\leq k\leq n,$ generated by spectral projections of $A_k$ are independent.

%Later, we also need a much weaker notion of conditional independence. Unital subalgebras (not necessarily commuting) $\mathcal{A}_k,$ $1\leq k\leq n$ of $\mathcal{M}$ are called \textit{conditionally independent} over {\cb a} von Neumann subalgebra $\mathcal{M}_0$ if
%$$\mathscr{E}_{\mathcal{M}_0}(xy)=\mathscr{E}_{\mathcal{M}_0}(x)\mathscr{E}_{\mathcal{M}_0}(y),\quad \mathscr{E}_{\mathcal{M}_0}(yx)=\mathscr{E}_{\mathcal{M}_0}(y)\mathscr{E}_{\mathcal{M}_0}(x)$$
%for every $x\in\mathcal{A}_k$ and for every $y$ in the subalgebra generated by all $\mathcal{A}_l,$ $l\neq k.$ Here, $\mathscr{E}_{\mathcal{M}_0}$ denotes the conditional expectation onto the von Neumann subalgebra $\mathcal{M}_0.$
%{\cb For the definition of conditional expectation we refer to \cite{U}.}

\subsection{Quantization of the algebra $\mathcal{M}$}\label{quant subsect}

For a finite (semifinite) von Neumann algebra we use the following notations.
$$\mathcal{M}^{\otimes k}=\underbrace{\mathcal{M}\bar{\otimes}\cdots\bar{\otimes}\mathcal{M}}_{\mbox{$k$ times}},\quad \tau^{\otimes k}=\underbrace{\tau\otimes\cdots\otimes\tau}_{\mbox{$k$ times}},\quad k\in\mathbb{N},$$
for von Neumann algebra tensor products. It is a standard fact (see e.g. Proposition 11.2.20 and Proposition 11.2.21 in \cite{KR}) that $\mathcal M^{\otimes k}$ is a finite (semifinite) von Neumann algebra and $\tau^{\otimes k}$ is a faithful normal finite (semifinite) trace on $\mathcal M^{\otimes k}$. For convenience of notations, we set $\mathcal{M}^{\otimes 0}$ to be $\mathbb{C}$ and $\tau^{\otimes 0}:\mathbb{C}\to\mathbb{C}$ to be identity mapping. Define the positive linear mapping
\begin{equation}\label{eq_E_k}
E_k:\mathcal{M}^{\otimes k}\to\mathcal{M}^{\otimes k},\quad E_k(x_1\otimes\cdots\otimes x_k):=\frac1{k!}\sum_{\rho}x_{\rho(1)}\otimes\cdots\otimes x_{\rho(k)},\quad k\in\mathbb{N},
\end{equation}
where $\rho$ runs over the set $\mathfrak{S}_k$ of all permutations of $\{1,\cdots,k\}.$ It can be directly verified that
$$\mathcal{M}_s^{\otimes k}:=E_k(\mathcal{M}^{\otimes k})$$
is a von Neumann subalgebra and that $E_k:\mathcal{M}^{\otimes k}\to\mathcal{M}^{\otimes k}_s$ is the conditional expectation. For convenience of notations, we set $\mathcal{M}^{\otimes 0}_s$ to be $\mathbb{C}.$

Define the \textit{(symmetric or bosonic) quantized von Neumann algebra}
 $M_s(\mathcal{M})$ by setting
\begin{equation}\label{eq_M_s}
M_s(\mathcal{M}):=\bigoplus_{k=0}^{\infty}\mathcal{M}^{\otimes k}_s.
\end{equation}
if $\tau(\mathbbm{1})<\infty,$ then the expression
\begin{equation}\label{eq_sigma}\sigma:=\exp(-\tau(\mathbbm{1}))\bigoplus_{k=0}^{\infty}\frac1{k!}\tau^{\otimes k}\end{equation}
defines a trace on $\bigoplus_{k=0}^{\infty}\mathcal{M}^{\otimes k}$ (and, obviously on the $*$-subalgebra  $M_s(\mathcal{M})$). Observe that $\sigma(\mathbbm{1})=1.$
For a $\tau$-finite projection $e\in\mathcal M$ we will also consider the algebra $M_s(e\mathcal{M}e)$  equipped with the trace
\begin{equation}\label{eq_sigma_e}\sigma_e:=\exp(-\tau({e}))\bigoplus_{k=0}^{\infty}\frac1{k!}\tau^{\otimes k}.\end{equation}
Note that $\sigma_e$ is defined on $\bigoplus_{k=0}^{\infty}(e\mathcal{M}e)^{\otimes k}.$ By $p_k$ we denote natural projections of the algebra $M_s(\mathcal{M})$ given by
\begin{equation}\label{eq_p_k}
p_k:=0\oplus \ldots\oplus 0\oplus {\mathbbm 1^{\otimes k}}\oplus 0\oplus\ldots \in M_s(\mathcal{M}),\quad k\ge 0.
\end{equation}
We identify the algebra $\mathcal{M}^{\otimes k}_s$ with a subalgebra of $ M_s(\mathcal{M})$
by means of the projection $p_k$.

\subsection{Inductive limits of $C^*$-algebras}

We refer the reader to the books \cite{Sakai} and \cite{Murphy} for the details of the inductive limit construction used in Section \ref{junge section} (note that in \cite{Murphy} the notion of direct limit is used instead of inductive limit).

\begin{lem}\cite[Proposition 1.23.2]{Sakai}
\label{dirlimdef1} Let $\{\mathcal{A}_n\}_{n\geq0}$ be $C^*$-algebras. Suppose that, for every $n\leq m,$ there exists a $^*-$unital isomorphic embedding $\pi_{n,m}:\mathcal{A}_n\to\mathcal{A}_m.$ If
$$\pi_{n,k}=\pi_{m,k}\circ\pi_{n,m}\quad n\leq m\leq k,$$
then there exists a $C^*$-algebra $A$ and a net of $^*$-isomorphic embeddings $\pi_n:\mathcal{A}_n\to A,$ $n\geq0,$ such that
\begin{equation}\label{eq_pi_n}\pi_n=\pi_m\circ\pi_{n,m}\quad 0\leq n\leq m\end{equation}
 and $\bigcup_{n\geq 1}\pi_n(\mathcal{A}_n)$ is uniformly dense in $A$.
\end{lem}
%\begin{proof} Set
%$$A=\{\{x_n\}_{n\geq0}, \ {\cb x_n\in\mathcal{A}_n}:\ \exists n_0\geq0\mbox{ such that }x_n=\pi_{n_0,n}(x_{n_0})\mbox{ for all }n\geq n_0\}$$
%and
%$$\|\{x_n\}_{n\geq0}\|=\|x_{n_0}\|_{\mathcal{A}_{n_0}}.$$
%The algebra $A$ (equipped with componentwise algebraic operations) becomes a pre-$C^*$-algebra. Its completion $\mathcal{A}$ is a $C^*$-algebra.
%
%For a given $n_0\geq0$ and $x\in\mathcal{A}_{n_0},$ set $\pi_{n_0}(x)=\{x_n\}_{n\geq0},$ where $x_n=\pi_{n_0,n}(x)$ for $n\geq n_0$ and $x_n=0$ for $n<n_0.$ Clearly, $\pi_{n_0}:\mathcal{A}_{n_0}\to \mathcal{A}$ is a $^*$-isomorphic embedding satisfying the required conditions.
%\end{proof}

 The algebra $A$ from Lemma \ref{dirlimdef1} is recall \textit{the inductive limit} of $\{\mathcal{A}_n\}_{n\geq0}.$
In parallel, we construct a inductive  limit of states.

\begin{lem}\cite[Definition 1.23.10]{Sakai}\label{dirlimdef2} Let the $C^*$-algebra $A$ be the inductive limit of $C^*$-algebras $\mathcal{A}_n,$ $n\geq0.$ If states $\sigma_n$ on $\mathcal{A}_n,$ $n\geq0,$ are such that
$$\sigma_n=\sigma_m\circ\pi_{n,m},\quad 0\leq n\leq m,$$
then there exists a state $\sigma$ on $A$ such
that
\begin{equation}\label{eq_sigma_n}\sigma_n=\sigma\circ\pi_n,\quad n\geq0.\end{equation}
If the states $\sigma_n,$ $n\geq0,$ are tracial, then so is the
limit state  $\sigma$.
\end{lem}
%\begin{proof} Let $\{x_n\}_{n\geq0}\in A,$ {\cb where $A$ is defined in the proof of preceding lemma.} That is, $x_n=\pi_{n_0,n}(x_{n_0})$ for all $n\geq n_0.$ Define the functional {\cb $\varsigma$} on $\mathcal{A}_0$ by setting ${\cb \varsigma}(x)=\varsigma_{n_0}(x_{n_0}).$ It is easy to see that ${\cb \varsigma}(x)$ is well defined (that is, it does not depend on the particular choice of $n_0$). Clearly, such a functional is linear and bounded on $A.$ Thus, it extends to a bounded functional on $\mathcal{A}.$
%
%If $x\in\mathcal{A}$ is positive, then there exists
%$y\in\mathcal{A}$ such that $x=y^*y.$ Hence, for a given
%$\varepsilon>0,$ there exists $z\in A$ such that
%$\|x-z^*z\|_{\mathcal{A}}\leq\varepsilon.$ We have
%$z=\{z_n\}_{n\geq0}$ and $z_n=\pi_{n_0,n}(z_{n_0})$ for all $n\geq
%n_0.$ It follows that ${\cb
%\varsigma}(z^*z)=\varsigma_{n_0}(z_{n_0}^*z_{n_0})\geq 0.$ By
%continuity of ${\cb \varsigma},$ we have ${\cb
%\varsigma}(x)\geq0.$ Since ${\cb \varsigma}(\mathbbm 1)=1,$ it
%follows that ${\cb \varsigma}$ is a state on $\mathcal{A}.$ {\cb
%Since} each $\varsigma_n$ is a trace, {\cb it follows that } ${\cb
%\varsigma}$ is a trace on $A.$ By continuity, it extends to a
%trace on $\mathcal{A}.$
%\end{proof}
Recall that a pair $(\mathcal{A},\sigma)$ is called a noncommutative probability space, if $\mathcal{A}$ is a von Neumann algebra equipped with faithful tracial state $\sigma$.
\begin{lem}\label{lem_cyclic}
Let  noncommutative probability spaces $(\mathcal{A}_n,\sigma_n), n\geq 0,$ be as in Lemmas \ref{dirlimdef1} and \ref{dirlimdef2}, let $A$ and $\sigma$ be the inductive limits of $\mathcal{A}_n$ and $\sigma_n$, respectively. Then $\sigma$ induces a tracial state on von Neumann algebra generated by cyclic representation $(\theta,H)$ of $A$ with respect to $\sigma$, such that
\begin{equation}\label{eq_sigma_theta_pi}
\sigma\circ\theta\circ\pi_n=\sigma_n.
\end{equation}
\end{lem}
\begin{proof}
Let $(\theta,H)$ be cyclic representation of $A$ with respect to $\sigma$ and let
$\mathcal{A}:=\overline{\theta(A)}^{wo}$. For every $x\in\mathcal{A}$ we set $\sigma(x):=\langle x\xi,\xi\rangle_H$, where $\xi$ be the cyclic vector of the representation $(\theta,H)$. One can easily prove that in this way we can extend the state $\sigma$ defined on $A$ up to a faithful state on $\mathcal{A}$. By construction of $\sigma$ and \eqref{eq_sigma_n}
we have $\sigma\circ\theta\circ\pi_n=\sigma_n.$
It remains to show that this extension is in fact trace on $\mathcal{A}$.
Let firstly $x\in \mathcal A, y\in A$, then there exists a sequence $\{x_n\}_{n\geq 0}\subset A$ such that $x_n\xrightarrow{wo}x$. Since the state $\sigma$ is tracial on $A$ we have $\sigma(x_ny)=\sigma(yx_n)$. Hence, $wo$-convergence of $x_n$ to $x$ implies that $\sigma(xy)=\sigma(yx)$. Further, for arbitrary $x,y\in\mathcal A$ we have $\sigma(xy_n)=\sigma(y_nx)$, where $y_n\xrightarrow{wo} y$, and therefore,  $\sigma(xy)=\sigma(yx)$ for all $x,y\in\mathcal A$.
\end{proof}

\begin{lem}\label{lem_il_hyperf}
If $(\mathcal{A}_n,\sigma_n), (\mathcal{A},\sigma), n\geq 0,$ are as in Lemma \ref{lem_cyclic} and in addition, every $\mathcal{A}_n$ is hyperfinite, then $\mathcal{A}$ is also hyperfinite.
\end{lem}
\begin{proof}Let $(\theta,H)$ be as in Lemma \ref{lem_cyclic}.
 For a given $n\geq0,$ select finite dimensional subalgebras $\mathcal{A}_{n,m},$ $m\geq0,$ in $\mathcal{A}_n$ such that their union is weak$^*$ dense in $\mathcal{A}_n.$ Define finite dimensional subalgebras $\mathcal{B}_{n,m},$ $n,m\geq0,$ in $\mathcal{A}$ by setting $\mathcal{B}_{n,m}=(\theta\circ\pi_n)(\mathcal{A}_{n,m}).$
We have that $\bigcup_{n}\mathcal{B}_{n,m}$ is weak$^*$ dense in $(\theta\circ\pi_n)(\mathcal{A}_n)$, since trace preserving $*$-homomorphism  $\theta\circ\pi_n$ from $\mathcal{A}_n$ onto $(\theta\circ\pi_n)(\mathcal{A}_n)$ is weak$*$-continuous \cite[Theorem 2.4.23]{Brat_Robin}. Consequently, $\bigcup_{n,m}\mathcal{B}_{n,m}$ is weak$^*$ dense in $\bigcup_{n}(\theta\circ\pi_n)(\mathcal{A}_n)$. By Lemma \ref{dirlimdef1} the union $\bigcup_{n}\pi_n(\mathcal{A}_n)$ is uniformly dense in the algebra $A$, and therefore $\bigcup_{n}(\theta\circ\pi_n)(\mathcal{A}_n)$ is weak$^*$ dense in $\mathcal{A}$. Thus, $\bigcup_{n}\mathcal{B}_{n,m}$ is weak$^*$ dense in $\mathcal{A}$ and, therefore, $\mathcal{A}$ is also hyperfinite.
\end{proof}

\subsection{Ultraproducts of von Neumann algebras and $\mathcal{L}_p-$spaces}
Let $(\mathcal{N}_k,\tau_k),$ $k\geq0,$ be noncommutative probability spaces. Consider the $C^*-$algebra
$$l_{\infty}(\{\mathcal{N}_k\}_{k\geq 0})=\Big\{\{x(k)\}_{k\geq 0}\in\prod_{k\geq 0}\mathcal{N}_k:\ \sup_{k\geq 0}\|x(k)\|_{\infty}<\infty\Big\}.$$
We fix a free ultrafilter $\omega$ on $\mathbb{Z}_+$ and equip $l_{\infty}(\{\mathcal{N}_k\}_{k\geq 0})$ with a tracial state $\Big(\tau_k\Big)_{k\geq0}^{\omega}$ defined by the formula
$$\Big(\tau_k\Big)_{k\geq0}^{\omega}(\{x(k)\}_{k\geq 0})=\lim_{k\to\omega}\{\tau_k(x(k))\}_{k\geq 0}.$$
Consider the ideal $J$ in $l_{\infty}(\{\mathcal{N}_k\}_{k\geq 0})$ defined by the formula
$$J=\Big\{ x\in l_{\infty}(\{\mathcal{N}_k\}_{k\geq 0}):\ \Big(\tau_k\Big)_{k\geq0}^{\omega}(|x|^2)=0\Big\}.$$
The quotient $C^*-$algebra
$$\Big(\mathcal{N}_k\Big)_{k\geq0}^{\omega}\stackrel{def}{=}l_{\infty}(\{\mathcal{N}_k\}_{k\geq 0})/J$$
is, in fact, a von Neumann algebra (see e.g. \cite{Mcduff,Vesterstrom}). The tracial state $\Big(\tau_k\Big)_{k\geq0}^{\omega}$ generates a faithful tracial state on $\Big(\mathcal{N}_k\Big)_{k\geq0}^{\omega}$ (which we also denote by $\Big(\tau_k\Big)_{k\geq0}^{\omega}$). Thus, $(\Big(\mathcal{N}_k\Big)_{k\geq0}^{\omega},\Big(\tau_k\Big)_{k\geq0}^{\omega})$ becomes a noncommutative probability space. For a special case $\mathcal{N}_k=\mathcal{N},$ $\tau_k=\tau,$ it is common to write $\mathcal{N}^{\omega}$ and $\tau^{\omega}.$

To define a Banach space $\Big(\mathcal{L}_p(\mathcal{N}_k,\tau_k)\Big)_{k\geq0}^{\omega},$ we equip the set of all bounded sequences $\{x(k)\}_{k\geq0}\in\prod_{k\geq0}\mathcal{L}_p(\mathcal{N}_k,\tau_k)$ with a seminorm
$$\|\{x(k)\}_{k\geq0}\|_p=\lim\limits_{k\to\omega}\|x(k)\|_p,\quad \{x(k)\}_{k\geq0}\in\prod_{k\geq0}\mathcal{L}_p(\mathcal{N}_k,\tau_k)\mbox{ is bounded}.$$
We define $\Big(\mathcal{L}_p(\mathcal{N}_k,\tau_k)\Big)_{k\geq0}^{\omega}$ to be a quotient space over the linear subspace
$$\Big\{\{x(k)\}_{k\geq0}:\ \lim\limits_{k\to\omega}\|x(k)\|_p=0\Big\}$$
and equip it with a quotient norm.

It is tempting to make an identification
\begin{equation}\label{false}
\Big(\mathcal{L}_p(\mathcal{N}_k,\tau_k)\Big)_{k\geq0}^{\omega}=\mathcal{L}_p(\Big(\mathcal{N}_k\Big)_{k\geq0}^{\omega},\Big(\tau_k\Big)_{k\geq0}^{\omega}).
\end{equation}
Unfortunately, such an identification does not hold. In fact, we have (see Lemma 2.13 in \cite{hrs} and diagram following it) that
\begin{equation}\label{true}
\Big(\mathcal{L}_p(\mathcal{N}_k,\tau_k)\Big)_{k\geq0}^{\omega}\supset\mathcal{L}_p(\Big(\mathcal{N}_k\Big)_{k\geq0}^{\omega},\Big(\tau_k\Big)_{k\geq0}^{\omega}).
\end{equation}
However, the following lemma shows how to use the uniform integrability condition to deduce that a certain element of $\Big(\mathcal{L}_p(\mathcal{N}_k,\tau_k)\Big)_{k\geq0}^{\omega}$ belongs to the $\mathcal{L}_p-$space $\mathcal{L}_p(\Big(\mathcal{N}_k\Big)_{k\geq0}^{\omega},\Big(\tau_k\Big)_{k\geq0}^{\omega}).$

The sequence $\{y(k)\}_{k\geq 0}\in\prod_{k\geq0}\mathcal{L}_1(\mathcal{N}_k,\tau_k)$ is said to be uniformly integrable if, for every $\varepsilon>0,$ there exists $K$ such that $\tau(|y(k)|E_{|y(k)|}(K,\infty))\leq\varepsilon$ for all $k\geq0.$

\begin{lem}\label{equi2} If the sequence $x=\{x(k)\}_{k\geq0}\in\prod_{k\geq0}\mathcal{L}_p(\mathcal{N}_k,\tau_k)$ is such that the sequence $\{|x(k)|^p\}_{k\geq0}\in \prod_{k\geq0}\mathcal{L}_1(\mathcal{N}_k,\tau_k)$ is uniformly integrable, then $$x\in\mathcal{L}_p(\Big(\mathcal{N}_k\Big)_{k\geq0}^{\omega},\Big(\tau_k\Big)_{k\geq0}^{\omega}).$$
\end{lem}
\begin{proof} Without loss of generality, we have $x(k)\geq0,$ $k\geq0.$ For a given $t>0,$ consider a operator $x_t=\{x(k)E_{x(k)}[0,t]\}_{n\geq0}\in \Big(\mathcal{N}_k\Big)_{k\geq0}^{\omega}.$ We have that
$$\|x-x_t\|_{\Big(\mathcal{L}_p(\mathcal{N}_k,\tau_k)\Big)_{k\geq0}^{\omega}}=\lim\limits_{k\to\omega}\|x(k)E_{x(k)}(t,\infty)\|_p\leq(\sup_{k\geq0}\tau(x(k)^pE_{x(k)}(t,\infty)))^{1/p}.$$
Since the sequence $\{x(k)^p\}_{k\geq0},$ is uniformly integrable, it follows that the right hand side tends to $0$ as $t\to\infty.$ Thus, $x$ admits an approximation in $\Big(\mathcal{L}_p(\mathcal{N}_k,\tau_k)\Big)_{k\geq0}^{\omega}$ by bounded operators. By Lemma 2.13 in \cite{hrs}, we have that $x\in\mathcal{L}_p(\Big(\mathcal{N}_k\Big)_{k\geq0}^{\omega},\Big(\tau_k\Big)_{k\geq0}^{\omega}).$
\end{proof}

\section{Noncommutative Kruglov operator ($\tau$-finite case)}\label{krug fin}

The Kruglov operator was originally introduced in the setting of classical (commutative) probability theory in \cite{as_zam,AS}. The origins of this construction go back to Braverman \cite{Braverman} and Kruglov \cite{Kruglov}. Note that in commutative setting the Kruglov operator is originally defined on the space of all measurable functions (see \cite{AS} and \cite{as1}). In order to achieve main objective of this paper we will consider this operator on $L_1$-spaces associated with semifinite von Neumann algebras.

Let $\Omega=\prod_{k=0}^{\infty}(0,1)$ be the (infinite
dimensional) hypercube equipped with the product Lebesgue
measure $dm^{\infty}.$ The Kruglov operator $K$ acts from
$L_1(0,1)$ to $L_1(\Omega)$ by the following formula
\begin{equation}\label{kastsuk}
Kx=\sum_{k=1}^{\infty}\sum_{m=1}^k\chi_{A_k}\otimes \chi_{(0,1)}^{\otimes(m-1)}\otimes x\otimes \chi_{(0,1)}^{\otimes\infty},\quad x\in L_1(0,1),
\end{equation}
where $A_k,$ $k\geq0,$ are pairwise disjoint sets with $m(A_k)=\frac1{e\cdot k!}$ for all $k\geq0$ (so that $\cup_{k\geq0}A_k=(0,1)$) and where $\chi_{B}$ is the indicator function of the measurable set $B\subset (0,1).$

The following assertion is proved in \cite{ASorlicz}

\begin{thm}\label{as krug} The Kruglov operator $K:L_1(0,1)\to L_1(\Omega)$ defined in \eqref{kastsuk} satisfies the following properties.
\begin{enumerate}[{\rm (i)}]
\item $K$ maps positive random variables into positive ones.
\item $K$ maps symmetrically distributed random variables into symmetrically distributed ones.
\item If $0\leq x,y\in L_1(0,1)$ are equimeasurable, then so are $Kx,Ky\in L_1(\Omega).$
\item If $x_k\in L_1(0,1),$ $1\leq k\leq n,$ are pairwise disjointly supported functions, then the functions $Kx_k\in L_1(0,1),$ $1\leq k\leq n,$ are independent.
\end{enumerate}
\end{thm}

The main aim of this section is to construct the noncommutative Kruglov operator in a $\tau$-finite von Neumann algebra. Recall that
  if $x_k\in \mathcal{S}(\mathcal{M},\tau),$ $k\geq0,$ if $e_k\in\mathcal{M},$ $k\geq0,$ are pairwise orthogonal projections, and if $y_k\in e_k\mathcal{S}(\mathcal{M},\tau)e_k$ are such that $\mu(y_k)=\mu(x_k),$ $k\geq0,$ then we write
\begin{equation}\label{def_direct_sum_of_operators}
\bigoplus_{k=0}^\infty x_k:=\sum_{k=0}^\infty y_k.
\end{equation}
If $x_0,x_1,\ldots\in \mathcal S(\mathcal M,\tau)$ then we view $\bigoplus_{k=0}^{\infty}x_k$ as an element  of $\mathcal S\big(\bigoplus_{k=0}^{\infty}\mathcal{M},\bigoplus_{k=0}^{\infty}\tau \big).$

Let $\mathcal{N}$ be a von Neumann algebra equipped with a faithful normal finite trace $\tau.$ For every $x\in\mathcal{S}(\mathcal{N},\tau),$ set
\begin{equation}\label{k fin def}
\mathscr{K}x=0\oplus\bigoplus_{k=1}^{\infty}\left(\sum_{m=1}^k\mathbbm 1^{\otimes(m-1)}\otimes x\otimes \mathbbm 1^{\otimes (k-m)}\right).
\end{equation}
Here, $\mathbbm 1$ denotes the unit element in $\mathcal{N}.$ By definition $\mathscr{K}x$ belongs to $\mathcal{S}\Big(\bigoplus_{k=0}^{\infty}\mathcal{N}^{\otimes k},\sigma\Big)$ for every $x\in\mathcal{S}(\mathcal{N},\tau),$  where $\sigma$ is defined by \eqref{eq_sigma}.
The operator
$$\mathscr{K}:\mathcal{S}(\mathcal{N},\tau)\to \mathcal{S}\Big(\bigoplus_{k=0}^{\infty}\mathcal{N}^{\otimes k},\sigma\Big)$$
defined by \eqref{k fin def} is said to be the \textit{(noncommutative) Kruglov operator}.

The following lemma shows that the values of the Kruglov operator belong in fact to a smaller algebra.

\begin{lem}\label{lem1} For any $x\in\mathcal{S}(\mathcal{N},\tau),$ we have $\mathscr{K}x\in \mathcal{S}(M_s(\mathcal{N}),\sigma),$ where the algebra $M_s(\mathcal{N})$ is defined by \eqref{eq_M_s}.
\end{lem}
\begin{proof}
In order to prove the assertion of the lemma it is enough to show that
$y_k(x)\in \mathcal{S}(\mathcal{N}^{\otimes k}_s,\tau^{\otimes k}),$ for any  $x\in\mathcal{S}(\mathcal{N},\tau),$ where
\begin{equation}\label{eq_y_k}
y_k(x):= \sum_{m=1}^k\mathbbm 1^{\otimes(m-1)}\otimes x\otimes \mathbbm 1^{\otimes (k-m)},\quad k\ge 1.
\end{equation}
Let us first show that
\begin{equation}\label{xisbound2}
y_k(x)\in\mathcal N^{\otimes k}_s\quad\text{for every}\quad x\in\mathcal{N},\quad k\ge 1.
\end{equation}
Indeed, since $x\in \mathcal{N},$ it follows that $y_k(x)\in \mathcal N^{\otimes k}.$
Consider the element $$z_k(x):=k (x\otimes \mathbbm 1^{\otimes (k-1)})\in\mathcal N^{\otimes k}, \ \ \ k\ge 1.$$
Then by the definition of $E_k$ (see \eqref{eq_E_k}), we have
\begin{equation}\label{xisbound1} E_k(z_k(x))=k \cdot \frac1{k!}\sum_{\rho}\tilde x_{\rho(1)}\otimes\cdots\otimes \tilde x_{\rho(k)}, \ \ k\ge 1,
\end{equation}
where $\tilde x_{ \rho(j)}=x,$ if $\rho(j)=1$ and $\tilde x_{ \rho(j)}=\mathbbm 1$ otherwise, $1\le j\le k,$ and $\rho$ runs over the set of all permutations $\mathfrak{S}_k$. Expanding the right hand side of
\eqref{xisbound1}, we obtain
 \begin{multline*} E_k(z_k(x))=k \cdot \frac1{k!}\left((k-1)! ( x\otimes \mathbbm 1^{\otimes (k-1)})+
 (k-1)! ( \mathbbm 1\otimes x\otimes \mathbbm 1^{\otimes (k-2)})+\ldots\right.\\\left.+(k-1)! ( \mathbbm 1^{\otimes (k-1)}\otimes x)\right)=
 \sum_{m=1}^k\mathbbm 1^{\otimes(m-1)}\otimes x\otimes \mathbbm 1^{\otimes (k-m)}=y_k(x), \ \ k\ge 1,
\end{multline*}
that completes the proof of \eqref{xisbound2}.

Now let $x\in\mathcal{S}(\mathcal{N},\tau).$
Fix $k\ge 1$ and consider $x_n=x E_{x}([-n,n])$ and
$y_k(x_n)$,  $n\ge 1.$
Since $x_n\in \mathbb N,$ by \eqref{xisbound2} we have that $y_k(x_n)\in \mathcal N^{\otimes k}_s$ for every $n\ge 1.$
Now observe that the sequence $\{y_k(x_n)\}_{n=1}^\infty$ converges to $y_k(x)$ in the measure topology~$t_{\tau^{\otimes k}}$. Since
$\mathcal S(\mathcal N^{\otimes k}_s,\tau^{\otimes k})$ is complete with respect to the measure topology, we obtain that $y_k(x)\in
\mathcal S(\mathcal N^{\otimes k}_s,\tau^{\otimes k}).$
 \end{proof}

The following proposition collects some of important properties of the operator  $\mathscr{K}.$
Recall that $[x,y]:=xy-yx,$ $x,y\in\mathcal{S}(\mathcal{N},\tau)$.

 \begin{prop} \label{pr_prop_K}
\begin{enumerate}[{\rm (i)}]
        \item\label{prop_K3} $\mathscr{K}x\in\mathcal{L}_1(M_s(\mathcal{N}),\sigma)$ and $\sigma(\mathscr{K}x)=\tau(x)$ for every $x\in\mathcal{L}_1(\mathcal{N},\tau).$
       \item\label{kommut lemma}  For any  $x,y\in\mathcal{S}(\mathcal{N},\tau)$, we have that  $[\mathscr{K}x,\mathscr{K}y]=\mathscr{K}[x,y].$
       \item\label{prop_K5} $\mathscr{K}x \ge 0$ whenever $x\ge 0;$
        \item\label{K is self_adj} $(\mathscr{K}x)^* =\mathscr{K}x$ whenever $x=x^*.$
        \end{enumerate}
\end{prop}
\begin{proof}
To prove \eqref{prop_K3},
  we show that $\sigma(\mathscr{K}x)=\tau(x)$ for every $x\in\mathcal{L}_1(\mathcal{N},\tau).$ Indeed,
  \begin{align*}
  \sigma(\mathscr{K}x)&=\sum_{k=1}^\infty\frac1{\exp{(\tau(\mathbbm 1))}\cdot k!} \ \tau^{\otimes k}\left(\sum_{m=1}^k\mathbbm 1^{\otimes(m-1)}\otimes x\otimes \mathbbm 1^{\otimes (k-m)}\right)\\
  &=\sum_{k=1}^\infty\frac{k\cdot\tau(x)\tau(\mathbbm 1)^{k-1}}{\exp(\tau(\mathbbm 1))\cdot k!}=\frac{\tau(x)}{\exp(\tau(\mathbbm 1))}\sum_{k=1}^\infty\frac{\tau(\mathbbm 1)^{k-1}}{ (k-1)!}=\tau(x).
  \end{align*}

\eqref{kommut lemma}. A straightforward computation yields
\begin{align*}
& \Big(\sum_{m_1=1}^k1^{\otimes(m_1-1)}\otimes x\otimes 1^{\otimes (k-m_1)}\Big)\Big(\sum_{m_2=1}^k1^{\otimes(m_2-1)}\otimes y\otimes 1^{\otimes (k-m_2)}\Big)\\
& =\sum_{m_1=1}^{k-1}\sum_{m_2=m_1+1}^k1^{\otimes(m_1-1)}\otimes x\otimes 1^{\otimes (m_2-m_1-1)}\otimes y\otimes 1^{\otimes (k-m_2)} \\
&+\sum_{m_2=1}^{k-1}\sum_{m_1=m_2+1}^k1^{\otimes(m_2-1)}\otimes y\otimes 1^{\otimes (m_1-m_2-1)}\otimes x\otimes 1^{\otimes (k-m_1)}\\
&+\sum_{m=1}^k1^{\otimes(m-1)}\otimes xy\otimes 1^{\otimes (k-m)} \ \ \end{align*}
for every $k\ge 1$. It follows that, for every $k\geq 1,$
\begin{align*}
 &\Big(\sum_{m=1}^k1^{\otimes(m-1)}\otimes x\otimes 1^{\otimes
(k-m)}\Big)\Big(\sum_{m=1}^k1^{\otimes(m-1)}\otimes y\otimes
1^{\otimes (k-m)}\Big) \\
 & -\Big(\sum_{m=1}^k1^{\otimes(m-1)}\otimes y\otimes 1^{\otimes
(k-m)}\Big)\Big(\sum_{m=1}^k1^{\otimes(m-1)}\otimes  x\otimes
1^{\otimes (k-m)}\Big)\\
& =\sum_{m=1}^k1^{\otimes(m-1)}\otimes
xy\otimes 1^{\otimes (k-m)}-\sum_{m=1}^k1^{\otimes(m-1)}\otimes
yx\otimes 1^{\otimes (k-m)}\\
&=\sum_{m=1}^k1^{\otimes(m-1)}\otimes
[x,y]\otimes 1^{\otimes (k-m)}.\end{align*} 
Thus, we have
$$[\mathscr{K}x,\mathscr{K}y]=0\oplus\bigoplus_{k=1}^{\infty}\left(\Big(\sum_{m=1}^k1^{\otimes(m-1)}\otimes x\otimes 1^{\otimes (k-m)}\Big)\Big(\sum_{m=1}^k1^{\otimes(m-1)}\otimes y\otimes 1^{\otimes (k-m)}\Big)\right.$$
$$\left.-\Big(\sum_{m=1}^k1^{\otimes(m-1)}\otimes y\otimes 1^{\otimes (k-m)}\Big)\Big(\sum_{m=1}^k1^{\otimes(m-1)}\otimes x\otimes 1^{\otimes
(k-m)}\Big)\right)=\mathscr{K}[x,y].$$

Property \eqref{prop_K5} follows from the definition of the operator $\mathscr{K}$. Property \eqref{K is self_adj} is a straightforward subsequence of \eqref{prop_K5}.
\end{proof}

The following result shows the connection between noncommutative and commutative Kruglov operators.

\begin{lem}\label{lem2} Let $\tau(\mathbbm 1)=1.$ Then for every positive $x\in \mathcal{L}_1(\mathcal{N},\tau),$  we have $\mu(K\mu(x))=\mu(\mathscr{K}x).$
\end{lem}
\begin{proof} Fix $k\in \mathbb N$. Observe that the operators $y_k(x)$ defined by \eqref{eq_y_k} with respect to the algebra $(\mathcal{N}^{\otimes k}, \tau^{\otimes k})$  and
$$\sum_{m=1}^k \chi_{(0,1)}^{\otimes(m-1)}\otimes \mu(x)\otimes \chi_{(0,1)}^{\otimes (k-m)},$$ with respect to $(\Omega, dm^\infty)$ are equimeasurable (see~e.g.~\cite{C-S}). Therefore, the operator $y_k(x)$
 with respect to the algebra $(\mathcal{N}^{\otimes k}, \frac{1}{e^{\tau(\mathbbm 1)}\cdot k!}\cdot\tau^{\otimes k})$ is equimeasurable with the operator
$$\chi_{A_k}\otimes\sum_{m=1}^k \chi_{(0,1)}^{\otimes(m-1)}\otimes \mu(x)\otimes \chi_{(0,1)}^{\otimes (k-m)},$$ with respect to $(\Omega, dm^\infty)$,
where $A_k$ is the same as in \eqref{kastsuk}.
Thus, the operators $K\mu(x)$ and $\mathscr{K}x$ are equimeasurable as a direct sum of equimeasurable operators.
\end{proof}

It is not difficult to see that there is a much stronger connection than indicated by Lemma \ref{lem2} between the commutative and noncommutative Kruglov operators. The following theorem states this connection in full generality. Since we do not use this theorem in this paper, we omit the proof.

 \begin{thm}
 Suppose that $\mathcal N=L_\infty(0,1).$ Then there exists trace preserving unital $*$-homomorphism
$\gamma:M_s(\mathcal{N})\to L_\infty(\Omega)$ such that its $L_1$-extension $\tilde\gamma$ satisfies
$\tilde\gamma(\mathscr{K}(x))=Kx$ for any $x\in L_1(0,1).$
 \end{thm}

The following important lemma delivers the expression for the characteristic function of the random variable $\mathscr{K}x$  in terms of that of the random variable $x$
(compare with the Kruglov property studied in \cite{Braverman,AS_UMN}).

\begin{lem}\label{k finite} Let $\mathcal{N}$ be a finite von Neumann algebra equipped with a finite trace $\tau.$ The Kruglov operator $\mathscr{K}:\mathcal{L}_1(\mathcal{N},\tau)\to\mathcal{L}_1(M_s(\mathcal{N}),\sigma)$ satisfies the condition
$$\sigma(\exp(i\mathscr{K}x))=\exp(\tau(\exp(ix)-\mathbbm 1)),\quad x=x^*\in\mathcal{L}_1(\mathcal{N},\tau),$$
where $\sigma$ is defined by \eqref{eq_sigma} and $\mathbbm 1$ is the unit element of the algebra $\mathcal{N}.$
\end{lem}
\begin{proof}
Let $x=x^*\in\mathcal{L}_1(\mathcal{N},\tau)$ and let $p_k$ ($k\ge 0$) be a projection of the algebra $M_s(\mathcal{N})$ defined in \eqref{eq_p_k}.
By definition of $\sigma$ we have that
\begin{equation}\label{k_finite1}
\sigma(\exp(i\mathscr{K}x))=\exp(-\tau(\mathbbm 1))\sum_{k=0}^{\infty}\frac1{k!}\tau^{\otimes k}(p_k \exp(i\mathscr{K}x)).
\end{equation}
Let $y\in\mathcal{S}(M_s(\mathcal{N}),\sigma).$ Observing that $p_ky^m=(p_ky)^m$ for all $m\geq0$ and $k\ge 0,$ we infer that $p_k(\exp(iy))=\exp(i p_k y)$. It follows that
\begin{equation}\label{k_finite2}
p_k \exp(i\mathscr{K}x)=\exp(ip_k\mathscr{K}x),\quad k\ge 0.
\end{equation}
Next we will show that
\begin{equation}\label{k_finite3}
\tau^{\otimes k}(\exp(ip_k\mathscr{K}x))=\tau^k(\exp(ix)),\quad k\geq0.
\end{equation}
 Indeed,  for every $x=x^*\in\mathcal{L}_1(\mathcal{N},\tau)$ and $k\ge 0$ we have
\begin{equation}\label{k_finite4}
\exp\left(i\sum_{m=1}^k\mathbbm{1}^{\otimes(m-1)}\otimes x\otimes \mathbbm{1}^{\otimes (k-m)}\right)=
\prod_{m=1}^k \exp\left(\mathbbm{1}^{\otimes (m-1)}\otimes ix\otimes \mathbbm{1}^{\otimes (k-m)}\right).
\end{equation}
Considering each term of the latter product separately for every $k\ge 0$ and $1\le m\le k,$ we obtain
\begin{align}\label{k_finite5}
 &\exp(\mathbbm{1}^{\otimes (m-1)}\otimes ix\otimes \mathbbm{1}^{\otimes (k-m)})\\
 &=
   \sum_{j=0}^\infty\frac{1}{j!}\Big(\mathbbm{1}^{\otimes (m-1)}\otimes ix\otimes \mathbbm{1}^{\otimes (k-m)}\Big)^j=
   \mathbbm{1}^{\otimes (m-1)}\otimes\left(\sum_{j=0}^\infty\frac{(ix)^j}{j!}\right) \otimes \mathbbm{1}^{\otimes (k-m)} \nonumber \\
 &=
   \mathbbm{1}^{\otimes (m-1)}\otimes \underbrace{\exp(ix)}_m \otimes \mathbbm{1}^{\otimes (k-m)}. \nonumber
\end{align}
Combining \eqref{k_finite4} and \eqref{k_finite5} we arrive at
$$\exp\left(i\sum_{m=1}^k\mathbbm 1^{\otimes(m-1)}\otimes x\otimes \mathbbm 1^{\otimes (k-m)}\right)=\exp(ix)^{\otimes k}.$$
Hence, for $k\geq0$ we obtain
 \begin{equation*}\tau^{\otimes k}(\exp(ip_k\mathscr{K}x))=\tau^{\otimes k}\left(\exp\left(i\sum_{m=1}^k\mathbbm 1^{\otimes(m-1)}\otimes x\otimes \mathbbm 1^{\otimes (k-m)}\right)\right)=\tau^k(\exp(ix)), \end{equation*}
which completes the proof of \eqref{k_finite3}.
Applying subsequently \eqref{k_finite1}, \eqref{k_finite2} and \eqref{k_finite3} we conclude
\begin{align*}&\sigma(\exp(i\mathscr{K}x))=\exp(-\tau(\mathbbm 1))\sum_{k=0}^{\infty}\frac1{k!}\tau^{\otimes k}(p_k \exp(i\mathscr{K}x))\\
  &=\exp(-\tau(\mathbbm 1))\sum_{k=0}^{\infty}\frac1{k!}\tau^{\otimes k}(\exp(ip_k\mathscr{K}x))
=\exp(-\tau(\mathbbm 1))\sum_{k=0}^{\infty}\frac{\tau^k(\exp(ix))}{k!} \\
 &=\exp(\tau(\exp(ix)-\mathbbm 1)).\qedhere \end{align*}
 \end{proof}

The following result is the noncommutative extension of Theorem \ref{as krug}.

\begin{thm}\label{fincor} The operator $\mathscr{K}:\mathcal{L}_1(\mathcal{N},\tau)\to\mathcal{L}_1(M_s(\mathcal{N}),\sigma)$ satisfies the following properties.
\begin{enumerate}[{\rm (i)}]
\item\label{fincor1} $\mathscr{K}$ maps positive random variables into positive ones.
\item\label{fincor2} $\mathscr{K}$ maps symmetrically distributed random variables into symmetrically distributed ones.
\item\label{kkd} If random variables $x,y\in\mathcal{L}_1(\mathcal{N},\tau)$ are equimeasurable, then so are $\mathscr{K}x,\mathscr{K}y\in \mathcal{L}_1(M_s(\mathcal{N}),\sigma).$
\item\label{kke} If random variables $x_k\in \mathcal{L}_1(\mathcal{N},\tau),$ $1\leq k\leq n,$ are such that $x_jx_k=0$ for $j\neq k,$ then random variables $\mathscr{K}x_k\in \mathcal{L}_1(M_s(\mathcal{N}),\sigma),$ $1\leq k\leq n,$ are independent (in the sense of definition given in Subsection \ref{indep subsect}).
\end{enumerate}
\end{thm}
\begin{proof} The property \eqref{fincor1} is proved in Proposition \ref{pr_prop_K} \eqref{prop_K5}.
The property \eqref{fincor2} is a straightforward consequence of \eqref{kkd} and \eqref{kke}. The property \eqref{kkd} is also straightforward. Indeed, since
$$\sum_{m=1}^k1^{\otimes(m-1)}\otimes x\otimes 1^{\otimes (k-m)} \ \ \mbox{and} \ \ \sum_{m=1}^k1^{\otimes(m-1)}\otimes y\otimes 1^{\otimes (k-m)},$$ are equimeasurable,
 whenever  $x,y\in \mathcal S(\mathcal{N},\tau)$ are, the equimeasurability of $\mathscr{K}x$ and $\mathscr{K}y$ follows from \eqref{k fin def}.

Next, we concentrate on \eqref{kke}. Throughout the proof we fix $x_k=x_k^*\in\mathcal{L}_1(\mathcal{N},\tau),$ $1\leq k\leq n,$ such that $x_jx_k=0$ for $j\neq k.$ First we show that
\begin{equation}\label{xdg1}
\exp\Big(i\sum_{k=1}^n\lambda_kx_k\Big)-\mathbbm 1=\sum_{k=1}^n\Big(\exp(i\lambda_kx_k)-\mathbbm 1\Big),\quad \lambda_k\in\mathbb{R},
\end{equation}
where $\mathbbm 1$ is a unit element of the algebra $\mathcal{N}.$
Indeed, the condition on $x_k\in\mathcal{L}_1(\mathcal{N},\tau),$ $1\leq k\leq n,$ implies that there exists a commutative von Neumann subalgebra $\mathcal{N}_0$ such that $x_k\in\mathcal{L}_1(\mathcal{N}_0,\tau),$ $1\leq k\leq n.$ By spectral theorem, $x_k\in\mathcal{L}_1(\mathcal{N}_0,\tau),$ $1\leq k\leq n,$ can be realised as functions (with disjoint support) on some measure space. For these functions equality \eqref{xdg1} is verified pointwise.

Therefore, applying $\tau$ and $\exp$ to both parts in \eqref{xdg1}, we obtain
\begin{equation}\label{etaue}\exp\Big(\tau\Big(\exp\Big(i\sum_{k=1}^n\lambda_kx_k\Big)-\mathbbm 1\Big)\Big)=\prod_{k=1}^n\exp(\tau(\exp(i\lambda_kx_k)-\mathbbm 1)).\end{equation}
Applying Lemma \ref{k finite} to both expressions below, we obtain
$$\exp\Big(\tau\Big(\exp\Big(i\sum_{k=1}^n\lambda_kx_k\Big)-\mathbbm 1\Big)\Big)=\sigma\Big(\exp\Big(i\sum_{k=1}^n\lambda_k\mathscr{K}x_k\Big)\Big)$$
and
$$\prod_{k=1}^n\exp(\tau(\exp(i\lambda_kx_k)-\mathbbm 1))=\prod_{k=1}^n\sigma(\exp(i\lambda_k\mathscr{K}x_k)).$$
Using \eqref{etaue}, we obtain
\begin{equation}\label{eq_fincor1}\sigma\Big(\exp\Big(i\sum_{k=1}^n\lambda_k\mathscr{K}x_k\Big)\Big)=\prod_{k=1}^n\sigma(\exp(i\lambda_k\mathscr{K}x_k)).\end{equation}
By Lemma \ref{pr_prop_K} \eqref{kommut lemma}, the operators $\mathscr{K}x_k,$ $1\leq k\leq n,$ commute. Hence,
\begin{equation}\label{eq_fincor2}\exp\Big(i\sum_{k=1}^n\lambda_k\mathscr{K}x_k\Big)=\prod_{k=1}^n\exp(i\lambda_k\mathscr{K}x_k).\end{equation}
Combining  \eqref{eq_fincor1} and \eqref{eq_fincor2}, we arrive at
\begin{equation}\label{xdg2}
\sigma\Big(\prod_{k=1}^n\exp(i\lambda_k\mathscr{K}x_k)\Big)=\prod_{k=1}^n\sigma(\exp(i\lambda_k\mathscr{K}x_k)).
\end{equation}
Let $f_k,$ $1\leq k\leq n,$ be Schwartz functions. Denoting by $\mathcal{F}$ and $\mathcal{F}^{-1}$ the Fourier transform and the inverse Fourier transform respectively, we have
\begin{equation}\label{eq_Fourier}
f_k(\mathscr{K}x_k)=(\mathcal{F}^{-1}\mathcal{F} f_k)(\mathscr{K}x_k)=\int_{\mathbb{R}}(\mathcal{F}f_k)(\lambda_k)\ \exp(i\lambda_k\mathscr{K}x_k)d\lambda_k,
\end{equation}
where the latter integral is understood as $L_1$-limit of the series
$$2^{-n}\sum_{m=-2^{2n}}^{2^{2n}}(\mathcal{F}f_k)\Big(\frac{m}{2^n}\Big)\exp\Big(i\frac{m}{2^n}\mathscr{K}x_k\Big),\quad\mbox{when}\quad n\to\infty.$$
The fact that this series converges in $\mathcal{L}_1(H)$ can be found in \cite[Theorem 6.2]{H-S}. Therefore, using \eqref{xdg2}, we obtain
\begin{gather}\begin{split}\label{fincor3}
\sigma\Big(\prod_{k=1}^nf_k(\mathscr{K}x_k)\Big)&=\sigma\Big(\prod_{k=1}^n
\int_{\mathbb{R}}(\mathcal{F}f_k)(\lambda_k)\exp(i\lambda_k\mathscr{K}x_k)d\lambda_k\Big)\\&=
\sigma\Big(
\int_{\mathbb{R}^n}\prod_{k=1}^n\Big((\mathcal{F}f_k)(\lambda_k)\exp(i\lambda_k\mathscr{K}x_k)\Big)d\lambda_1\ldots d\lambda_n\Big)\\&=
\int_{\mathbb{R}^n}\prod_{k=1}^n(\mathcal{F}f_k)(\lambda_k)\cdot \sigma\Big(\prod_{k=1}^n\exp(i\lambda_k\mathscr{K}x_k)\Big)d\lambda_1\ldots d\lambda_n\\& \stackrel{\eqref{xdg2}}{=}
\int_{\mathbb{R}^n}\prod_{k=1}^n(\mathcal{F}f_k)(\lambda_k)\cdot \prod_{k=1}^n\sigma(\exp(i\lambda_k\mathscr{K}x_k)) \ d\lambda_1\ldots d\lambda_n\\&=\prod_{k=1}^n\Big(\int_{\mathbb{R}}(\mathcal{F}f_k)(\lambda_k)\cdot \sigma(\exp(i\lambda_k\mathscr{K}x_k)) \ d\lambda_k\Big)\\&=\prod_{k=1}^n\sigma\Big(\int_{\mathbb{R}}(\mathcal{F}f_k)(\lambda_k) \ \exp(i\lambda_k\mathscr{K}x_k)\ d\lambda_k\Big)\\&=\prod_{k=1}^n\sigma(f_k(\mathscr{K}x_k)).
\end{split}\end{gather}
Let $\mathcal{A}_k$ be the von Neumann subalgebra in $\mathcal{M}$ generated by the spectral family of the element $\mathscr{K}x_k,$ $1\le k\le n.$. Since Schwartz functions are  weak$^*$ dense in $L_\infty(\mathbb R),$ it follows that the set
$$\{f(\mathscr{K}x_k),\ f\mbox{ is Schwartz function}\}$$
is weak$^*$ dense in $\mathcal{A}_k.$ Hence, by \eqref{fincor3}, we arrive at
$$\sigma\Big(\prod_{k=1}^nu_k\Big)=\prod_{k=1}^n\sigma(u_k),\quad u_k\in\mathcal{A}_k,\quad 1\le k\le n.$$
By definition given in Subsection \ref{indep subsect}, elements $\mathscr{K}x_k,$ $1\leq k\leq n,$ are independent.
\end{proof}

\section{Noncommutative Kruglov operator ($\tau-$infinite case)}\label{junge section}
The main objective of this section is to adapt definition \eqref{k fin def} given in the preceding section for $\tau$-finite von Neumann algebras to the setting when $\mathcal{N}$ is a semifinite von Neumann algebra. This is a much harder task and its accomplishment is based on the approach from \cite{junge}.

Fix a semifinite (nonfinite) von Neumann algebra $(\mathcal{N},\tau)$ with a faithful normal semifinite trace $\tau$.
For a fixed $\tau$-finite projection $e$ from $\mathcal{N}$ and $n\in\mathbb{N}$ we consider the
mapping $\alpha_{n,k}:e\mathcal{N}e\to \mathcal{N}^{\otimes n}, 0\leq k\leq n,$ defined by  setting
\begin{equation}\label{alpha_n,k}
\alpha_{n,k}(x)=\binom{n}{k}E_n(x^{\otimes k}\otimes(\mathbbm{1}-e)^{\otimes(n-k)}),\quad x\in e\mathcal{N}e,
\end{equation}
where $E_n$ is the conditional expectation from $\mathcal{N}^{\otimes n}$ onto $\mathcal{N}^{\otimes n}_s$ defined by \eqref{eq_E_k}.

\begin{lem}\label{first homomorphism lemma} For every $e\in\mathcal{P}_{fin}(\mathcal{N})$ and $k,n\in\mathbb{N}, 0\leq k\leq n$, the mapping $\alpha_{n,k}:e\mathcal{N}e\to\mathcal{N}^{\otimes n}$ defined by \eqref{alpha_n,k}
preserves multiplication.
\end{lem}
\begin{proof}We denote by $2^n$  the set of all subsets of $\{1,\cdots,n\}.$
 Since for every $x\in e\mathcal{N}e$ the mapping $E_n(x^{\otimes k}\otimes (\mathbbm{1}-e)^{\otimes(n-k)})$ contains $k!(n-k)!$ equal summands (see \eqref{eq_E_k}), by \eqref{alpha_n,k} we infer that
\begin{equation}\label{E_n_hatx}
\alpha_{n,k}(x)=
\binom{n}{k}E_n(x^{\otimes k}\otimes (\mathbbm{1}-e)^{\otimes(n-k)})=\sum_{\substack{A\in 2^n\\|A|=k}}\mathcal Z_A^{x,\mathbbm{1}-e},
%\Big(\bigotimes_{l\in A}x\bigotimes_{l\notin A}(1-e)\Big).
\end{equation}
where
\begin{equation}\label{x_A}
\mathcal Z_A^{x,y}=z_1\otimes\dots\otimes z_n\in\mathcal{N}^{\otimes n},\quad
z_i=\left\{\begin{aligned}x, &\quad i\in A\\
y,& \quad \mathrm{otherwise}\end{aligned}\right.\quad 1\leq i\leq n.
\end{equation}
Therefore,
$$\alpha_{n,k}(x)\alpha_{n,k}(y)=\sum_{\substack{A_1,A_2\in 2^n\\|A_1|=|A_2|=k}}
\mathcal Z_{A_1}^{x,\mathbbm{1}-e}\mathcal Z_{A_2}^{y,\mathbbm{1}-e}$$
If $A_1\neq A_2,$ then, , we may assume that there exists $l\in A_1$ such that $l\notin A_2$ and, since $x(\mathbbm{1}-e)=0,$ it follows that the corresponding summand is $0.$ Hence,
$$\alpha_{n,k}(x)\alpha_{n,k}(y)=\sum_{\substack{A\in 2^n\\|A|=k}}
\mathcal Z_{A}^{x,\mathbbm{1}-e}\mathcal Z_{A}^{y,\mathbbm{1}-e}=
\sum_{\substack{A\in 2^n\\|A|=k}}\mathcal Z_{A}^{xy,\mathbbm{1}-e}=
\alpha_{n,k}(xy).$$
\end{proof}

Now, we construct a $*-$homomorphism from $(e\mathcal{N}e)^{\otimes k}_s$ into $\mathcal{N}^{\otimes n}_s,$ $e\in \mathcal{P}_{fin}(\mathcal{N}),$ $n,k\in\mathbb{N}$. We define the linear mapping $\pi_{n,k}:(e\mathcal{N}e)^{\otimes k}\to \mathcal{N}_s^{\otimes n}, 0\leq k\leq n,$ by the following rule:
\begin{equation}\label{pi_n,k}
\pi_{n,k}^{e,\mathbbm{1}}:z\to\binom{n}{k}E_n(z\otimes(\mathbbm{1}-e)^{\otimes(n-k)}),\quad z\in(e\mathcal{N}e)^{\otimes k}.
\end{equation}
In what follows, we frequently suppress the dependence of $\pi_{n,k}^{e,\mathbbm{1}}$ on $e$ and $\mathbbm{1}$ and simply write $\pi_{n,k}$.

It is clear that for $\alpha_{n,k}$, defined by \eqref{alpha_n,k}, the equality
\begin{equation}\label{alpha_pi}
\alpha_{n,k}(x)=\pi_{n,k}(x^{\otimes k}), \, x\in e\mathcal{N}e
\end{equation}
holds.

We need the following technical lemma, whose proof consists of tedious algebraic computations and therefore is omitted.
\begin{lem}\label{lem_horror}Let $n\geq k$. We have
\begin{enumerate}[{\rm (i)}]
\item  $\pi_{n,k}(x)=\pi_{n,k}\circ E_k(x), \, x\in (e\mathcal{N}e)^{\otimes k}$;
\item  For $u=u_1\otimes\cdots\otimes u_k\in(e\mathcal{N}e)^{\otimes k}$ we have $$E_k(u)=\frac1{k!}\frac{\partial^k}{\partial^k\lambda} \Big(\sum_{i=1}^k\lambda_iu_i\Big)^{\otimes k},\lambda_i\in\mathbb{R},$$
where we use the notation $\partial^k\lambda$ as a shorthand for $\partial\lambda_1\cdots\partial\lambda_k$;
\item  $\pi_{n,k}\left(\frac{\partial^k}{\partial^k\lambda}(\sum_{i=1}^k\lambda_iu_i)^{\otimes k}\right)=\frac{\partial^k}{\partial^k\lambda}\pi_{n,k}((\sum_{i=1}^k\lambda_iu_i)^{\otimes k}).$
\end{enumerate}
\end{lem}

\begin{lem}\label{second homomorphism lemma}
Let  $e\in\mathcal{P}_{fin}(\mathcal{N}), \,n,k\in\mathbb{N}, 0\leq k\leq n$ and let $\pi_{n,k}$ be defined by \eqref{pi_n,k}. Then, the restriction of $\pi_{n,k}$ onto the algebra $(e\mathcal{N}e)^{\otimes k}_s$ is a $*$-homomorphism from $(e\mathcal{N}e)^{\otimes k}_s$ into $\mathcal{N}^{\otimes n}_s$.
\end{lem}
\begin{proof} For the proof of this lemma we firstly show that for all $u,v\in(e\mathcal{N}e)^{\otimes k}$ the equality
\begin{equation}\label{pi_E_kE_k}
\pi_{n,k}(u)\pi_{n,k}(v)=\pi_{n,k}(E_k(u)E_k(v))
\end{equation} holds.  Fix $u=u_1\otimes\cdots\otimes u_k$ and $v=v_1\otimes\cdots\otimes v_k$. By Lemma \ref{lem_horror} (ii) we have
\begin{equation}\label{second linear}
E_k(u)E_k(v)=\frac1{k!^2}\frac{\partial^{2k}}{\partial^k\lambda\partial^k\mu}\left(\Big(\sum_{i=1}^k\lambda_iu_i\Big)\Big(\sum_{j=1}^k\mu_jv_j\Big)\right)^{\otimes k}.\end{equation}

Therefore, again applying Lemma \ref{lem_horror}, we obtain
\begin{gather*}
\begin{split}
\pi_{n,k}(u)\pi_{n,k}(v)&\stackrel{(i)}{=}\pi_{n,k}(E_k(u))\pi_{n,k}(E_k(v))\\
&=\frac1{k!}\pi_{n,k}\Big(\frac{\partial^k}{\partial^k\lambda}\Big(\sum_{i=1}^k\lambda_iu_i\Big)^{\otimes k}\Big)\cdot\frac1{k!}\pi_{n,k}\Big(\frac{\partial^k}{\partial^k\mu}\Big(\sum_{j=1}^k\mu_ju_j\Big)^{\otimes k}\Big)\\
&\stackrel{(iii)}{=}\frac1{k!^2}\frac{\partial^{2k}}{\partial^k\lambda\partial^k\mu}\pi_{n,k}\Big(\Big(\sum_{i=1}^k\lambda_iu_i\Big)^{\otimes k}\Big)\pi_{n,k}\Big(\Big(\sum_{j=1}^k\mu_ju_j\Big)^{\otimes k}\Big)\\
&\stackrel{\eqref{alpha_pi}}{=}\frac1{k!^2}\frac{\partial^{2k}}{\partial^k\lambda\partial^k\mu}\alpha_{n,k}\Big(\sum_{i=1}^k\lambda_iu_i\Big)\alpha_{n,k}\Big(\sum_{j=1}^k\mu_ju_j\Big).
\end{split}
\end{gather*}
Since by Lemma \ref{first homomorphism lemma} the mapping $\alpha_{n,k}$ preserves multiplication, it follows that
\begin{equation*}
\begin{split}
\pi_{n,k}(u)\pi_{n,k}(v)&=\frac1{k!^2}\frac{\partial^{2k}}{\partial^k\lambda\partial^k\mu}\alpha_{n,k}\Big(\Big(\sum_{i=1}^k\lambda_iu_i\Big)\Big(\sum_{j=1}^k\mu_ju_j\Big)\Big)
\\
&\stackrel{\eqref{alpha_pi}}{=}\frac1{k!^2}\frac{\partial^{2k}}{\partial^k\lambda\partial^k\mu}\pi_{n,k}\Big(\Big(\Big(\sum_{i=1}^k\lambda_iu_i\Big)\Big(\sum_{j=1}^k\mu_ju_j\Big)\Big)^{\otimes k}\Big)\\
&=\pi_{n,k}\Big(\frac1{k!^2}\frac{\partial^{2k}}{\partial^k\lambda\partial^k\mu}\Big(\Big(\sum_{i=1}^k\lambda_iu_i\Big)\Big(\sum_{j=1}^k\mu_ju_j\Big)\Big)^{\otimes k}\Big)\\
&\stackrel{\eqref{second linear}}{=}\pi_{n,k}(E_k(u)E_k(v)).
\end{split}
\end{equation*}

Since $E_k$ is a conditional expectation from $(e\mathcal{N}e)^{\otimes k}$ into $(e\mathcal{N}e)_s^{\otimes k}$, for every $u,v\in(e\mathcal{N}e)_s^{\otimes k}$ we have $E_k(u)=u,E_k(v)=v$, and therefore
 by \eqref{pi_E_kE_k} we obtain
$$\pi_{n,k}(u)\pi_{n,k}(v)=\pi_{n,k}(E_k(u)E_k(v))=\pi_{n,k}(uv).$$
Since $E_n(x^*)=(E_n(x))^*, x\in\mathcal{N}^{\otimes n}$, one can easily obtain that $\pi_{n,k}(u^*)=\pi_{n,k}(u)^*$ for all $u\in(e\mathcal{N}e)_s^{\otimes k}$. Thus, $\pi_{n,k}$ is a $*$-homomorphism from $(e\mathcal{N}e)_s^{\otimes k}$ into $\mathcal{N}_s^{\otimes n}$.

\end{proof}

The following auxiliary lemma is used in Proposition \ref{k compat} below.

\begin{lem}\label{strange equality} Let $1\leq k\leq n.$ For every $x\in\mathcal{N},$ we have
$$\sum_{k=1}^nk\binom{n}{k}E_n(x\otimes e^{\otimes(k-1)}\otimes(\mathbbm{1}-e)^{\otimes(n-k)})=\sum_{k=1}^n\mathbbm{1}^{\otimes(k-1)}\otimes x\otimes \mathbbm{1}^{\otimes(n-k)}\in\mathcal{N}^{\otimes n}_s.$$
\end{lem}
\begin{proof} It is sufficient to consider only the terms on the LHS having $x$ in the first position, since all other terms can be treated similarly. Therefore, by definition of $E_n$ (see \eqref{eq_E_k}),  we have to prove the equality
$$x\otimes\Big(\sum_{k=1}^nk\binom{n}{k}\frac1nE_{n-1}(e^{\otimes(k-1)}\otimes(\mathbbm{1}-e)^{\otimes(n-k)})\Big)=x\otimes \mathbbm{1}^n.$$
By \eqref{E_n_hatx}, we have that
$$\binom{n-1}{k-1}E_{n-1}(e^{\otimes(k-1)}\otimes(\mathbbm{1}-e)^{\otimes(n-k)})=\sum_{\substack{A\in 2^{n-1}\\|A|=k-1}}e_A,$$
where $e_A$ is defined by \eqref{x_A}.
Hence,
$$LHS=x\otimes\Big(\sum_{k=1}^nk\binom{n}{k}\frac1{n\binom{n-1}{k-1}}\sum_{\substack{A\in 2^{n-1}\\|A|=k-1}}e_A\Big)=x\otimes\Big(\sum_{A\in 2^{n-1}}e_A\Big)=x\otimes \mathbbm{1}^{\otimes(n-1)}.$$

The second assertion follows from the proof of Lemma \ref{lem1}.
\end{proof}

 Let $(\mathcal{N}, \tau)$ be a semifinite (nonfinite) von Neumann algebra with faithful normal semifinte trace $\tau$, let $e\leq f$ be arbitrary $\tau$-finite projections from $\mathcal{N}$ and let $M_s(e\mathcal{N}e), M_s(f\mathcal{N}f)$ be symmetric quantized algebras with corresponding finite traces $\sigma_e,\sigma_f$ given by \eqref{eq_sigma_e}. Denote by $\mathscr{K}_e$ (respectively, $\mathscr{K}_f$) the corresponding Kruglov operators on $e\mathcal{N}e$ (respectively, on $f\mathcal{N}f$) constructed in Section \ref{krug fin} (see \eqref{k fin def}).

The following proposition is the key result used in the construction of the Kruglov operator $\mathcal{K}$ on $\mathcal{L}_1(\mathcal{N}).$ It allows us to pass from Kruglov  operator defined on a smaller algebra, to the Kruglov operator on a larger algebra.

\begin{prop}\label{k compat} Let $(\mathcal{N},\tau)$ be a semifinite von Neumann algebra. For all $\tau$-finite projections $e\leq f\in\mathcal{N},$ there exists a unital  $*-$homomorphism $$\pi_{e,f}:S(M_s(e\mathcal{N}e), \sigma_e)\to S(M_s(f\mathcal{N}f),\sigma_f)$$ such that
\begin{enumerate}[{\rm (i)}]
\item for all $ z\in  M_s(e\mathcal{N}e)$, we have
\begin{equation}\label{eq_sigma_pi}
\sigma_e(z)=\sigma_\mathbbm{1}\circ\pi_{e,\mathbbm{1}}(z),
\end{equation}
in particular $\pi_{e,f}:\mathcal{L}_1(M_s(e\mathcal{N}e), \sigma_e)\to\mathcal{L}_1(M_s(f\mathcal{N}f),\sigma_f)$;
\item for all $x\in \mathcal{L}_1(e\mathcal{N}e,\tau)$ the equality
\begin{equation}\label{pi_K}
\pi_{e,f}(\mathscr{K}_e(x))=\mathscr{K}_f(x)
\end{equation} holds.
\end{enumerate}
\end{prop}
\begin{proof} Without loss of generality, we can assume that $f=\mathbbm{1}.$ Let $p_k$ be the natural projection from $M_s(e\mathcal{N}e)=\bigoplus_{k\geq 0}(e\mathcal{N}e)^{\otimes k}_s$ into $(e\mathcal{N}e)_s^{\otimes k}.$ Define the mapping $\pi_{e,\mathbbm{1}}:M_s(e\mathcal{N}e)\to M_s(\mathcal{N})$ by setting
\begin{equation}\label{pi def}
\pi_{e,\mathbbm{1}}=\bigoplus_{n\geq0}\sum_{k=0}^n\pi_{n,k}\circ p_k,
\end{equation}
where $\pi_{n,k}$ is defined by \eqref{pi_n,k}.

Let $k\neq l, k,l\leq n$ and $u=u_1\otimes\dots\otimes u_k\in (e\mathcal{N}e)^{\otimes k}, v=v_1\otimes \dots \otimes v_l\in (e\mathcal{N}e)^{\otimes l}$. Each term of the product
$$E_n(u_1\otimes\dots\otimes u_k\otimes (\mathbbm{1}-e)^{\otimes(n-k)})E_n(v_1\otimes \dots \otimes v_l\otimes (\mathbbm{1}-e)^{\otimes(n-l)})$$ contains the product $u_i(\mathbbm{1}-e)=0$ or $(\mathbbm{1}-e)v_i=0$ on some $j$-th place of tensor product, and therefore,
$$\pi_{n,k}(u)\pi_{n,l}(v)=\binom{n}{k}\binom{n}{l}E_n(u\otimes (\mathbbm{1}-e)^{\otimes(n-k)})E_n(v\otimes (\mathbbm{1}-e)^{\otimes(n-l)})=
0$$ for all $u\in(e\mathcal{N}e)^{\otimes k},v\in(e\mathcal{N}e)^{\otimes l}.$
Consequently, since $p_k(z_1)\in (e\mathcal{N}e)_s^{\otimes k}, p_l(z_2)\in (e\mathcal{N}e)_s^{\otimes l}, z_1,z_2\in M_s(e\mathcal{N}e)$ we have
\begin{equation}\label{pi_n,k,l}
(\pi_{n,k}\circ p_k)(z_1)\cdot(\pi_{n,l}\circ p_l)(z_2)=0,\quad z_1,z_2\in M_s(e\mathcal{N}e), k\neq l, k,l\leq n.
\end{equation}

Since, by Lemma \ref{second homomorphism lemma}, $\pi_{n,k}:(e\mathcal{N}e)^{\otimes k}_s\to\mathcal{N}^{\otimes n}_s,$ $n\geq k,$ is a $*-$homomorphism, for every $z_1,z_2\in M_s(e\mathcal{N}e),$  we obtain
\begin{gather*}
\Big(\sum_{k=0}^n(\pi_{n,k}\circ p_k)(z_1)\Big)\Big(\sum_{l=0}^n(\pi_{n,l}\circ p_l)(z_2)\Big)=\sum_{k,l=0}^n(\pi_{n,k}\circ p_k)(z_1)(\pi_{n,l}\circ p_l)(z_2)\\
\stackrel{\eqref{pi_n,k,l}}{=}\sum_{k=0}^n(\pi_{n,k}\circ p_k)(z_1)(\pi_{n,k}\circ p_k)(z_2)=\sum_{k=0}^n(\pi_{n,k}\circ p_k)(z_1z_2),
\end{gather*}
and similarly $$\sum_{k=0}^n(\pi_{n,k}\circ p_k)(z^*)=(\sum_{k=0}^n(\pi_{n,k}\circ p_k)(z))^*, z\in M_s(e\mathcal{N}e).$$
Therefore,  the mapping
$$\pi_{e,\mathbbm{1}}=\bigoplus_{n\geq 0}\sum_{k=0}^n\pi_{n,k}\circ p_k:M_s(e\mathcal{N}e)\to M_s(\mathcal{N})$$
is a $*-$homomorphism as a direct sum of $*$-homomorphisms. In addition, (see also the end of the proof of Lemma \eqref{strange equality}), $$\sum_{k=1}^n\pi_{n,k}(\mathbbm{1}^{\otimes k})\stackrel{\eqref{alpha_pi}}{=}\sum_{k=1}^n\alpha_{n,k}(\mathbbm{1})\stackrel{\eqref{E_n_hatx}}{=}\sum_{k=1}^n\sum_{\substack{A\in 2^n\\|A|=k}}e_A=\sum_{A\in 2^n}e_A=\mathbbm{1}^{\otimes n}.$$
That is $\pi_{e,\mathbbm{1}}$ is a unital $*$-homomorphism from $M_s(e\mathcal{N}e)$ into $M_s(f\mathcal{N}f)$.

Let us show now, that the $*$-homomorphism  $\pi_{e,\mathbbm{1}}$ is trace preserving, i.e.
\begin{equation*}
\sigma_e(z)=\sigma_\mathbbm{1}\circ\pi_{e,\mathbbm{1}}(z), \quad z\in  M_s(e\mathcal{N}e).
\end{equation*}
Clearly, this assertion needs to be verified on each direct summand of $M_s(e\mathcal{N}e)$.

Let $z=(e\mathcal{N}e)_s^{\otimes k}$. Since $p_m(z)=0$ for all $m\neq k$, by the definition of $\sigma_e,$ (see \eqref{eq_sigma_e}) we have
$$\sigma_e(z)=\exp(-\tau(e))\frac{\tau^{\otimes k}(z)}{k!}.$$
Again, by the definition of $\sigma_\mathbbm{1},$ we have
\begin{gather*}
\begin{split}
(\sigma\circ\pi_{e,\mathbbm{1}})(z)&=\sigma\Big(\bigoplus_{n\geq 0}\sum_{k=0}^n\pi_{n,k}\circ p_k(z)\Big)
\\&=\sigma\Big(\bigoplus_{n\geq k}\pi_{n,k}(z)\Big)=
\exp(-\tau(\mathbbm{1}))\sum_{n\geq k}\frac1{n!}\tau^{\otimes n}(\pi_{n,k}(z))\\
&=\exp(-\tau(\mathbbm{1}))\sum_{n\geq k}\frac1{n!}\binom{n}{k}\tau^{\otimes k}(z)\tau^{n-k}(e-\mathbbm{1})=\exp(-\tau(e))\frac{\tau^{\otimes k}(z)}{k!},
\end{split}
\end{gather*}
that is $\sigma_e(z)=(\sigma_\mathbbm{1}\circ\pi_{e,\mathbbm{1}})(z).$

Thus the $*$-homomorphism $\pi_{e,\mathbbm{1}}:M_s(e\mathcal{N}e)\to M_s(\mathcal{N})$ is trace preserving, and therefore (see e.g. \cite{DdePS}) $\pi_{e,\mathbbm{1}}$ can be extended to a unital $*$-homomorphism from
$S(M_s(e\mathcal{N}e), \sigma_e)$ into $S(M_s(\mathcal{N}),\sigma)$
In particular, $\pi_{e,\mathbbm{1}}$ maps
$\mathcal{L}_1(M_s(e\mathcal{N}e), \sigma_e)$ into $\mathcal{L}_1(M_s(\mathcal{N}),\sigma)$. We preserve the notation $\pi_{e,\mathbbm{1}}$ for this extension.

For the proof of equality \eqref{pi_K} fix $x\in\mathcal{L}_1(e\mathcal{N}e,\tau)$.
By the definition of $\mathcal{K}_e$ (see \eqref{k fin def}), we have $p_0(\mathscr{K}_ex)=0$ and
$$p_k(\mathscr{K}_ex)=\sum_{m=1}^ke^{\otimes(m-1)}\otimes x\otimes e^{\otimes(k-m)},\quad k\geq1.$$
It follows now from Lemma \ref{strange equality} that
\begin{gather*}
\begin{split}
\pi_{e,\mathbbm{1}}(\mathscr{K}_ex)&\stackrel{\eqref{pi def}}{=}
\bigoplus_{n\geq 0}\sum_{k=1}^n\pi_{n,k}\circ p_k(\mathscr{K}_ex)\\
&\stackrel{\eqref{pi_n,k}}{=}
0\oplus\bigoplus_{n\geq1}\sum_{k=1}^n\binom{n}{k}\sum_{m=1}^kE_n(e^{\otimes(m-1)}\otimes x\otimes e^{\otimes(k-m)}\otimes(\mathbbm{1}-e)^{\otimes(n-k)})\\
&=0\oplus\bigoplus_{n\geq1}\sum_{k=1}^nk\binom{n}{k}E_n(x\otimes e^{\otimes(k-1)}\otimes(\mathbbm{1}-e)^{\otimes(n-k)})\\
&\stackrel{L. \ref{strange equality}}{=}0\oplus\bigoplus_{n\geq1}
\sum_{k=1}^n\mathbbm{1}^{\otimes(k-1)}\otimes x\otimes \mathbbm{1}^{\otimes(n-k)}=
\mathscr{K}_\mathbbm{1}x.
\end{split}
\end{gather*}
\end{proof}

Note that, the equality \eqref{eq_sigma_pi} can be viewed as a noncommutative analogue of the compatibility of measures in Kolmogorov extension theorem \cite{Fremlin}.

The following lemma shows that $*$-homomorphisms constructed in Proposition \ref{k compat} are well-compatible, that allows us to construct the inductive limit of algebras of the form $M_s(e\mathcal{N}e), e\in\mathcal{P}_{fin}(\mathcal{N})$ (see the proof of Theorem \ref{junge thm}).

\begin{lem}\label{pi compat} We have $\pi_{e,g}(z)=(\pi_{f,g}\circ\pi_{e,f})(z)$ whenever $e\leq f\leq g, z\in M_s(e\mathcal{N}e).$
\end{lem}
\begin{proof} Let $x\in e\mathcal{N}e$ and let $n\geq k\geq m.$ We have (cf. \eqref{E_n_hatx} and \eqref{alpha_pi}),
$$\pi_{k,m}^{e,f}(x^{\otimes m})=\sum_{\substack{A_1\in 2^k\\|A_1|=m}}
\mathcal Z_{A_1}^{x,f-e},$$
where $\mathcal Z_{A_1}^{x,f-e}$ is defined in \eqref{x_A},
and therefore,
$$\pi_{n,k}^{f,g}(\pi_{k,m}^{e,f}(x^{\otimes m}))=\binom{n}{k}E_n\Big(\Big(\sum_{\substack{A_1\in 2^k\\|A_1|=m}}
\mathcal{Z}_{A_1}^{x,f-e}\Big)\otimes (g-f)^{\otimes (n-k)}\Big)=$$
$$=\sum_{\substack{A_1\in 2^k\\|A_1|=m}}\sum_{\substack{A_2\in 2^n\\A_2\supset A_1\\|A_2|=k}}
\mathcal Z_{A_1,A_2}^{x,f-e,g-f},$$
where
$$
\mathcal Z_{A_1,A_2}^{x,f-e,g-f}=z_1\otimes\dots\otimes z_n\in\mathcal{N}^{\otimes n},\quad
z_i=\left\{\begin{aligned}x,&\quad i\in A_1\\
f-e,& \quad i\in A_2\setminus A_1\\
g-f, & \quad \mathrm{otherwise }\end{aligned}\right.\quad 1\leq i\leq n,
$$
And hence we obtain,
\begin{equation}\label{pc3}\sum_{k=m}^n\pi_{n,k}^{f,g}(\pi_{k,m}^{e,f}(x^{\otimes m}))=\sum_{\substack{A_1\in 2^n\\|A_1|=m}}\sum_{\substack{A_2\in 2^n\\A_2\supset A_1}}\mathcal Z_{A_1,A_2}^{x,f-e,g-f}.\end{equation}
Next we shall prove that for a fixed $A_1\in 2^n$ such that $|A_1|=m,$ we have
\begin{equation}\label{pc2}
\mathcal Z_{A_1}^{x,g-e}=\sum_{\substack{A_2\in 2^n\\A_2\supset A_1}}\mathcal Z_{A_1,A_2}^{x,f-e,g-f}.
\end{equation}  Indeed, observing that the following formula
 \begin{equation}\label{pc1}
(y+z)^{\otimes (n-m)}=\sum_{A\in 2^{n-m}} \mathcal Z_{A}^{y,z}, \ \ \mbox{for all} \ \ y,z\in\mathcal N.
\end{equation} holds, that can be easily established by induction, plugging $x$ on fixed $A_1$ positions  between elements of tensors
on the both sides of \eqref{pc1}, we obtain that
 \begin{equation}\label{pc4}
\mathcal Z_{A_1}^{x,y+z}=\sum_{\substack{A\in 2^{n}\\A\cap A_1=\varnothing}} \mathcal Z_{A_1,A_1\cup A}^{x,y,z}=
\sum_{\substack{A_2\in 2^n\\A_2\supset A_1}}\mathcal Z_{A_1,A_2}^{x,y,z}, \ \ \mbox{for all} \ \ y,z\in\mathcal N.
\end{equation}
Now taking $y=f-e,$ $z=g-f$ in \eqref{pc4}, we arrive at \eqref{pc2}. Combining \eqref{pc3} and \eqref{pc2}, we obtain that
$$\sum_{k=m}^n\pi_{n,k}^{f,g}(\pi_{k,m}^{e,f}(x^{\otimes m}))
=\sum_{\substack{A_1\in 2^n\\|A_1|=m}}Z_{A_1}^{x,g-e}.$$
Consequently,  we have that
\begin{equation}\label{compatpi1}
\pi_{n,m}^{e,g}(x^{\otimes m})=\sum_{k=m}^n\pi_{n,k}^{f,g}(\pi_{k,m}^{e,f}(x^{\otimes m})),\, x\in e\mathcal{N}e.
\end{equation}
 By polarization identity every element $z_1\otimes\dots\otimes z_m,z_i\in (e\mathcal{N}e),1\leq i\leq m,$ is a (finite) linear combination of the elements of the form $x^{\otimes m}$, and hence equality  \eqref{compatpi1} holds for these elements too. Since  elements of the form $z_1\otimes\dots\otimes z_m,z_i\in (e\mathcal{N}e),1\leq i\leq m$ are dense in $(e\mathcal{N}e)_s^{\otimes m}$ (in the weak$^*$ topology) and by construction the $*-$homomorphism $\pi_{n,k}$ is continuous in the weak$^*$ topology (see \eqref{pi_n,k}), we infer the equality \eqref{compatpi1} for all $z\in(e\mathcal{N}e)_s^{\otimes m}.$

Since for every $z\in(e\mathcal{N}e)_s^{\otimes m},$ we have $p_k(z)=0$ for $k\neq m$ and $p_m(z)=z,$ it follows from construction of the mapping $\pi_{e,g}$ that
\begin{gather*}
\begin{split}
\pi_{f,g}(\pi_{e,f}(z))&\stackrel{\eqref{pi def}}{=}\pi_{f,g}(\bigoplus_{k\geq 0}\sum_{m=0}^k(\pi_{k,m}^{e,f}\circ p_m)(z))=\pi_{f,g}(\bigoplus_{k\geq m}\pi_{k,m}^{e,f}(z))\\
&=\bigoplus_{n\geq m}\sum_{k=m}^n\pi_{n,k}^{f,g}(\pi_{k,m}^{e,f}(z))\stackrel{\eqref{compatpi1}}{=}\bigoplus_{n\geq m}\pi_{n,m}^{e,g}(z)=\pi_{e,g}(z).
\end{split}
\end{gather*}
for every $z\in(e\mathcal{N}e)_s^{\otimes m}.$ It follows immediately that
$$\pi_{e,g}(z)=\pi_{f,g}(\pi_{e,f}(z)),\quad z\in M_s(e\mathcal{N}e).$$
\end{proof}

The following theorem is the main result of this section. It presents the construction of Kruglov operator for arbitrary semifinite von Neumann algebra.
The idea of the construction used in Theorem \ref{junge thm} is adopted from \cite{junge}.

\begin{thm}\label{junge thm} Let $(\mathcal{N},\tau)$ be a semifinite von Neumann algebra. There exists a finite von Neumann algebra $(\mathcal{M},\sigma)$ and a positive bounded linear operator $\mathscr{K}:\mathcal{L}_1(\mathcal{N},\tau)\to\mathcal{L}_1(\mathcal{M},\sigma)$ such that
\begin{equation}\label{kinf main property}
\sigma(\exp(i\mathscr{K}x))=\exp(\tau(\exp(ix)-\mathrm{1})),\quad x=x^*\in\mathcal{L}_1(\mathcal{N},\tau).
\end{equation}
If $\mathcal{N}$ is hyperfinite, then so is $\mathcal{M}.$
\end{thm}
\begin{proof} Let $e_n\in\mathcal{N}$ be a sequence of $\tau-$finite projections such that $e_n\uparrow \mathbbm{1}$.  For all $n\in\mathbb{N}$ consider the finite von Neumann algebra $(e_n\mathcal{N}e_n,\tau_{e_n})$.
By Proposition \ref{k compat} and Lemma \ref{pi compat} there exist unital $*$-homomorphisms
$$\pi_{e_n,e_m}:M_s(e_n\mathcal{N}e_n)\to M_s(e_m\mathcal{N}e_m), n\leq m $$ such that $ \pi_{e_n,e_m}=\pi_{e_k,e_m}\circ\pi_{e_n,e_k}$ for all $n\leq k\leq m$. Hence, by Lemma \ref{dirlimdef1} there exists the inductive limit of unital $C^*$-algebras
$$M=\lim_nM_s(e_n\mathcal{N}e_n)$$ and  canonical $*$-homomorphisms $\pi_n:M_s(e_n\mathcal{N}e_n)\to M$.
Furthermore, since every algebra $M_s(e_n\mathcal{N}e_n)$ is equipped with finite trace $\sigma_{e_n}$ and by Proposition \ref{k compat} $*$-homomorphisms $\pi_{e_n,e_k}, k\leq n,$ preserve these traces,   by Lemma \ref{dirlimdef2} the equality
 $\sigma=\lim_n\sigma_{e_n}$ defines a tracial state on $M$.
By Lemma \ref{lem_cyclic} the trace $\sigma$ can be extended on the von Neumann algebra $\mathcal{M}:=\overline{\theta(M)}^{wo}$, where $\theta$ is cyclic representation of $M$ for $\sigma$ on a Hilbert space $H$. Thus $(\mathcal{M},\sigma)$ is a finite von Neumann algebra with faithful normal finite trace $\sigma$.
In addition, if  $\mathcal{N}$ is hyperfinite, then  every von Neumann algebra $e_n\mathcal{N}e_n$ is hyperfinite. Consequently, by Lemma \ref{lem_il_hyperf} the algebra $\mathcal{M}$ is also hyperfinite.

%Recall the following properties of $\pi_n.$
%\begin{equation}\label{pie compat}
%\sigma_{e_n}=\sigma\circ\pi_n,\quad \pi_m\circ\pi_{e_n,e_m}=\pi_n,\quad n\leq m.
%\end{equation}

\textbf{Step 1.} Here we define a one-parameter unitary group $\exp(it\mathscr{K}x)$ for every $x=x^*\in e_k\mathcal{N}e_k$.
Fix $k\in\mathbb{N}$ and $x=x^*\in e_k\mathcal{N}e_k$. Since the $*$-homomorphism $\pi_{e_k,e_n}$ given by \eqref{pi_K} is trace preserving, it follows that it
is continuous in the topology of convergence in measure \cite[Proposition 3.3]{DdePS}, and therefore, for all $n\geq k$ we have
\begin{gather*}
\begin{split}
\pi_{e_k,e_n}(\exp(it\mathscr{K}_{e_k}x))&=\pi_{e_k,e_n}\left(\sum_{j=0}^\infty\frac{(it\mathscr{K}_{e_k}x)^j}{j!}\right)=
\sum_{j=0}^\infty\frac{\pi_{e_k,e_n}\left((it\mathscr{K}_{e_k}x)^j\right)}{j!}\\
&=\sum_{j=0}^\infty\frac{(it)^j\left(\pi_{e_k,e_n}(\mathscr{K}_{e_k}x)\right)^j}{j!}\stackrel{\eqref{pi_K}}{=} \sum_{j=0}^\infty\frac{(it\mathscr{K}_{e_n}x)^j}{j!}=\exp(it\mathscr{K}_{e_n}x).
\end{split}
\end{gather*}

Therefore,
$$(\theta\circ\pi_k)(\exp(it\mathscr{K}_{e_k}x))\stackrel{\eqref{eq_pi_n}}{=}(\theta\circ\pi_n\circ\pi_{e_k,e_n})(\exp(it\mathscr{K}_{e_k}x))=
(\theta\circ\pi_n)(\exp(it\mathscr{K}_{e_n}x))\in\mathcal{M}$$ for all $ n\geq k.$ Consequently, the formula $t\to(\theta\circ\pi_k)(\exp(it\mathscr{K}_{e_k}x))$ defines a one-parameter unitary group in $\mathcal{M},$ which does not depend on $k$. By Stone's theorem, there exists a self-adjoint operator $\mathscr{K}x$ on $H$ such that
\begin{equation}\label{k def}
\exp(it\mathscr{K}x)=(\theta\circ\pi_k)(\exp(it\mathscr{K}_{e_k}x)),\quad x=x^*\in e_k\mathcal{N}e_k.
\end{equation}
In addition, the operator $\mathscr{K}x$ is affiliated with $\mathcal{M}$ since $(\theta\circ\pi_k)(\exp(it\mathscr{K}_{e_k}x))\in\mathcal{M}$.
Thus, for every $x=x^*\in e_k\mathcal{N}e_k$ it is correctly defined a self-adjoint operator  $\mathscr{K}x\in S(\mathcal{M},\sigma)$, since the trace $\sigma$ finite on $\mathcal{M}$.

\textsc{Step 2.}
We claim that the mapping $x\mapsto \mathscr{K}x, x\in e_k\mathcal{N}e_k,$ is linear. It is clear that $\mathscr{K}(\lambda x)=\lambda\mathscr{K}(x)$ for all $\lambda\in\mathbb{R}$ and $x=x^*\in e_k\mathcal{N}e_k$.
Let $x=x^*,y=y^*\in e_k\mathcal{N}e_k.$ By Trotter formula \cite[Theorem VIII.31]{ReedSimon} (which is applicable to the operators $\mathscr{K}x$ and $\mathscr{K}y$, since the sum $\mathscr{K}x+\mathscr{K}y$ is understood in strong sense), we have
$$\exp(it(\mathscr{K}x+\mathscr{K}y))=(so)-\lim_{m\to\infty}(e^{i\frac{t}{m}\mathscr{K}x}e^{i\frac{t}{m}\mathscr{K}y})^m,$$
where the notation $(so)-\lim_{m\rightarrow\infty}$ stands for the limit in topology of convergence in measure.
 By \cite[Theorem 2, p. 59]{Dixmier}, $*$-homomorphisms on von Neumann algebras are continuous with respect to the strong operator topology. Therefore,
\begin{gather*}
\begin{split}
\exp(it(\mathscr{K}x+\mathscr{K}y))\\
&\stackrel{\eqref{k def}}{=}(so)-\lim_{m\to\infty}((\theta\circ\pi_k)(\exp(i\frac{t}{m}\mathscr{K}_{e_k}x))(\theta\circ\pi_k)(\exp(i\frac{t}{m}\mathscr{K}_{e_k}y)))^m\\
&=(so)-\lim_{m\to\infty}(\theta\circ\pi_k)\Big(\exp(i\frac{t}{m}\mathscr{K}_{e_k}x)\exp(i\frac{t}{m}\mathscr{K}_{e_k}y)\Big)^m\\
&=(\theta\circ\pi_k)\Big((so)-\lim_{m\to\infty}(\exp(i\frac{t}{m}\mathscr{K}_{e_k}x)\exp(i\frac{t}{m}\mathscr{K}_{e_k}y))^m\Big)\\
&=(\theta\circ\pi_k)(\exp(it(\mathscr{K}_{e_k}x+\mathscr{K}_{e_k}y)))\\
&=(\theta\circ\pi_k)(\exp(it\mathscr{K}_{e_k}(x+y)))\\
&\stackrel{\eqref{k def}}{=}\exp(it\mathscr{K}(x+y)).
\end{split}
\end{gather*}
Consequently,
\begin{equation}\label{k prelim linearity}
\mathscr{K}(x+y)=\mathscr{K}x+\mathscr{K}y,\quad x=x^*,y=y^*\in e_k\mathcal{N}e_k.
\end{equation}
Thus, the mapping $x\mapsto \mathscr{K}x$ is linear on $e_k\mathcal{N}e_k$.

\textbf{Step 3.}
Let us show that $\mathscr{K}$ acts boundedly with respect to $L_1$-norm.
Firstly, we show that for $x=x^*\in e_k\mathcal{N}e_k$ the operators $\mathscr{K}_{e_k}x$ and $\mathscr{K_{e_k}}x$ are equimeasurable. For a Schwartz function $f\in S(\mathbb{R})$, we have (see \eqref{eq_Fourier})
\begin{gather*}
\begin{split}
(\theta\circ\pi_k)(f(\mathscr{K}_{e_k}x))&=(\theta\circ\pi_k)\int_{\mathbb{R}}(\mathcal{F}(f))(\lambda)\exp(i\lambda\mathscr{K}_{e_k}x)d\lambda\\
&=\int_{\mathbb{R}}(\mathcal{F}(f))(\lambda)(\theta\circ\pi_k)\exp(i\lambda\mathscr{K}_{e_k}x)d\lambda
\\&\stackrel{\eqref{k def}}{=}
\int_{\mathbb{R}}(\mathcal{F}(f))(\lambda)\exp(i\lambda\mathscr{K}x)d\lambda=f(\mathscr{K}x),
\end{split}
\end{gather*}
where the second equality holds since $(\theta\circ\pi_k)$ is a trace preserving $*$-homomorphism (see \eqref{eq_sigma_theta_pi}).
For an interval $(s,t), s,t\in\mathbb{R}, s<t, $ choose a sequence $f_n$ of Schwartz functions, such that $f_n\downarrow\chi_{(s,t)}$. Then we have
$$(\theta\circ\pi_k)(f_n(\mathscr{K}_{e_k}x))\rightarrow(\theta\circ\pi_k)(E_{\mathscr{K}_{e_k}x}(s,t))$$
and
$$f_n(\mathscr{K}x)\rightarrow E_{\mathscr{K}x}(s,t)$$
in the topology of convergence in measure.
Consequently,
\begin{equation}\label{spec.pr}
(\theta\circ\pi_k)(E_{\mathscr{K}_{e_k}x}(s,t))=(E_{\mathscr{K}x}(s,t)), s<t.
\end{equation}
In particular, by general properties of spectral projections, we infer that equality \eqref{spec.pr} holds for $t=\infty$. Hence, by \eqref{eq_sigma_theta_pi} $$\sigma_{e_k}(E_{\mathscr{K}_{e_k}x}(s,\infty))=\sigma((\theta\circ\pi_k)(E_{\mathscr{K}_{e_k}x}(s,\infty)))=\sigma(E_{\mathscr{K}x}(s,\infty))$$ for all $s\in\mathbb{R}$. In other words, $d_{\mathscr{K}_{e_k}x}(s)=d_{\mathscr{K}x}(s)$ for all $s\in\mathbb{R}$ and $x=x^*\in e_k\mathcal{N}e_k$. Thus, the operators $\mathscr{K}_{e_k}x$ and $\mathscr{K_{e_k}}x$ are equimeasurable, in particular,
\begin{equation}\label{eq_Kek_K}
\mu(\mathscr{K}_{e_k}x)=\mu(\mathscr{K}(x)),\quad x=x^*\in e_k\mathcal{N}e_k.
\end{equation}

Furthermore, from equality \eqref{spec.pr} we also have that $(\theta\circ\pi_k)(E_{\mathscr{K}_{e_k}x}(-\infty,0))=(E_{\mathscr{K}x}(-\infty,0))$, and therefore,   $\mathscr{K}x\geq 0$ for all $e_k\mathcal{N}e_k\ni x\geq 0$. Hence, by Proposition \ref{pr_prop_K} and equality \eqref{eq_Kek_K} we infer that $\sigma(\mathscr{K}x)=\sigma_{e_k}(\mathscr{K}_{e_k}x)=\tau(x).$
Consequently, $\|\mathscr{K}x\|_1\leq 4\|x\|_1$ for every $x\in e_k\mathcal{N}e_k.$ Since the union of $e_k\mathcal{N}e_k,$ $k\geq0,$ is $\|\cdot\|_1-$dense in $\mathcal{L}_1(\mathcal{N},\tau)$ \cite{C-S} it follows that $\mathscr{K}$ admits a bounded extension $$\mathscr{K}:\mathcal{L}_1(\mathcal{N},\tau)\to\mathcal{L}_1(\mathcal{M},\sigma).$$
Thus, we have constructed the finite von Neumann algebra $\mathcal{M}$ and operator $\mathscr{K}$ acting from $\mathcal{L}_1(\mathcal{N},\tau)$ into $\mathcal{L}_1(\mathcal{M},\sigma)$.

\textbf{Step 4.} We now prove equality \eqref{kinf main property}. For this purpose fix
 $x=x^*\in e_k\mathcal{N}e_k.$ We have
$$\sigma(\exp(i\mathscr{K}x))\stackrel{\eqref{k def}}{=}(\sigma\circ\pi_k)(\exp(i\mathscr{K}_{e_k}(x)))\stackrel{\eqref{eq_sigma_n}}{=}\sigma_{e_k}(\exp(i\mathscr{K}_{e_k}(x))).$$
It follows now from Lemma \ref{k finite} that
\begin{equation}\label{pred  nuj}
\sigma(\exp(i\mathscr{K}x))=\exp(\tau_{e_k}(\exp(ix)-\mathbbm{1})),\quad x=x^*\in e_k\mathcal{N}e_k.
\end{equation}
Using again $\|\cdot\|_1$-density of union of  $e_k\mathcal{N}e_k,$ $k\geq0,$ in $\mathcal{L}_1(\mathcal{N},\tau),$ for a fixed $x=x^*\in \mathcal{L}_1(\mathcal{N},\tau)$ choose a sequence $\{x_k\}_{n\geq 1},x_k\in e_k\mathcal{N}e_k$ such that $\|x_k-x\|_1\rightarrow 0,k\rightarrow\infty$. It follows from Duhamel's formula (see e.g. \cite[Lemma 5.2]{ACDS}) that $\|e_k(\exp(ix_k)-\exp(ix))e_k\|_1\rightarrow 0,$
as well as $\|\exp(i\mathscr{K}x_k)-\exp(i\mathscr{K}x)\|_1\rightarrow 0,$ for $k\rightarrow\infty$. Therefore, $$\sigma(\exp(i\mathscr{K}x_k))\rightarrow\sigma(\exp(i\mathscr{K}x))$$ and $$\tau_{e_k}(\exp(ix_k)-\mathbbm{1})=\tau(e_k(\exp(ix_k)-\mathbbm{1})e_k)\rightarrow\tau((\exp(ix)-\mathbbm{1})).$$
Thus, by \eqref{pred  nuj} we infer
 equality \eqref{kinf main property}.
\end{proof}

The following corollary lists main properties of the Kruglov operator constructed in Theorem \ref{junge thm}.
\begin{cor}\label{infincor}Let $(\mathcal{N},\tau)$ be a semifinite atomless von Neumann algebra,  let $(\mathcal{M},\sigma)$ and  $\mathscr{K}$ be the finite von Neumann algebra and the Kruglov operator constructed in Theorem \ref{junge thm}, respectively. The following assertions hold:
\begin{enumerate}[{\rm (i)}]
\item\label{kkka} If self-adjoint operators $x,y\in\mathcal{L}_1(\mathcal{N},\tau)$ are equimeasurable, then $\mathscr{K}x,\mathscr{K}y\in \mathcal{L}_1(\mathcal{M},\sigma)$ are equimeasurable too. In addition, if $x=x^*\in\mathcal{L}_1(\mathcal{N},\tau)$ is equimeasurable with $y\in L_1(0,1),$ then the elements $\mathscr{K}x\in\mathcal{L}_1(\mathcal{M})$ and $Ky\in L_1(0,1)$ are equimeasurable.
\item\label{kkk_L_2}The operator $\mathscr{K}$ acts boundedly from $(\mathcal{L}_1\cap\mathcal{L}_2)(\mathcal{N})$, equipped with the norm $\|\cdot\|_{\mathcal{L}_1\cap\mathcal{L}_2}=\|\cdot\|_{\mathcal{L}_1}+\|\cdot\|_{\mathcal{L}_2}$  into $\mathcal{L}_2(\mathcal{M})$.
\item\label{kkkcom}For all $x,y\in (\mathcal{L}_1\cap\mathcal{L}_2)(\mathcal{N})$ the equality $[\mathscr{K}x,\mathscr{K}y]=\mathscr{K}[x,y]$ holds.
\item\label{kkke} If $x_k=x_k^*\in\mathcal{L}_1(\mathcal{N},\tau),$ $1\leq k\leq n,$ are such that $x_jx_k=0$ for $j\neq k,$ then the random variables $\mathscr{K}x_k\in \mathcal{L}_1(\mathcal{M},\sigma),$ $1\leq k\leq n,$ are independent in the sense of Section \ref{indep subsect}.
\item\label{kkkc} The operator $\mathscr{K}$ maps symmetrically distributed operators from $\mathcal{L}_1(\mathcal{N},\tau)$ into symmetrically distributed operators from $\mathcal{L}_1(\mathcal{M},\sigma)$.
\end{enumerate}
\end{cor}
\begin{proof}
\eqref{kkka}. We have that
\begin{gather*}
\begin{split}
\tau(\exp(ix)-\mathbbm{1})&\stackrel{\eqref{xdg1}}{=}\tau(\exp(ix_+)-\mathbbm{1})+\tau(\exp(ix_-)-\mathbbm{1})\\
&=\int_0^\infty(\exp(i\mu(s,x_+))-1)ds+
\int_0^\infty(\exp(i\mu(s,x_-))-1)ds\\
&=\int_0^\infty(\exp(i\mu(s,y_+))-1)ds+\int_0^\infty(\exp(i\mu(s,y_-))-1)ds\\
&=\tau(\exp(iy_+)-\mathbbm{1})+\tau(\exp(iy_-)-\mathbbm{1})\stackrel{\eqref{xdg1}}{=}\tau(\exp(iy)-\mathbbm{1}),\
\end{split}
\end{gather*}
 and therefore by equality \eqref{kinf main property}
we obtain that
\begin{gather*}
\begin{split}
\sigma(\exp(i\mathscr{K}x))&=\exp(\tau(\exp(ix)-\mathbbm{1}))=\exp(\tau(\exp(iy)-\mathbbm{1}))=\sigma(\exp(i\mathscr{K}y))
\end{split}
\end{gather*}
By \cite[Theorem 2.12.1]{Gned_Kol} the preceding equality implies that $d_{\mathscr{K}x}=d_{\mathscr{K}y}$. Since the von Neumann algebra $\mathcal{M}$ is finite we infer that the  operators $\mathscr{K}x$ and $\mathscr{K}y$ are equimeasurable. If $x$ is equimeasurable with $y\in L_1[0,1]$, then the assertion is proved similarly using Lemma \ref{lem2}.

\eqref{kkk_L_2}. Suppose firstly that $0\leq x\in\mathcal{N}$ with finite support. There exist a finite sequence of projections $\{q_k\}_{k=0}^{N-1}\subset\mathcal{N}$ with $\tau(q_k)=1$ such that $x\leq \sum_{k=0}^{N-1}\|x\|_\infty q_k$. Since the operator $\mathscr{K}$ is positive we have that $\mathscr{K}x\leq \|x\|_\infty\sum_{k=0}^{N-1}\mathscr{K}q_k$. Hence
$$\mu(\mathscr{K}x)\leq\|x\|_\infty\mu(\sum_{k=0}^{N-1}\mathscr{K}q_k)\leq \|x\|_\infty ND_N\mu(\mathscr{K}q_0).$$
In addition by \eqref{kkka} the equality $\mu(\mathscr{K}q_0)=\mu(K\mathbbm{1})\in\bigcap_{n\geq 1}L_n(\Omega)$ holds \cite[Theorem 4.4]{AS}, and therefore we obtain $\mathscr{K}x\in\bigcap_{n\geq 1}\mathcal{L}_n(\mathcal{M})$.  Furthermore, by standard arguments one can prove that $\frac{d^2}{dt^2}\sigma(\exp(it\mathscr{K}x))\big|_{t=0}=-\sigma((\mathscr{K}x)^2)$ and
$\frac{d^2}{dt^2}\exp(\tau(\exp(itx)-\mathbbm{1}))\big|_{t=0}=-\tau(x^2)-\tau^2(x)$. Consequently, by \eqref{kinf main property} we infer that $\sigma((\mathscr{K}x)^2)=\tau(x^2)+\tau^2(x)$, and therefore  $\|\mathscr{K}x\|_2\leq \|x\|_1+\|x\|_2$.

Now, for arbitrary $x\geq 0$ from $(\mathcal{L}_1\cap\mathcal{L}_2)(\mathcal{N})$ we set $x_n=xE_x[\frac1n,n]$. We have that $x_n\uparrow x$ and $x_n\xrightarrow{\|\cdot\|_1}x$ \cite{C-S}, and therefore $\mathscr{K}x_n\uparrow$ and by Theorem \ref{junge thm} $\mathscr{K}x_n\xrightarrow{\|\cdot\|_1}\mathscr{K}x$. Since in addition $\|\mathscr{K}x_n\|_2\leq\|x\|_1+\|x_2\|$ for all $n\in\mathbb{N}$ we obtain that $\|\mathscr{K}x\|_2\leq\|x\|_1+\|x\|_2$ due to Fatou property of $\mathcal{L}_2(\mathcal{M})$ (see e.g. \cite{ddp}).

\eqref{kkkcom}. Let $x,y\in(\mathcal{L}_1\cap\mathcal{L}_2)(\mathcal{N})$. Set $x_n=e_nxe_n, y_n=e_nye_n,$ where $e_n$ is sequence of $\tau$-finite projections from the proof of Theorem \ref{junge thm}. By Proposition \ref{pr_prop_K} \eqref{prop_K5} the equality $[\mathscr{K}x_n,\mathscr{K}y_n]=\mathscr{K}[x_n,y_n]$ holds. Since $[x_n,y_n]\xrightarrow{\|\cdot\|_1}[x,y]$ \cite{C-S} we have that $\mathscr{K}[x_n,y_n]\xrightarrow{\|\cdot\|_1}\mathscr{K}[x,y].$ On the other hand, by \eqref{kkk_L_2} we obtain
\begin{gather*}
\begin{split}
\|[\mathscr{K}x_n,\mathscr{K}y_n]&-[\mathscr{K}x,\mathscr{K}y]\|_1\leq \|[\mathscr{K}x_n-\mathscr{K}x,\mathscr{K}y_n]\|_1+\|[\mathscr{K}x,\mathscr{K}y_n-\mathscr{K}y]\|_1\\
&\leq
2\|\mathscr{K}x_n-\mathscr{K}x\|_2\|\mathscr{K}y_n\|_2+2\|\mathscr{K}x\|_2\|\mathscr{K}y_n-\mathscr{K}y\|_2
\\&\leq 2\|\mathscr{K}\|^2_{\mathcal{L}_1\cap\mathcal{L}_2\to\mathcal{L}_2}(\|x_n-x\|_{\mathcal{L}_1\cap\mathcal{L}_2}\|y_n\|_{\mathcal{L}_1\cap\mathcal{L}_2} +\|x\|_{\mathcal{L}_1\cap\mathcal{L}_2}\|y_n-y\|_{\mathcal{L}_1\cap\mathcal{L}_2}).
\end{split}
\end{gather*}
And therefore, since $\|x_n-x\|_{\mathcal{L}_1\cap\mathcal{L}_2}\rightarrow 0, \|y_n-y\|_{\mathcal{L}_1\cap\mathcal{L}_2}\rightarrow 0$ we infer that $$[\mathscr{K}x_n,\mathscr{K}y_n]\xrightarrow{\|\cdot\|_1}[\mathscr{K}x,\mathscr{K}y]$$
and, therefore, $[\mathscr{K}x,\mathscr{K}y]=\mathscr{K}[x,y].$

The proof of property \eqref{kkke} is identical to that of Theorem \ref{fincor} and is, therefore, omitted. Property \eqref{kkkc} is a straightforward consequence of \eqref{kkka} and \eqref{kkke}.

\end{proof}

\begin{cor}
Let $\mathcal{R}$ be the hyperfinite $II_1$ factor, $H$ be a separable infinite-dimensional Hilbert space and let $\mathcal{R}\bar{\otimes}\mathcal{L}(H)$ be the hyperfinite $II_\infty$ factor. Then the operator $\mathscr{K}$ defined in Theorem \ref{junge thm} acts from $\mathscr{K}:\mathcal{L}_1(\mathcal{R}\bar{\otimes}\mathcal{L}(H),\tau\otimes{\rm Tr})$ into $\mathcal{L}_1(\mathcal{R},\tau).$
\end{cor}
\begin{proof} Set $\mathcal{N}=\mathcal{R}\bar{\otimes}\mathcal{L}(H)$ in Theorem \ref{junge thm}. We have $\mathscr{K}:\mathcal{L}_1(\mathcal{N},\tau)\to\mathcal{L}_1(\mathcal{M},\sigma),$ where $\mathcal{M}$ is hyperfinite. Since every hyperfinite II$_1-$type von Neumann algebra embeds into $\mathcal{R}$ (see e.g. the proof of Theorem 5.2 in \cite{hrs}), it follows that $\mathscr{K}:\mathcal{L}_1(\mathcal{N},\tau)\to\mathcal{L}_1(\mathcal{R},\tau).$
\end{proof}

\section{Proof of Theorem \ref{orlicz embedding theorem}}

 In this section we prove Theorem \ref{orlicz embedding theorem}.
 Let us first note that the
equivalence \eqref{Oetii}$\Leftrightarrow$\eqref{Oetiii} of
Theorem \ref{orlicz embedding theorem} follows from the equivalence
\eqref{pri}$\Leftrightarrow$\eqref{prv} of Theorem \ref{prop_O}.

\subsection{Proof of Theorem \ref{orlicz embedding theorem}: \eqref{Oeti}$\Rightarrow$\eqref{Oetii}}

\begin{proof}[Proof of Theorem \ref{orlicz embedding theorem}: \eqref{Oeti}$\Rightarrow$\eqref{Oetii}]
%This implication is somewhat similar to Theorem 1.d.7 in \cite{LT2}. However, the proof is quite different.
%\begin{lem}\label{upper p lower 2} If an ideal $\mathcal{I}$ admits an isomorphic embedding into $\mathcal{L}_p(\mathcal{R}),$ $1\leq p<2,$ then it satisfies upper $p-$estimate and lower $2-$estimate. That is, given operators $A_k\in\mathcal{I},$ $1\leq k\leq n,$ such that $A_kA_j=A_k^*A_j=0$ for $j\neq k,$ we have
%$${\rm const}\cdot(\sum_{k=1}^n\|A_k\|_{\mathcal{I}}^2)^{1/2}\leq\|\sum_{k=1}^nA_k\|_{\mathcal{I}}\leq{\rm const}\cdot(\sum_{k=1}^n\|A_k\|_{\mathcal{I}}^p)^{1/p}.$$
%\end{lem}
 %Since the implication \eqref{Oeti}$\Rightarrow$\eqref{Oetii} is in fact valid for an arbitrary Banach ideal $\mathcal I,$ we do not specify the ideal in the proof below.
We first prove that if  a Banach ideal $\mathcal{I}$ isomorphically embeds into  $\mathcal{L}_p(\mathcal{R})$ ($1\le p<2$), then
$\mathcal{I}$ satisfies an upper $p$-estimate and a lower $2$-estimate.

  Let $T$ be an isomorphic embedding of $\mathcal{I}$ into $\mathcal{L}_p(\mathcal{R})$ ($1\le p<2$). Fix a constant $C>0$ such that \begin{equation}\label{eq_Oet1}C^{-1}\|A\|_{\mathcal{I}}\leq\|T(A)\|_p\leq C\|A\|_{\mathcal{I}} \ \ \ \mbox{for every} \ \  A\in \mathcal I.\end{equation}
Fix $n\ge 0$ and operators $A_k\in\mathcal{I},$ $0\leq k\leq n,$
such that $l(A_j)l(A_k)=0$ for $j\neq k.$
 Observe that $l(A_j)l(A_k)=0$ implies $A_j^*A_k=0$ for $j\neq k.$ It follows that for almost all $t\in(0,1),$ we have
\begin{equation}\label{eq_Oet2}\Big\|\sum_{k=0}^nr_k(t)A_k\Big\|_{\mathcal{I}}=\Big\|\sum_{k=0}^nA_k\Big\|_\mathcal{I},\end{equation}
where $r_k,$ $0\leq k\leq n,$ are the Rademacher functions.
Applying subsequently \eqref{eq_Oet2}, \eqref{eq_Oet1} and
identifying the Bochner space
$L_p((0,1),\mathcal{L}_p(\mathcal{R}))$ with
$\mathcal{L}_p(\mathcal{R}\bar{\otimes}L_{\infty}(0,1),\tau\otimes
dm)$ (see e.g \cite[Lemma 6.2]{BGM}), we obtain
\begin{multline}\label{eq_Oet3}\Big\|\sum_{k=0}^nA_k\Big\|_{\mathcal{I}}=\Big(\int_0^1\Big\|\sum_{k=0}^nr_k(t)A_k\Big\|_{\mathcal{I}}^pdt\Big)^{1/p}\leq C\Big(\int_0^1\Big\|\sum_{k=0}^nr_k(t)T(A_k)\Big\|_p^pdt\Big)^{1/p}\\=C\Big\|\sum_{k=0}^nr_kT(A_k)\Big\|_{L_p((0,1),\mathcal{L}_p(\mathcal{R}))}=
C\Big\|\sum_{k=0}^nT(A_k)\otimes
r_k\Big\|_{\mathcal{L}_p(\mathcal{R}\bar{\otimes}L_{\infty}(0,1))}.\end{multline}
  Similarly, we obtain
\begin{equation}\label{eq_Oet4}\Big\|\sum_{k=0}^nA_k\Big\|_{\mathcal{I}}
\geq C^{-1}\Big\|\sum_{k=0}^nT(A_k)\otimes
r_k\Big\|_p,\end{equation} where for brevity we use $\|\cdot\|_p$
instead of
$\|\cdot\|_{\mathcal{L}_p(\mathcal{R}\bar{\otimes}L_{\infty}(0,1))}$.
By Theorem \ref{lp theorem}, we have
\begin{equation}\label{eq_Oet5}{\rm const}\cdot\Big(\sum_{k=0}^n\|T(A_k)\|_p^2\Big)^{1/2}\leq\Big\|\sum_{k=0}^nT(A_k)\otimes r_k\Big\|_p\leq\Big(\sum_{k=0}^n\|T(A_k)\|_p^p\Big)^{1/p}.\end{equation}
Combining  \eqref{eq_Oet3} and \eqref{eq_Oet4} with \eqref{eq_Oet5} and \eqref{eq_Oet1}, we conclude
$${\rm const}\cdot\Big(\sum_{k=0}^n\|A_k\|_{\mathcal{I}}^2\Big)^{1/2}\leq\Big\|\sum_{k=0}^nA_k\Big\|_{\mathcal{I}}\leq{\rm const}\cdot\Big(\sum_{k=0}^n\|A_k\|_{\mathcal{I}}^p\Big)^{1/p}.$$
Now let $\mathcal{I}=\mathcal L_M(H)$ isomorphically embeds into  $\mathcal{L}_p(\mathcal{R})$ ($1\le p<2$). Then, as we proved above,
$\mathcal L_M(H)$  satisfies an upper $p$-estimate and a lower $2$-estimate. The implication follows from Theorem \ref{prop_O}.
\end{proof}
%\begin{lem}\label{interpolation criterion} An Orlicz function $M$ is equivalent to a $p-$convex and $2-$concave Orlicz function (on the interval $[0,1]$) with $1\leq p<2$ if and only if $\mathcal{L}_M(H)$ is an interpolation space for the couple $(\mathcal{L}_p(H),\mathcal{L}_2(H)).$
%\end{lem}

\subsection{Proof of Theorem \ref{orlicz embedding theorem}: \eqref{Oetiii}$\Leftrightarrow$\eqref{Oetiv}}

\begin{proof}[Proof of Theorem \ref{orlicz embedding theorem}: \eqref{Oetiii}$\Rightarrow$\eqref{Oetiv}] Let $M$ satisfies \eqref{Oetiii}. We may assume that $M$ is $p$-convex and $2$-concave. Define an Orlicz function $M_1$ by setting $$M_1(t):=M(t^{1/p}), \ \ \ t>0.$$
According to Remark \ref{rem_Orl_mon_submaj}, for every $0\leq x\in l_{M_1}$ and for every $0\leq y\in l_{\infty},$ we have
\begin{equation}\label{eq_(iii)to(iv)1}y\prec\prec x\Longrightarrow \|y\|_{l_{M_1}}\leq\|x\|_{l_{M_1}},\end{equation}
for the notion of submajorization ("$\prec\prec$") see
Definition \ref{hlp majorization def first}.
Now if $0\le x\in l_M,$ then $x^p\in l_{M_1}$ and $y^p\in l_{\infty}.$ Applying \eqref{eq_(iii)to(iv)1} to $x^p$ and $y^p,$ we obtain that the implication
$$y^p\prec\prec x^p\Longrightarrow \|y\|_{l_M}\leq\|x\|_{l_M}$$ holds.
By \cite[Theorem 2]{LSh}, we have that $l_M$ is an interpolation space for the couple $(l_p,l_{\infty}).$

Since by the assumption $M$ is $2$-concave,  it similarly follows from \cite[Theorem 3]{LSh} that  $l_M$ is an interpolation space for the couple $(l_1,l_2).$  Applying \cite[Corollary]{arazycwikel}, we obtain that $l_M$ is an interpolation space for the couple $(l_p,l_2).$ By \cite[Theorem 2.7]{arazy}, $\mathcal{L}_M(H)$ is an interpolation space for the couple $(\mathcal{L}_p(H),\mathcal{L}_2(H)).$
\end{proof}

\begin{proof}[Proof of Theorem \ref{orlicz embedding theorem}:  \eqref{Oetiv}$\Rightarrow$\eqref{Oetiii}]
Suppose that the Orlicz space $\mathcal{L}_M(H)$ is an interpolation space for the
couple $(\mathcal{L}_p(H),\mathcal{L}_2(H)).$ Then again by
\cite[Theorem 2.7]{arazy}, the commutative core $l_M$ is an interpolation space for the
couple $(l_p,l_2).$ For every $n\geq 1$ and $x=\{x(k)\}_{k\ge 0}\in l_\infty$.
Using \eqref{eq_D1},
we obtain
\begin{multline*}\|D_n\|_{l_M\to l_M}\leq{\rm const}\cdot \max\{n^{1/2},n^{1/p}\}={\rm const}\cdot n^{1/p},\\ \|D_{1/n}\|_{l_M\to l_M}\leq
{\rm const}\cdot \max\{n^{-1/2},n^{-1/p}\}={\rm const}\cdot n^{-1/2},\quad n\geq 1,\end{multline*}
where $D_n$ and $D_{1/n}$ are defined by \eqref{eq_D_discrete}.
Therefore, for $m\ge 1$, we have
$$\Big\|\sum_{j=0}^{nm-1}e_j\Big\|_{l_M}\leq\|D_n\|_{l_M\to l_M}\Big\|\sum_{j=0}^{m-1}e_j\Big\|_{l_M}\leq{\rm const}\cdot n^{1/p}\Big\|\sum_{j=0}^{m-1}e_j\Big\|_{l_M}$$
and
$$\Big\|\sum_{j=0}^{m-1}e_j\Big\|_{l_M}\leq\|D_{1/n}\|_{l_M\to l_M}\Big\|\sum_{j=0}^{nm-1}e_j\Big\|_{l_M}\leq{\rm const}\cdot n^{-1/2}\Big\|\sum_{j=0}^{nm-1}e_j\Big\|_{l_M}.$$ Thus, we arrive at \eqref{eq_(ii)to(iii)1}.
Arguing as in the proof of Theorem \ref{prop_O}
\eqref{priv}$\Rightarrow$\eqref{prv} (see \eqref{eq_(ii)to(iii)1}
and subsequent arguments),
 we obtain that $M$ is equivalent to a $p$-convex and $2$-concave Orlicz function.
\end{proof}

\subsection{Proof of Theorem \ref{orlicz embedding theorem}: \eqref{Oetiii}$\Rightarrow$\eqref{Oeti}}

The proof of this implication  is much more difficult and is in
fact completely different from previous ones. Here we use the concept of Kruglov operators developed in Sections \ref{krug fin}
and \ref{junge section}. Before we proceed with the proof of the
implication, let us present several technical estimates, which are important for the subsequent exposition.

%In this section, we prove that an Orlicz ideal in $\mathcal{L}(H),$ which is an interpolation space for the couple $(\mathcal{L}_p(H),\mathcal{L}_2(H)),$ $1\leq p<2,$ (distinct from $\mathcal{L}_p(H)$), can be isomorphically embedded into $\mathcal{L}_p(\mathcal{R}).$

%Let $x,y\in L_{\infty}(0,\infty)$ (or $x,y\in L_{\infty}(0,1)$). Recall that (see Definition \ref{hlp majorization def first}) that $y$ is submajorized by $x$ (written $y\prec\prec x$) if
%$$\int_0^t\mu(s,y)ds\leq\int_0^t\mu(s,x)ds,\quad t>0.$$

The following lemma yields a submajorization estimate for the element
$Kx,$ $x\in L_1(0,1),$ where $K$ is the Kruglov operator defined by \eqref{kastsuk}. Recall that the notion of submajorization ("$\prec\prec$") is given in
Definition \ref{hlp majorization def first} and the dilation operator $D$ is defined by \eqref{eq_D_continuous}.

\begin{lem}\label{maj lemma} For every $x\in L_1(0,1),$  we have
$$Kx\prec\prec\sum_{n=1}^{\infty}nD_{\frac1{e\cdot n!}}\mu(x).$$
\end{lem}
\begin{proof} Recall (see \eqref{kastsuk}) that
\begin{equation}\label{kastsuk_recall}
Kx=\sum_{n=1}^{\infty}\chi_{A_n}\otimes\Big(\sum_{m=1}^n
1^{\otimes(m-1)}\otimes x\otimes 1^{\otimes\infty}\Big),
\end{equation}
where $A_n,$ $n\geq0,$ are disjoint sets with $m(A_n)=1/(e\cdot n!)$ for all $n\geq0.$

It follows from the inequality \eqref{sum maj} that
$$\sum_{m=1}^n 1^{\otimes(m-1)}\otimes x\otimes 1^{\otimes\infty}\prec\prec n\mu(x),\quad n\ge 1$$
and, therefore,
$$\chi_{A_n}\otimes\Big(\sum_{m=1}^n 1^{\otimes(m-1)}\otimes x\otimes 1^{\otimes\infty}\Big)\prec\prec nD_{m(A_n)}\mu(x)=
nD_{\frac1{e\cdot n!}}\mu(x),\quad n\ge 1.$$
Again applying the inequality \eqref{sum maj}, we obtain
$$Kx\prec\prec\sum_{n=1}^{\infty}nD_{\frac1{e\cdot n!}}\mu(x).$$
\end{proof}

\begin{lem}\label{modified ms} If $1\le p\le 2$ and $x_k\in L_p(0,1),$ $1\leq k\leq n,$ then
$$\Big\|\bigoplus_{k=0}^{\infty}Kx_k\Big\|_{L_p+L_2}\sim\Big\|\bigoplus_{k=0}^{\infty}x_k\Big\|_{L_p+L_2}.$$
\end{lem}
\begin{proof} It follows from Lemma \ref{maj lemma} that
\begin{equation}\label{modified ms2}\bigoplus_{k=1}^{\infty}Kx_k\prec\prec\bigoplus_{k=1}^{\infty}\Big(\sum_{n=1}^{\infty}nD_{\frac1{e\cdot
n!}}\mu(x_k)\Big)=\sum_{n=1}^{\infty}n\Big(\bigoplus_{k=1}^{\infty}D_{\frac1{e\cdot
n!}}\mu(x_k)\Big).\end{equation} Since the norm $\|\cdot\|_{L_p+L_2}$ respects submajorization (see e.g.~\cite[Theorem 4.10]{KPS}), using \eqref{modified ms2}
and the triangle inequality in $(L_p+L_2)(0,\infty),$ we infer
that
\begin{gather}\label{modified ms1}\begin{split}\Big\|\bigoplus_{k=0}^{\infty}Kx_k\Big\|_{L_p+L_2}&\leq\sum_{n=1}^{\infty}n\Big\|\bigoplus_{k=1}^{\infty}D_{\frac1{e\cdot n!}}\mu(x_k)\Big\|_{L_p+L_2}
\\&=\sum_{n=1}^{\infty}n\Big\|D_{\frac1{e\cdot
n!}}\Big(\bigoplus_{k=1}^{\infty}\mu(x_k)\Big)\Big\|_{L_p+L_2}
 \\& \leq \Big(\sum_{n=1}^{\infty}n\big\|D_{\frac1{e\cdot
n!}}\big\|_{L_p+L_2\to
L_p+L_2}\Big)\Big\|\bigoplus_{k=0}^{\infty}x_k\Big\|_{L_p+L_2}.\end{split}\end{gather}
Taking into account that $$\big\|D_u\big\|_{L_p+L_2\to
L_p+L_2}\le u^{1/p},\quad 0<u\leq 1,$$ we infer that
\begin{equation}\label{modified ms3}\sum_{n=1}^{\infty}n\big\|D_{\frac1{e\cdot
n!}}\big\|_{L_p+L_2\to
L_p+L_2}\leq\sum_{n=1}^{\infty}\frac{n}{(e\cdot n!)^{1/p}}\leq
{\rm const}<\infty.\end{equation} Combining \eqref{modified ms1}
and \eqref{modified ms3}, we obtain
$$\Big\|\bigoplus_{k=0}^{\infty}Kx_k\Big\|_{L_p+L_2}\leq {\rm const} \ \Big\|\bigoplus_{k=0}^{\infty}x_k\Big\|_{L_p+L_2}.$$

On the other hand, taking only the first term in
\eqref{kastsuk_recall}, we obtain that
\begin{equation}\label{modified ms4}
\mu(Kx_k)\geq\mu(A_1\otimes x_k\otimes \mathbbm 1^{\infty})=D_{m(A_1)}\mu(x_k)=D_{1/e} \ \mu(x_k), \ \ 1\leq k\leq n.
\end{equation}
Observing that
$$\big\|D_ux\big\|_{L_p+L_2}\ge u^{1/p}\|x\|_{L_p+L_2},\ \ \mbox{for all} \ \ x\in(L_p+L_2)(0,\infty), \ \ 0<u\leq 1,$$
and applying \eqref{modified ms4}, we conclude that
\begin{align*}
\Big\|\bigoplus_{k=0}^{\infty}Kx_k\Big\|_{L_p+L_2}& \geq\Big\|\bigoplus_{k=0}^{\infty}D_{1/e} \ \mu(x_k)\Big\|_{L_p+L_2}\\
&=\Big\|D_{1/e}\Big(\bigoplus_{k=0}^{\infty}\mu(x_k)\Big)\Big\|_{L_p+L_2} \geq{\rm const}\Big\|\bigoplus_{k=0}^{\infty}x_k\Big\|_{L_p+L_2}.  
\end{align*}
\end{proof}

The following lemma is the (noncommutative analogue of) special case of the inequality due to Johnson and Schechtman (see \cite[Theorem~1]{js}). This \lq\lq noncommutative Rosenthal inequality\rq\rq (see \cite{JX}) has become popular in compressed sensing (see e.g. \cite{Rauhut}).

\begin{lem}\label{js lemma} Let $1\le p\le 2$ and let $A_k\in\mathcal{L}_p(\mathcal{R}),$ $1\leq k\leq n,$ be independent symmetrically distributed random variables. Then
$$\Big\|\sum_{k=1}^nA_k\Big\|_p\sim\Big\|\bigoplus_{k=1}^nA_k\Big\|_{\mathcal{L}_p+\mathcal{L}_2} \ \ \ n\ge 1.$$
\end{lem}
\begin{proof} Recall that the independence of random variables $A_k\in\mathcal{L}_p(\mathcal{R}),$ $1\leq k\leq n,$ implies that
 $A_k A_l=A_l A_k,$  $1\le k,l\le n$ (see definition of independence in Subsection \ref{indep subsect}).
 Hence, there exists a commutative von Neumann subalgebra $\mathcal{M}$ in $\mathcal{R}$ such that $A_k\in\mathcal{L}_p(\mathcal{M}),$ $1\leq k\leq n.$
 Without loss of generality, we may assume that
  $\mathcal{M}$ is $^*$-isomorphic to $L_{\infty}(0,1).$
  Hence, the assertion can be viewed as the statement
   about functions. The latter is proved in \cite[Theorem 1]{js}.
\end{proof}

The next lemma describes an isomorphic embedding of the Banach space $(\mathcal{L}_p+\mathcal{L}_2)(\mathcal{R}\overline{\otimes}\mathcal{L}(H))$ into the Banach space $\mathcal{L}_p(\mathcal{R}).$ Recall that the von Neumann algebra $\mathcal{R}\overline{\otimes}\mathcal{L}(H)$ is a hyperfinite $II_\infty$ factor. By $\|\cdot\|_{\mathcal{L}_p+\mathcal{L}_2}$ we denote the standard norm on the space $(\mathcal{L}_p+\mathcal{L}_2)(\mathcal{R}\overline{\otimes}\mathcal{L}(H)).$
For brevity we set $r=r_1,$ where $r_1$ is a Rademacher function defined in \eqref{eq_Rad}.

\begin{lem}\label{lpl2 embed} Let $1\le p\le2$ and let $A\in\mathcal{L}_p(\mathcal{R}\overline{\otimes}\mathcal{L}(H))$ be finitely supported.  Then
$$\|\mathscr{K}(A\otimes r)\|_p\sim\|A\|_{\mathcal{L}_p+\mathcal{L}_2}.$$
Moreover, the mapping $\mathscr{R}:A\to\mathscr{K}(A\otimes r)$ extends to an isomorphic embedding of $(\mathcal{L}_p+\mathcal{L}_2)(\mathcal{R}\overline{\otimes}\mathcal{L}(H))$ into $\mathcal{L}_p(\mathcal{R}).$
\end{lem}
\begin{proof} Observe that for every finitely supported $A\in\mathcal{L}_p(\mathcal{R}\overline{\otimes}\mathcal{L}(H)),$ the element $A\otimes r$ belongs to $\mathcal{L}_1(\mathcal{R}\bar{\otimes}\mathcal{L}(H)).$ In particular, $\mathscr{K}(A\otimes r)$ is well defined.

Suppose first that $A=A^*.$ Since $A$ is finitely supported, it follows that one can select $A_k=A_k^*\in\mathcal{L}_p(\mathcal{R}\overline{\otimes}\mathcal{L}(H)),$ $1\leq k\leq n,$ such that $A_kA_j=0$ for $k\neq j,$ $A=\sum_{k=1}^nA_k$ and $(\tau\otimes{\rm Tr})({\rm supp}(A_k))\leq 1.$ Since the random varaibles $A_k\otimes r,$ $1\leq k\leq n,$ are symmetrically distributed, it follows from Corollary \ref{infincor} \eqref{kkkc} and \eqref{kkke} that the random variables $\mathscr{K}(A_k\otimes r),$ $1\leq k\leq n,$ are independent and symmetrically distributed. Applying Lemma \ref{js lemma}, we infer that
\begin{equation}\label{lpl2 embed1}\|\mathscr{K}(A\otimes r)\|_p=\Big\|\sum_{k=1}^n\mathscr{K}(A_k\otimes r)\Big\|_p\sim\Big\|\bigoplus_{k=1}^n\mathscr{K}(A_k\otimes r)\Big\|_{\mathcal{L}_p+\mathcal{L}_2}.\end{equation}
For every $1\leq k\leq n,$ select a function $x_k\in L_p(0,1)$ which is equimeasurable with $A_k\otimes r.$ By Corollary \ref{infincor} \eqref{kkka}, we have that  $\mathscr{K}(A_k\otimes r)$ is equimeasurable with $Kx_k.$ Therefore, by Lemma \ref{modified ms}, we have
\begin{align}\label{lpl2embed2}\Big\|\bigoplus_{k=1}^n\mathscr{K}(A_k\otimes r)\Big\|_{\mathcal{L}_p+\mathcal{L}_2}=\Big\|\bigoplus_{k=1}^nKx_k\Big\|_{L_p+L_2} \stackrel{L. \ref{modified ms}}{\sim}\Big\|\bigoplus_{k=1}^nx_k\Big\|_{L_p+L_2}=\Big\|\bigoplus_{k=1}^nA_k\Big\|_{\mathcal{L}_p+\mathcal{L}_2}.
\end{align}
Combining \eqref{lpl2 embed1} and \eqref{lpl2embed2}, we obtain
$$\|\mathscr{K}(A\otimes r)\|_p\sim \Big\|\bigoplus_{k=1}^nA_k\Big\|_{\mathcal{L}_p+\mathcal{L}_2}=\|A\|_{\mathcal{L}_p+\mathcal{L}_2}.$$

For a not necessarily self-adjoint finitely supported operator $A\in\mathcal{L}_p(\mathcal{R}\overline{\otimes}\mathcal{L}(H)),$ we have
\begin{align*}\|A\|_{\mathcal{L}_p+\mathcal{L}_2}&\sim\|\Re A\|_{\mathcal{L}_p+\mathcal{L}_2}+\|\Im A\|_{\mathcal{L}_p+\mathcal{L}_2}\\
& \sim\|\mathscr{K}(\Re A\otimes r)\|_p+\|\mathscr{K}(\Im A\otimes r)\|_p\sim\|\mathscr{K}(A\otimes r)\|_p.
\end{align*}
The final assertion follows from the fact that the linear subspace of all elements $A\in \mathcal{L}_p(\mathcal{R}\overline{\otimes}\mathcal{L}(H))$ with finite support is dense in $(\mathcal{L}_p+\mathcal{L}_2)(\mathcal{R}\overline{\otimes}\mathcal{L}(H))$.
\end{proof}

By $\mathcal{O}_{p,q}$ we denote the class of $p$-convex and $q$-concave Orlicz functions $M$ on $\mathbb R$ such that $M(1)=1.$
 For $1\le p\le2$ we define an Orlicz function $M_p\in\mathcal{O}_{p,2}$ by setting
\begin{equation}\label{lp+l2 orlicz}
M_p(t):=
\begin{cases}
t^2,\quad 0\leq t\leq 1\\
1+\frac{2}{p}(t^p-1),\quad t\geq 1.
\end{cases}
\end{equation}

One can show that $(L_p+L_2)(0,\infty)=L_{M_p}(0,\infty)$  and, therefore,
$$(\mathcal{L}_p+\mathcal{L}_2)(\mathcal{R}\bar{\otimes}\mathcal{L}(H))=\mathcal{L}_{M_p}(\mathcal{R}\bar{\otimes}\mathcal{L}(H)).$$
In what follows, we use the Orlicz norm $\|\cdot\|_{\mathcal{L}_p+\mathcal{L}_2}:=\|\cdot\|_{\mathcal{L}_{M_p}}$ in the space $(\mathcal{L}_p+\mathcal{L}_2)(\mathcal{R}\bar{\otimes}\mathcal{L}(H)).$

The following lemma shows how {\it some} Orlicz spaces can be embedded into $(\mathcal{L}_p+\mathcal{L}_2)(\mathcal{R}\overline{\otimes}\mathcal{L}(H)).$

\begin{lem}\label{ms} Let  $A\in\mathcal{L}(H)$ be an operator of finite rank and let  $A_0\in\mathcal{L}_p(\mathcal{R})$ for $1\le p\le2$. Then
$$\|A_0\otimes A\|_{\mathcal{L}_p+\mathcal{L}_2}=\|A\|_{\mathcal{L}_M},$$
where $\mathcal{L}_M$ is an Orlicz ideal associated to the Orlicz function $M$ defined by setting $$M(t):=\tau(M_p(|tA_0|)),\ \ \ t>0.$$
\end{lem}
\begin{proof} %For every $x\in(\mathcal{L}_p+\mathcal{L}_2)(\mathcal{R}\bar{\otimes}\mathcal{L}(H))$,  we have
%\begin{equation}\label{ms2}\|X\|_{\mathcal{L}_p+\mathcal{L}_2}=\inf\Big\{\lambda:\ \int_0^{\infty}M_p\Big(\frac{\mu(s,X)}{\lambda}\Big)ds=1\Big\}.\end{}
Employing standard properties of singular valued function $\mu,$ by definition of the function $M$, we have
\begin{equation}\label{ms1}M(t)=\int_0^1\mu(s,M_p(|tA_0|))ds=\int_0^1M_p(t\mu(s,A_0))ds,\quad t>0.\end{equation}
Denote, for brevity, $\|A\|_{\mathcal{L}_M}=\lambda.$
Observe that there exists measure preserving mapping $\mathcal G:\mathbb Z^+\times (0,1)\to (0,\infty),$
$\mathcal G(k,s)=k+s$ such that $\mathcal G$ extends to the measure preserving $^*$-homomorphism from $ L_\infty(0,1)\otimes l_\infty$ into $L_\infty(0,\infty).$ Observe also that the function
$f_{A,A_0}:k+s \mapsto \mu(s,A_0)\mu(k,A)$ is equimeasurable with $\mu(A_0\otimes A).$
Hence, using \eqref{ms1} and the facts listed above, we obtain
\begin{align}\label{ms3}1&={\rm Tr}\Big(M\Big(\frac{|A|}{\lambda}\Big)\Big)=\sum_{k\geq0}\mu\Big(k,M\Big(\frac{|A|}{\lambda}\Big)\Big)\\
&=\sum_{k\geq0}M\Big(\frac{\mu(k,A)}{\lambda}\Big)=
\sum_{k\geq0}\int_0^1M_p\Big(\frac{\mu(s,A_0)\mu(k,A)}{\lambda}\Big)ds\nonumber \\
&=\int_0^{\infty}M_p\Big(\frac{\mu(u,A_0\otimes A)}{\lambda}\Big)du=
({\rm Tr}\otimes \tau)\Big(M_p\Big(\frac{A_0\otimes A}{\lambda}\Big)\Big).  \nonumber \end{align}
Now \eqref{ms3} and well-known property of Orlicz norms (see e.g.~\cite[I,~1.2,~Proposition~11]{R-R}) implies
$\|A_0\otimes A\|_{\mathcal{L}_{M_p}}=\|A\|_{\mathcal{L}_M},$ that completes the proof of the lemma.
\end{proof}

Recall (see \cite{KPS}) that the function $\varphi:(0,\infty)\to(0,\infty)$ is called quasi-concave if
$$\varphi(s)\leq\varphi(t),\quad \frac{\varphi(s)}{s}\ge\frac{\varphi(t)}{t},\quad 0\leq s\leq t.$$

The following lemma is somewhat similar to the one in \cite{AS}. However, the proof presented here is simpler.

\begin{lem}\label{bk lemma} Let $\varphi:[0,1]\to[0,1]$ be a quasi-concave function such that $\varphi(0)=0$ and $\varphi(1)=1.$ For every $d>1,$ there exists $0\leq x\in L_1(0,1)$ with $\|x\|_1\leq 1$ such that
$$\frac14\varphi(t)\leq\int_0^1\min\{x(s),tx^d(s)\}ds\leq \frac52\varphi(t),\quad t\in[0,1].$$
\end{lem}
\begin{proof} Suppose first that $\varphi'(0)=\infty.$ It follows from the equalities
$$\lim_{t\to0}\varphi(t)=0,\quad \lim_{t\to0}\frac{\varphi(t)}{t}=\infty$$
that one can select sequences $t_k\downarrow0$ and $s_k\downarrow0$ such that
$$\varphi(t_k)=2^{-k}\varphi(1),\quad \frac{s_k}{\varphi(s_k)}=2^{-k}\frac{1}{\varphi(1)},\quad k\geq0.$$
Define a sequence $u_k\downarrow0$ by setting $u_0=t_0=1,$ $u_1=s_1$ and
$$u_{2k}=\max\{t_l:\ t_l<u_{2k-1}\},\quad u_{2k+1}=\max\{s_l:\ s_l\leq u_{2k}\},\quad k\geq 1.$$
{\bf Step 1.} Observe that
\begin{equation}\label{bk lemma1}\varphi(u_{2k})\leq 2^{l-k}\varphi(u_{2l}),\quad \frac{\varphi(u_{2k+1})}{u_{2k+1}}\leq 2^{k-l}\frac{\varphi(u_{2l+1})}{u_{2l+1}}, \ \ \ l\le k.\end{equation}
Indeed, by definition of the sequence $\{u_{2k}\}_{k\ge 1}$ we have that $$u_{2k}=t_m, \ \ \ u_{2k-2}=t_n, \ \ \mbox{for some} \ \ n<m.$$
The computation $$\varphi(u_{2k})=\varphi(t_m)=2^{-m}=2^{-m+n}\varphi(t_n)=2^{-m+n}\varphi(u_{2k-2})\le 2^{-1}\varphi(u_{2k-2}).$$
implies the first inequality in \eqref{bk lemma1}. The second one follows similarly.

\noindent{\bf Step 2.}
Fix $d>1.$
Since $\varphi(u_{2k})\le \varphi(t_k)=2^{-k}$ for all $k\geq0,$ it follows that
\begin{equation}\label{bk lemma2}\sum_{k\geq0}\frac12u_{2k}^{1/(d-1)}\varphi(u_{2k})\leq\frac12\sum_{k\geq0}\varphi(u_{2k})\leq\frac12\sum_{k\geq0}2^{-k}=1.\end{equation}
Equality \eqref{bk lemma2} guarantees that $$\sum_{k\ge 1} m\Big((0,\frac12u_{2k}^{1/(d-1)}\varphi(u_{2k}))\Big)\le 1,$$ where $m$ is the Lebesgue measure on $(0,1).$ Therefore, choosing the decreasing order of the summands below, we construct  a positive decreasing function
(see \eqref{def_direct_sum_of_operators})
\begin{equation}\label{x main def}
x:=\bigoplus_{k\geq0}u_{2k}^{-1/(d-1)}\chi_{(0,\frac12u_{2k}^{1/(d-1)}\varphi(u_{2k}))},
\end{equation}
that belongs to $L_1(0,1).$ Equality \eqref{bk lemma2} implies that $\|x\|_1\leq 1$.
For the function $x\in L_1(0,1)$ constructed above, for $t\in[0,1]$,
we have
\begin{align}\label{bk lemma3} &2\int_0^1\min\{x(s),tx^d(s)\}ds =\sum_{k\geq0}u_{2k}^{1/(d-1)}\varphi(u_{2k})\cdot\min\{u_{2k}^{-1/(d-1)},t\cdot u_{2k}^{-d/(d-1)}\}\\
&=\sum_{u_{2k}\leq t}\varphi(u_{2k})+\sum_{u_{2k}>t}t\cdot\frac{\varphi(u_{2k})}{u_{2k}}.\nonumber
\end{align}
Thus, to prove the lemma, it is enough to show that the right hand side of \eqref{bk lemma3} satisfies
\begin{equation}\label{bk lemma4}\frac12 \varphi(t)\le\sum_{u_{2k}\leq t}\varphi(u_{2k})+t\cdot\sum_{u_{2k}>t}\frac{\varphi(u_{2k})}{u_{2k}}\le 5\varphi(t), \ \ \ t\in[0,1].\end{equation}

We first show the right hand side of \eqref{bk lemma4}. Let $l\in\mathbb N$ be such that $t\in[u_{2l},u_{2l-2}).$

\noindent{\bf Step 3.}  Applying \eqref{bk lemma1}, we have
$$\sum_{u_{2k}\leq t}\varphi(u_{2k})=\sum_{k= l}^{\infty}\varphi(u_{2k})\leq\sum_{k= l}^{\infty}2^{l-k}\varphi(u_{2l})=2\varphi(u_{2l})\leq2\varphi(t).$$
Similarly,
\begin{align}\label{bk lemma5}\sum_{u_{2k}>t}\frac{\varphi(u_{2k})}{u_{2k}} &=\frac{\varphi(u_{2l-2})}{u_{2l-2}}+\sum_{k=0}^{l-2}\frac{\varphi(u_{2k})}{u_{2k}}\\
& \leq\frac{\varphi(u_{2l-2})}{u_{2l-2}}+\sum_{k=0}^{l-2}\frac{\varphi(u_{2k+1})}{u_{2k+1}}
\leq
\frac{\varphi(u_{2l-2})}{u_{2l-2}}+\sum_{k=0}^{l-2}2^{k-l+2}\frac{\varphi(u_{2l-3})}{u_{2l-3}}\nonumber \\ &\leq\frac{\varphi(u_{2l-2})}{u_{2l-2}}+2\frac{\varphi(u_{2l-3})}{u_{2l-3}}\leq 3\frac{\varphi(u_{2l-2})}{u_{2l-2}}\leq\frac{3\varphi(t)}{t}. \nonumber
\end{align}
Thus, we obtain the right hand side of \eqref{bk lemma3}.

\noindent{\bf Step 4.} Next we prove the left hand side of \eqref{bk lemma3}. We consider two cases  $t\in[u_{2l},u_{2l-1})$ and
$t\in[u_{2l-1},u_{2l-2})$ separately.
Let $t\in[u_{2l},u_{2l-1}).$ Since $u_{2l}=t_m$ for some $m$ and $t_{m-1}\geq u_{2l-1},$ it follows from the definition of ${u_{2k}}$ that
$$\sum_{u_{2k}\leq t}\varphi(u_{2k})+\sum_{u_{2k}>t}t\cdot\frac{\varphi(u_{2k})}{u_{2k}}\geq\varphi(u_{2l})=\frac12\varphi(t_{m-1})\geq\frac12\varphi(u_{2l-1})\geq\frac12\varphi(t).$$
Let $t\in[u_{2l-1},u_{2l-2}).$ Since $u_{2l-1}=s_m$ for some $m$ and $s_{m-1}\geq u_{2l-2},$ by the definition of ${u_{2k+1}}$, we have that
$$\sum_{u_{2k}\leq t}\varphi(u_{2k})+\sum_{u_{2k}>t}t\cdot\frac{\varphi(u_{2k})}{u_{2k}}\geq$$
$$\geq  t\cdot\frac{\varphi(u_{2l-2})}{u_{2l-2}}  \geq t\cdot\frac{\varphi(s_{m-1})}{s_{m-1}}=\frac{t}{2}\cdot\frac{\varphi(s_m)}{s_m}=\frac{t}{2}\cdot\frac{\varphi(u_{2l-1})}{u_{2l-1}}\geq\frac{t}{2}\cdot\frac{\varphi(t)}{t}=\frac12\varphi(t).$$
The two steps above prove the assertion in the case $\varphi'(0)=\infty.$

\noindent {\bf Step 5.} Let $\varphi'(0)<\infty.$ Then the sequence $\{s_k\}$ is finite and, therefore, the sequence $\{u_{2k}\}_{k\ge 0}$ is also finite. Let $m+1$ and  $n+1$ be the lengths  of the sequences $\{s_k\}$ and  $\{u_{2k}\}_{k\ge 0}$, respectively. If $t\in[u_{2l},u_{2l-2}),$ $l=0,\ldots,n$ then the assertion can be shown similarly (see Steps 1 and 2 above). Suppose that $t\in[0,u_{2n}).$ The right hand side of \eqref{bk lemma4} can be estimated similarly to \eqref{bk lemma5}, as follows
  \begin{equation}\label{bk lemma6}
 \sum_{u_{2k}\leq t}\varphi(u_{2k})+t\cdot\sum_{u_{2k}>t}\frac{\varphi(u_{2k})}{u_{2k}}=t\cdot\sum_{k=0}^n\frac{\varphi(u_{2k})}{u_{2k}}\le
 3\varphi(t).
  \end{equation}
Now we estimate the left hand side of \eqref{bk lemma4}. Since the sequence $\{t_k\}$ is infinite, it follows that the recursive procedure stops at $u_{2n}.$ Indeed, if $u_{2n+1}$ were defined, then by definition of $\{u_{2k}\},$ we would have that $u_{2n+2}$ is defined, which contradicts with the choice of $n$. Since    $u_{2n+1}$ is not defined, it follows that there is no $s_l$ such that $s_l\le u_{2n}.$ Therefore,  $s_k> u_{2n}$ for all $k=1,...,m$ and
$$\varphi'(0)={\lim_{s\to \, 0}} \ \frac{\varphi(s)}{s}<2\frac{\varphi(s_m)}{s_m}\le 2\frac{\varphi(u_{2n})}{u_{2n}}.$$
Thus,  we have
$$t\cdot\sum_{k=0}^n\frac{\varphi(u_{2k})}{u_{2k}}\ge t\cdot\frac{\varphi(u_{2n})}{u_{2n}} \ge \frac{t}2 \ \varphi'(0)\ge \frac{t}2 \cdot\frac{\varphi(t)}{t}=\frac{\varphi(t)}{2},$$
which together with \eqref{bk lemma6} completes the proof.

\end{proof}

Now, we are fully prepared to present the proof the implication.

\begin{proof}[Proof of Theorem \ref{orlicz embedding theorem}: \eqref{Oetiii}$\Rightarrow$\eqref{Oeti}]

 %Let $M$ be an Orlicz function on $(0,\infty)$ such that the ideal $\mathcal{L}_M(H)$ isomorphically embeds into $\mathcal{L}_p(\mathcal{R}).$ By Lemma \ref{upper p lower 2} $\mathcal{L}_M(H)$ satisfies upper $p-$estimate and lower $2-$estimate. By Lemma \ref{justify}, $M$ is equivalent to a $p-$convex and $2-$concave Orlicz function. By Lemma \ref{interpolation criterion}, $\mathcal{L}_M(H)$ is an interpolation space for the couple $(\mathcal{L}_p(H),\mathcal{L}_2(H)).$

 Let $M$ be equivalent to a $p$-convex and $2$-concave Orlicz function. Without loss of generality, we may assume that  $M\in \mathcal{O}_{p,2}$. Define the function $\varphi$ by setting
  $$\varphi(t):=\frac{M\big(t^{\frac1{2-p}}\big)}{t^{\frac{p}{2-p}}}, \ \ \mbox{so that} \ \ M(t)=t^p\varphi(t^{2-p}), \ \  t\in(0,1),$$
   Since $M$ is $p$-convex and $2$ concave, application of Lemma \ref{as lemma} implies that $\varphi$ is quasi-concave. It follows from Lemma \ref{as lemma} and $p$-convexity of $M$ that there exists a limit
$$\lim_{t\to0}\frac{M(t)}{t^p}=\lim_{t\to0}\varphi(t^{2-p}).$$
Since $M(t)\not\sim t^p,$ it follows that the limit above is $0$ and, therefore, $\varphi(0)=0.$ Thus, the function $\varphi$ satisfies the conditions of  Lemma \ref{bk lemma}.

Set $d=2/p,$ select $x\in L_1(0,1)$ as in Lemma \ref{bk lemma} and select a random variable $A_0\in\mathcal{L}_p(\mathcal{R})$ such that $\mu(A_0)=x^{1/p}.$ Recall that $M_p(t)\sim\min\{t^p,t^2\}$ for all $t>0,$ where $M_p$ is the Orlicz function defined by \eqref{lp+l2 orlicz}. we infer from Lemma \ref{bk lemma} that
$$M(t)=t^p\ \varphi(t^{2-p})\sim t^p\tau(\min\{|A_0|^p,t^{2-p}|A_0|^{2}\}) \sim$$
$$\sim \tau(\min\{(t|A_0|)^p,(t|A_0|)^2\})\sim\tau(M_p(t|A_0|)),\quad t\in(0,1).$$
By Lemma \ref{ms} and Lemma \ref{lpl2 embed}, we obtain
$$\|A\|_{\mathcal{L}_M}=\|A\otimes A_0\|_{\mathcal{L}_p+\mathcal{L}_2}\sim\|\mathscr{R}(A\otimes A_0)\|_p$$
for every finite rank operator $A\in\mathcal{L}(H).$ The set of all finite rank operators is dense in every separable ideal $\mathcal{I}, \|\cdot\|_{\mathcal{I}})$, in particular in the ideal $\mathcal{L}_M(H)$. Thus, the ideal $\mathcal{L}_M(H)$ isomorphically embeds into $\mathcal{L}_p(\mathcal{R}).$
\end{proof}

\section{Raynaud-Sch\"utt theorem for $1\leq p<2$}\label{RS section}

In this section, we extend the main result from \cite[Theorem 1.1]{RS}  established  there for $L_1$-spaces to the case of $L_p$-spaces, $1\leq p<2.$ For convenience of the reader, we present a complete version of our argument.

The assertion below extends  \cite[Lemma 2.9]{KwS}. In this and subsequent sections, we denote the set of all $q-$concave  (respectively, that of $p$-convex and $2$-concave) Orlicz functions $M$ such that $M(1)=1$ by $\mathcal{O}_{1,q}$ (respectively, by $\mathcal{O}_{p,2}$).

\begin{lem}\label{first orlicz compute} Let $1\leq q\leq 2$ be fixed. For every $n\geq1,$ there exists $M_n\in\mathcal{O}_{1,q}$ such that
\begin{equation}\label{foc6}\frac14\|x\|_{l_{M_n}}\leq\sum_{k=0}^{n-1}\mu(k,x)+n^{1-1/q}\Big(\sum_{k=n}^{\infty}\mu^q(k,x)\Big)^{1/q}\leq 4\|x\|_{l_{M_n}}\end{equation}
for every finitely supported $x\in l_{\infty}.$
\end{lem}
\begin{proof} Let $x=\{x(k)\}_{k\ge 0}\in l_{\infty}$ be finitely supported and let $n\ge 1$ be fixed. Without loss of generality we may assume that $x\ge 0.$  By $N_n$ we denote a $q$-concave Orlicz function defined by
$$N_n(t):=
\begin{cases}
n^{q-1}t^q,\quad 0\leq t\leq 1/n\\
qt-(q-1)/n,\quad t\geq 1/n.
\end{cases}
$$ Denoting
$$S(x):=\sum_{k=0}^{n-1}\mu(k,x)+n^{1-1/q}\Big(\sum_{k=n}^{\infty}\mu^q(k,x)\Big)^{1/q},$$
we shall show that the Orlicz function $N_n$ satisfies
\begin{equation}\label{foc1}
\frac14\|x\|_{l_{N_n}}\leq S(x)\leq 2\|x\|_{l_{N_n}}.
\end{equation}

{\bf Step 1:} We first show the left hand side of \eqref{foc1}.
Suppose that $S(x)\leq 1.$
Setting $y=\{y(k)\}_{k\ge 0}$ and $z=\{z(k)\}_{k\ge 0},$ where
$$y(k)=\left\{\begin{array}{cl}x(k), & x(k)>1/n \\ 0, & \mbox{otherwise}\end{array}\right. \ \ \mbox{and} \ \
z(k)=\left\{\begin{array}{cl}x(k), & x(k)\le1/n \\ 0, & \mbox{otherwise}\end{array}\right. \ \ \ k\ge 0,$$ we have $\mu(y),\mu(z)\leq\mu(x).$
Since $S(x)\le 1,$ it obviously follows that \begin{equation}\label{foc2}\sum_{k=0}^{n-1}\mu(k,x)\le 1.\end{equation}
Thus, we obtain
$$\mu(n-1,y)\leq\frac1n\sum_{k=0}^{n-1}\mu(k,y)\leq\frac1n\sum_{k=0}^{n-1}\mu(k,x)\leq\frac1n.$$
According to definition of $y,$ we have that $\mu(n-1,y)=0$ and, therefore, $\mu(k,y)=0$ for $k\geq n-1.$
 Hence, employing again \eqref{foc2}, we have
$$\sum_{k\geq 0}N_n(y(k))=\sum_{y(k)>1/n}\Big(qy(k)-\frac{q-1}{n}\Big)\leq\sum_{y(k)>1/n}qy(k)= q\sum_{k=0}^{n-1}\mu(k,y)\leq q,$$
which implies $\|y\|_{l_{N_n}}\leq q.$

On the other hand, by the assumption,  we have
$$n^{1-1/q}\Big(\sum_{k=n}^{\infty}\mu^q(k,z)\Big)^{1/q}\leq S(x)\leq 1$$
and, therefore,
\begin{equation}\label{foc3}\sum_{k=n}^{\infty}\mu^q(k,z)\leq n^{1-q}.\end{equation}
Since by \eqref{foc3} and by definition of $z$
$$\sum_{k\geq 0}z(k)^q=\sum_{k=n}^{\infty}\mu^q(k,z)+\sum_{k=0}^{n-1}\mu^q(k,z)\leq n^{1-q}+n\cdot n^{-q}=2n^{1-q},$$
it follows that
$$\sum_{k\geq 0}N_n(z(k))=\sum_{z_k\leq 1/n}n^{q-1}z(k)^q\leq 2.$$
Hence, $\|z\|_{l_{N_n}}\leq 2.$

Combining the estimates above, we infer that $$\|x\|_{l_{N_n}}\leq\|y\|_{l_{N_n}}+\|z\|_{l_{N_n}}\leq q+2\le 4,$$
which completes the proof of the left hand side of \eqref{foc1}.

\noindent{\bf Step 2:} Now we prove the right hand side of \eqref{foc1}. Suppose that $\|x\|_{l_{N_n}}\leq 1.$ Setting
$$y:=\Big(\underbrace{\frac1n\sum_{k=0}^{n-1}\mu(k,x)}_{\mbox{\small{$n$ times}}},\mu(n,x),\mu(n+1,x),\ldots\Big),$$
we have that $y\prec\prec x$ and, therefore
(see Remark \ref{rem_Orl_mon_submaj}),
 $\|y\|_{l_{N_n}}\leq 1.$ Hence,
\begin{equation}\label{foc4}\sum_{y(k)>1/n}y(k)\leq\sum_{y(k)>1/n}\Big(qy(k)-\frac{q-1}{n}\Big)\leq\sum_{k\geq0}N_n(y(k))\leq 1.\end{equation}
Suppose that $y(0)>1/n.$ Then $y(k)> 1/n$ for $0\leq k\le n-1$ and, therefore, $\sum_{y(k)>1/n}y(k)>1,$ which contradicts with
 \eqref{foc4}. Thus, we conclude that $y(0)\leq 1/n$ and, therefore, $y(k)\leq 1/n$ for all $k\geq 0.$ Combining the latter and the fact that $\|y\|_{l_{N_n}}\leq 1$, we obtain
$$n^{q-1}\sum_{k\geq0}y(k)^q=\sum_{k\geq0}N_n(y(k))\leq 1.$$
Hence, by definition of $y$,
$$S(x)=\sum_{k=0}^{n-1}\mu(k,x)+n^{1-1/q}\Big(\sum_{k=n}^{\infty}\mu^q(k,x)\Big)^{1/q} \leq ny(0)+n^{1-1/q}\Big(\sum_{k\geq 0}y(k)^q\Big)^{1/q}\leq$$
$$\leq n\cdot\frac1n+n^{1-1/q}\cdot (n^{1-q})^{1/q}\leq 2,$$
which completes the proof of the right hand side of \eqref{foc1}.

\noindent{\bf Step 3:} Next we set
$$M_n(t):=\frac{N_n(t)}{N_n(1)}, \ \ \ t\in[0,\infty).$$ Observe that $M_n\in \mathcal{O}_{1,q}.$
Moreover, since $1\le N_n(1)\le q\le 2,$ for all $n\ge 1,$ we have that
$$M_n(t)\le N_n(t)\le 2 \ M_n(t), \ \ \ t\in[0,\infty), \ \ n\ge 1.$$
Using the latter fact, one can infer that
\begin{equation}\label{foc5}\|x\|_{l_{M_n}}\le \|x\|_{l_{N_n}}\leq 2\|x\|_{l_{M_n}}\end{equation}
Combination of \eqref{foc1} and \eqref{foc5} completes the proof of the lemma.
\end{proof}

The assertion below extends  \cite[Lemma 2.10]{KwS}. Recall that $\mathfrak{S}_n$ denotes the group of all permutations of $n$ elements.

\begin{lem}\label{kws lemma} Let $1\leq p\leq 2$ and $n\in\mathbb N.$ For every $y\in\mathbb{R}^n,$ there exists $M_{n,y}\in\mathcal{O}_{p,2}$ such that
\begin{equation}\label{kws2}\frac1{80n}\|y\|_p^p\cdot\|x\|_{l_{M_{n,y}}}^p\leq\frac1{n!}\sum_{\rho\in\mathfrak{S}_n}
\Big(\sum_{k=0}^{n-1}x^2(\rho(k))y^2(k)\Big)^{\frac{p}2}\leq \frac{16}n\|y\|_p^p\cdot\|x\|_{l_{M_{n,y}}}^p\end{equation}
for every $x\in\mathbb{R}^n.$
\end{lem}
\begin{proof} Let $1\leq p\leq 2,$ $q=2/p,$ $n\in\mathbb N$ and $y\in\mathbb{R}^n$ be fixed. For every $x\in\mathbb{R}^n$ we obviously have
\begin{equation*}S_y(x):=\frac1{n!}\sum_{\rho\in\mathfrak{S}_n}\Big(\sum_{k=0}^{n-1}x^2(\rho(k))y^2(k)\Big)^{p/2}=\frac1{n!}\sum_{\rho\in\mathfrak{S}_n}
\Big(\sum_{k=0}^{n-1}(x^p(\rho(k))y^p(k))^q\Big)^{1/q}.\end{equation*}
Applying \cite[Theorem 1.2]{KwS} to the right hand side of the equality above, we obtain
\begin{multline}\label{kws1}\frac1{5n}\Big(\sum_{k=0}^{n-1}\mu(k,x^p\otimes y^p)+n^{1-\frac1q}\Big(\sum_{k=n}^{\infty}\mu^q(k,x^p\otimes y^p)\Big)^{\frac1{q}}\Big)\le S_y(x)\\
\leq\frac1n\Big(\sum_{k=0}^{n-1}\mu(k,x^p\otimes y^p)+n^{1-\frac1{q}}\Big(\sum_{k=n}^{\infty}\mu^q(k,x^p\otimes y^p)\Big)^{\frac1{q}}\Big).\end{multline}
By Lemma \ref{first orlicz compute} there exists the Orlicz function $M_n\in\mathcal{O}_{1,q}$ that satisfies \eqref{foc6} for any finitely supported vector. Thus, applying \eqref{foc6} for the vector $x^p\otimes y^p$ to
 the right and left hand sides of \eqref{kws1}, we infer
\begin{equation}\label{kws3}\frac1{20n}\|x^p\otimes y^p\|_{l_{M_n}}\leq S_y(x)\leq\frac4n\|x^p\otimes y^p\|_{l_{M_n}}.\end{equation}

Define the function $N_{n,y}$ by setting
$$N_{n,y}(t):=\sum_{k\geq0}M_n\Big(\frac{ty^p(k)}{\|y^p\|_{l_{M_n}}}\Big),  \ \ \ t\in[0,\infty).$$
Observe that $$N_{n,y}(1)=\sum_{k\geq0}M_n\Big(\frac{y^p(k)}{\|y^p\|_{l_{M_n}}}\Big)=1$$ and  $q$-concavity of the function $N_{n,y}$ follows from the $q$-concavity of $M_n$. It follows that $N_{n,y}\in\mathcal{O}_{1,q}.$
Moreover, by definition of the Orlicz norm, we have
\begin{align}\label{kws4}\|x^p\otimes y^p\|_{l_{M_n}} &=\inf\Big\{\lambda:\ \sum_{k,l\geq0}M_n\Big(\frac{x^p(l)y^p(k)}{\lambda}\Big)\leq 1\Big\}\\
&=\inf\Big\{\lambda:\ \sum_{l\geq0}N_{n,y}\Big(\frac{\|y^p\|_{l_{M_n}}x^p(l)}{\lambda}\Big)\leq 1\Big\}=\|y^p\|_{l_{M_n}}\|x^p\|_{l_{N_{n,y}}}. \nonumber \end{align}
Combining \eqref{kws3} and \eqref{kws4}, we obtain
\begin{equation}\label{kws5}\frac1{20n}\|x^p\|_{l_{N_{n,y}}}\|y^p\|_{l_{M_n}}\leq S_y(x)\leq\frac4n\|x^p\|_{l_{N_{n,y}}}\|y^p\|_{l_{M_n}}.\end{equation}
Using Lemma \ref{first orlicz compute} for the sequence $y\in\mathbb R^n$, we infer
\begin{equation}\label{kws6}\frac14\|y\|_p^p\leq\|y^p\|_{l_{M_n}}\leq 4\|y\|_p^p.\end{equation}
Setting $M_{n,y}(t):=N_{n,y}(t^p),$ $t\ge0,$ we obtain that
$\|x^p\|_{l_{N_{n,y}}}=\|x\|_{l_{M_{n,y}}}^p.$  Application of the latter and \eqref{kws6} in \eqref{kws5} completes the proof of \eqref{kws2}.
Since $N_{n,y}\in\mathcal{O}_{1,q},$ it follows that $M_{n,y}\in\mathcal{O}_{p,2}.$
\end{proof}

\begin{lem} Let $1\leq p\leq 2$. The set $\mathcal{O}_{p,2}$ equipped with the topology of pointwise convergence is a compact metrizable space.
\end{lem}
\begin{proof} Using that the set of rational numbers $\mathbb{Q}$ is countable, we numerate $\mathbb{Q}_+$ as $\mathbb{Q}_+=\{a_1,a_2,\cdots\}.$ We claim that the metric
$$d(M_1,M_2):=\sum_{n\geq 1}2^{-n}\frac{|(M_1-M_2)(a_n)|}{1+|(M_1-M_2)(a_n)|}, \ \ \ M_1,M_2\in \mathcal{O}_{p,2},$$
 generates the topology of pointwise convergence on $\mathcal{O}_{p,2}$.
Proof of the completeness is routine and, therefore, omitted.

Let us show compactness.
%Observe first that that $2$-concavity of $M$ implies $M(t)\leq\max\{t,t^2\}$ for all $t\le 0.$
Take an arbitrary sequence $\{M_n\}_{n\geq1}\subset\mathcal{O}_{p,2}.$ Since $M_n$ is $2$-concave for any $n\ge 1,$ it follows that the sequence $\{M_n(a_1)\}_{n\geq1}$ is bounded by $\max\{1,a_1^2\}$. Let $m_1$ be a limit of some subsequence of $\{M_{n,1}(a_1)\}_{n\geq1}.$ Select a subsequence $\{M_{n,1}\}_{n\geq1}\subset\{M_n\}_{n\geq1}$ such that $|M_{n,1}(a_1)-m_1|\leq 1/n$ for all $n\geq1.$ Using the same argument, we have that the sequence $\{M_{n,1}(a_2)\}_{n\geq1}$ is bounded. Let $m_2$ be a limit of some subsequence of $\{M_{n,1}(a_1)\}_{n\geq1}.$ Select a subsequence $\{M_{n,2}\}_{n\geq1}\subset\{M_{n,1}\}_{n\geq1}$ such that $|M_{n,2}(a_2)-m_2|\leq 1/n$ for all $n\geq1.$ Also, we have $|M_{n,2}(a_1)-m_1|\leq1/n$ for all $n\geq1.$ Proceeding this process ad infinitum, for every $l\geq 1,$ we obtain
 a sequence $\{M_{n,l}\}_{n\geq1}\subset\mathcal{O}_{p,2}$ such that $|M_{n,l}(a_k)-m_k|\leq 1/n$ for all $n\geq1$ and for all $k\leq l.$ Taking a diagonal subsequence $\{M_{n,n}\}_{n\geq1}\subset\{M_{n,l}\}_{n\geq1,l\geq 1},$ we have $|M_{n,n}(a_k)-m_k|\leq 1/n$ for all $1\le k\le n$.  Hence, the subsequence $\{M_{n,n}\}_{n\geq1}$ converges at every rational point.

Next we prove that $\{M_{n,n}\}_{n\geq1}$ converges at every $t>0.$ Fix $\varepsilon>0$ and select rational number $a\in(t,t(1+\varepsilon)).$ It follows from convexity and $2$-concavity of $M_{n,n}$ that
$$M_{n,n}(t)\leq M_{n,n}(a)\leq (1+\varepsilon)^2M_{n,n}(t), \ \ \ n\ge 1.$$
Passing to the limit, we obtain
$$\limsup_{n\to\infty}M_{n,n}(t)\leq\lim_{n\to\infty}M_{n,n}(a)\leq(1+\varepsilon)^2\liminf_{n\to\infty}M_{n,n}(t).$$
Since $\varepsilon$ is arbitrary and the sequence $\{M_{n,n}(a)\}_{n\ge 1}$ converges, it follows that
$$\limsup_{n\to\infty}M_{n,n}(t)=\liminf_{n\to\infty}M_{n,n}(t),$$ which
guarantees that the sequence $M_{n,n}(t)$ converges. Thus, compactness is proved.
\end{proof}

The following assertion is proved in \cite[Theorem 1.1]{RS} for $p=1.$ We provide a proof (the same as in \cite[Theorem 1.1]{RS}) for convenience of the reader. Due to previous lemma, we equip $\mathcal{O}_{p,2}$ with Borel $\sigma$-algebra.

\begin{thm}\label{rs theorem} Let $1\leq p\leq 2.$ If a symmetric sequence space $I$ isomorphically embeds into $L_p(0,1),$ then there exists a Radon probability measure $\nu$ on the space $\mathcal{O}_{p,2}$ such that
\begin{equation}\label{rs1}\|x\|_I\sim\Big(\int_{M\in\mathcal{O}_{p,2}}\|x\|_{l_M}^pd\nu(M)\Big)^{1/p}\end{equation}
for every finitely supported $x\in l_{\infty}.$
\end{thm}
\begin{proof} Let  $x=\{x(k)\}_{k\ge 0}\in l_{\infty}$ be finitely supported, let $\{e_k\}_{k\geq1}$ be the standard unit basis of $c_0$ and let $T$ be the isomorphic embedding of $I$ into $L_p(0,1).$ Fix a constant $C$ such that $$C^{-1}\|x\|_I\leq\|Tx\|_p\leq C\|x\|_I, \ \ \mbox{for any} \ \ x\in I.$$
Taking any $n\geq|{\rm supp}(x)|,$ we may assume
that $x=\sum_{k=0}^{n-1}x(k)e_k.$ Denoting $f_k:=Te_k\in L_p(0,1),$ $k\ge 0,$ for every $t\in(0,1)$ and for every permutation $\rho\in \mathfrak{S}_n,$
we have
$$\|x\|_I=\Big\|\sum_{k=0}^{n-1}x(\rho(k))r_k(t)e_k\Big\|_I\leq C\Big\|\sum_{k=0}^{n-1}x(\rho(k))r_k(t)f_k\Big\|_p$$
and, similarly,
$$\|x\|_I\geq C^{-1}\Big\|\sum_{k=0}^{n-1}x(\rho(k))r_k(t)f_k\Big\|_p,$$
where by $r_k,$ $k\ge 0,$ we denote the Rademacher functions
defined in \eqref{eq_Rad}. Hence,
$$\|x\|_I\leq C\Big(\int_0^1\Big\|\sum_{k=0}^{n-1}x(\rho(k))r_k(t)f_k\Big\|_p^pdt\Big)^{\frac1p}=
C\Big\|\Big(\int_0^1\Big|\sum_{k=0}^{n-1}x(\rho(k))r_k(t)f_k\Big|^pdt\Big)^{\frac1p}\Big\|_p$$
and, similarly
$$\|x\|_I\geq C^{-1}\Big\|\Big(\int_0^1\Big|\sum_{k=0}^{n-1}x(\rho(k))r_k(t)f_k\Big|^pdt\Big)^{\frac1p}\Big\|_p$$
for every permutation $\rho\in \mathfrak{S}_n.$

By Khinchine inequality (see  \cite[Theorem 2.b.3]{LT1} and its proof), we have
$$\frac1{36C}\Big\|\Big(\sum_{k=0}^{n-1}x^2(\rho(k))f_k^2\Big)^{\frac12}\Big\|_p\leq \|x\|_I\leq C\Big\|\Big(\sum_{k=0}^{n-1}x^2(\rho(k))f_k^2\Big)^{\frac12}\Big\|_p$$
for every permutation $\rho\in \mathfrak{S}_n.$ Summing by all
permutations from $\mathfrak{S}_n$ the inequality above, we obtain
\begin{align}\label{rs2}\|x\|_I\leq C\Big(\frac1{n!}\sum_{\rho\in\mathfrak{S}_n}\Big\|\Big(\sum_{k=0}^{n-1}x^2(\rho(k))f_k^2\Big)^{\frac12}\Big\|_p^p\Big)^{\frac1{p}}
=C\Big\|\Big(\frac1{n!}\sum_{\rho\in\mathfrak{S}_n}\Big(\sum_{k=0}^{n-1}x^2(\rho(k))f_k^2\Big)^{\frac{p}2}\Big)^{\frac1{p}}\Big\|_p
\end{align}
and, similarly, 
 \begin{align}\label{rs3}
\|x\|_I\geq\frac1{36C}\Big\|\Big(\frac1{n!}\sum_{\rho\in\mathfrak{S}_n}\Big(\sum_{k=0}^{n-1}x^2(\rho(k))f_k^2\Big)^{\frac{p}2}\Big)^{\frac1p}\Big\|_p.
\end{align}
For a fixed $s\in(0,1)$, by Lemma \ref{kws lemma} applied for
$y_s=\{f_k(s)\}_{k=0}^{n-1}\in\mathbb R^n,$ there exists the
Orlicz function $M_{n,y_s}\in\mathcal{O}_{p,2}$ such that
\begin{equation}\label{rs4}\frac1{80n}\|y_s\|_p^p\cdot\|x\|_{l_{M_{n,y_s}}}^p\leq\frac1{n!}\sum_{\rho\in\mathfrak{S}_n}
\Big(\sum_{k=0}^{n-1}x^2(\rho(k))f_k^2(s)\Big)^{\frac{p}2}\leq
\frac{16}n\|y_s\|_p^p\cdot\|x\|_{l_{M_{n,y_s}}}^p.\end{equation}
Combining the computation
$$\Big\|\Big(\frac1{n}\|y_s\|_p^p\cdot\|x\|_{l_{M_{n,y_s}}}^p\Big)^{\frac1p}\Big\|_p=
\Big(\int_0^1\frac1n\Big(\sum_{k=0}^{n-1}|f_k(s)|^p\Big)\|x\|_{l_{M_{n,y_s}}}^pds\Big)^{\frac1p}$$
with \eqref{rs2},\eqref{rs3} and \eqref{rs4}, we arrive at
 \begin{equation}\label{rs5}\|x\|_I\sim\Big(\int_0^1\frac1n\Big(\sum_{k=0}^{n-1}|f_k(s)|^p\Big)\|x\|_{l_{M_{n,y_s}}}^pds\Big)^{\frac1p}=
\Big(\int_0^1\|x\|_{M_{n,y_s}}^pd\mu_n(s)\Big)^{\frac1p},\end{equation}
where
$$\mu_n(A):=\frac1n\sum_{k=0}^{n-1}\int_A|f_k(s)|^pds,\quad A\subseteq (0,1).$$
Define a  positive Radon measure $\nu_n$ on
$\mathcal{O}_{p,2}$ equipped with Borel $\sigma$-algebra by
setting
$$\nu_n(G):=\mu_n(\{s\in(0,1):\ M_{n,y_s}\in G\}),\quad G\subseteq\mathcal{O}_{p,2}.$$
Via \eqref{rs5}, we conclude
$$\|x\|_I\sim\Big(\int_{M\in\mathcal{O}_{p,2}}\|x\|_{l_M}^pd\nu_n(M)\Big)^{1/p}$$
for every finitely supported $x\in l_{\infty}$ and for every $n\geq|{\rm supp}(x)|.$

For every finitely supported $x\in l_{\infty},$ the function
$M\to\|x\|_{l_M}$ is continuous on $\mathcal{O}_{p,2}$ equipped
with the topology of pointwise convergence. Therefore, every Radon
measure $\nu_n,$ $n\geq|{\rm supp}(x)|,$ generates a bounded
positive functional $\varphi_n\in C(\mathcal{O}_{p,2})^*,$ where
$C(\mathcal{O}_{p,2})^*$ denotes the dual space of the space
$C(\mathcal{O}_{p,2})$ of all continuous functions on
$\mathcal{O}_{p,2}.$ By Banach-Alaoglu theorem, there exists a
subnet $\varphi_{\psi(i)},$ $i\in\mathbb{I},$ and $\varphi\in
C(\mathcal{O}_{p,2})^*$ such that $\varphi_{\psi(i)}\to \varphi.$
By Riesz-Markov theorem, the (positive) functional $\varphi$ is
generated by a Radon measure $\nu$ on $\mathcal{O}_{p,2}.$ Hence,
we have
$$\|x\|_I\sim\Big(\int_{M\in\mathcal{O}_{p,2}}\|x\|_{l_M}^pd\nu(M)\Big)^{1/p}$$
for every finitely supported $x\in l_{\infty}.$
\end{proof}

\section{Proof of Theorem \ref{main theorem}}

In this section, we prove our main result.
Recall, that by Theorem \ref{rs theorem} the space $\mathcal{O}_{p,2}$ is equipped with a Radon probability measure $\nu$, defined on the Borel $\sigma$-algebra.

\begin{lem}\label{bochner} There exists a Bochner measurable function $\xi:\mathcal{O}_{p,2}\to\mathcal{L}_p(\mathcal{R}), 1\leq p<2$ such that
\begin{equation}\label{crucial estimate}
\frac1{12}\|A\|_{\mathcal{L}_M}\leq \|\xi(M)\otimes A\|_{  \mathcal{L}_p+\mathcal{L}_2(\mathcal{R}\overline{\otimes}\mathcal{L}(H))}\leq5\|A\|_{\mathcal{L}_M}
\end{equation}
for every finite rank operator $A\in\mathcal{L}(H)$ and for every $M\in\mathcal{O}_{p,2}$ satisfying the condition
\begin{equation}\label{assu}
\lim_{t\to0}\frac1{t^p}M(t)=0.
\end{equation}
\end{lem}
\begin{proof}
Let $M\in \mathcal{O}_{p,2}$.  Define an Orlicz function $N_M$ by setting
\begin{equation}\label{lem_N_M}
N_M(t)=
\begin{cases}
M(t),\quad 0\leq t\leq 1\\
1+\frac1pM'(1-0)(t^p-1),\quad t\geq 1.
\end{cases}
\end{equation}
It is straightforward to verify that $N_M\in\mathcal{O}_{p,2}.$
Furthermore, if in addition  $M$ satisfies \eqref{assu}, we define a quasi-concave function $\varphi_M:\mathbb{R}_+\to\mathbb{R}_+$ by setting
\begin{equation}\label{lem_phi}
\varphi_M(t)=\frac{N_M(t^{\frac{1}{2-p}})}{t^{\frac{p}{2-p}}}.
\end{equation}
By Lemma  \ref{bk lemma} with $d=2/p$ there exists a positive decreasing function $x_M\in L_1(0,1)$ given by formula \eqref{x main def}.

We claim that the mapping $\eta:\mathcal{O}_{p,2}\to L_1(0,1)$ defined by setting $\eta:M\mapsto x_M$ is Bochner measurable (if $M$ fails the condition \eqref{assu}, then we set $x_M=0$).

Let $t_k(M),s_k(M)$ and $u_{k}(M), k\geq 0, $ be the numbers constructed for the function $\varphi_M$ in the proof of Lemma \ref{bk lemma}.
 First note that  $t_k(M)$ and $s_k(M),$ $k\geq0,$ considered as functions of $M$,  are measurable on $\mathcal{O}_{p,2}$. Indeed, since the function $\varphi_M$ is quasi-concave, for every $s\in(0,1),$ by definition of numbers $t_k$ we have that
$$\{M:\ t_k(M)\geq s\}=\{M:\ \varphi_M(s)\leq 2^{-k}\}=\{M:\ M(s^{\frac{1}{2-p}})\leq s^{\frac{p}{2-p}}2^{-k}\}$$
is a closed set in $\mathcal{O}_{p,2}$ (equipped with topology of poitwise convergence) and therefore, the functions $t_k(\cdot)$ are measurable, and similarly $s_k(\cdot)$ are measurable for all $k\geq 0$. By definition of the numbers  $u_k(M),$ $k\geq0,$ we have that the functions $u_k(\cdot)$ are also measurable functions of $M.$
Furthermore, since $\varphi_M(u_k(M)),k\geq 0$ takes only the values $2^{-n},$ $n\geq0,$ and
$$\{M:\ \varphi_M(u_{2k}(M))=2^{-n}\}=\{M:\ u_{2k}(M)=t_n(M)\}$$
is a measurable set, we have that the function $M\mapsto\varphi_M(u_{2k}(M)),$ $k\geq0$ is also measurable.
 Set
$$v_k(M)=\sum_{l\geq k}\frac12u_{2l}^{1/(d-1)}(M)\varphi_M(u_{2l}(M)),$$
where the convergence of the series follows from \eqref{bk lemma1}. We have that $v_k(\cdot)$ also measurable functions of $M.$

Fix a positive element $y\in L_{\infty}(0,1)$ and denote its primitive (which is monotone) by $z.$ It follows from \eqref{x main def} that
$$\langle x_M,y\rangle=\int_0^1x_M(s)y(s)ds= \sum_{k\geq0}u_{2k}^{-1/(d-1)}(M)\Big(z(v_k(M))-z(v_{k+1}(M))\Big),$$
where the series converges since 
$$\left|u_{2k}^{-1/(d-1)}(M)\Big(z(v_k(M))-z(v_{k+1}(M))\Big)\right|\leq$$
$$\leq u_{2k}^{-1/(d-1)}(M)\|y\|_\infty|v_k(M)-v_{k+1}(M)|=\|y\|_\infty\varphi_M(u_{2k}(M))$$
and $\varphi_M(u_{2k}(M))$ is a geometric sequence.
Since $z$ is monotone, it follows that $z\circ v_k$ is a measurable function of $M$ for every $k\geq0.$ Hence, the mapping $M\to\langle x_M,y\rangle$ is also measurable. This proves that the function $\eta:M\to x_M\in L_1(0,1),$ is weakly measurable. Since the space $L_1(0,1)$ is separable, by Pettis theorem (see e.g. \cite[Proposition 2.15]{Ryan}), we have that $\eta$ is Bochner measurable.

Let $i:L_1(0,1)\to\mathcal{L}_1(\mathcal{R})$ be an isometric embedding. It is clear, that the function $\xi:\mathcal{O}_{p,2}\to\mathcal{L}_p(\mathcal{R})$ defined by the setting $\xi(M)=(i(\eta(M)))^{1/p}$ is Bochner measurable.

Now, let us show that the function $\xi$ satisfies inequality \eqref{crucial estimate}.
Fix a function $M\in\mathcal{O}_{p,2}$ satisfying condition \eqref{assu}. By definition of the function $N_M$ (see \eqref{lem_N_M}) we have that $M(1)=N_M(1)=1$, and therefore  for every finitely supported $x\in l_{\infty}$ we obtain

\begin{align*}&\{\lambda:\ \sum_{k\geq0}M(\frac{x(k)}{\lambda})\leq 1\}=\{\lambda\geq\|x\|_{\infty}:\ \sum_{k\geq0}M(\frac{x(k)}{\lambda})\leq 1\} \\ &=\{\lambda\geq\|x\|_{\infty}:\ \sum_{k\geq0}N_M(\frac{x(k)}{\lambda})\leq 1\}=\{\lambda:\ \sum_{k\geq0}N_M(\frac{x(k)}{\lambda})\leq 1\}.
\end{align*}
Consequently \begin{equation}\label{lem_norms}
\|x\|_{l_{N_M}}=\|x\|_{l_M}
\end{equation} for every finitely supported $x\in l_{\infty}.$ By Lemma \ref{bk lemma} we have that
\begin{equation}\label{lem_two_sided}
\frac14\varphi_M(t)\leq\int_0^1\min\{x_M(s),tx_M^d(s)\}ds\leq\frac52\varphi_M(t),\quad t\in[0,1],
\end{equation}
in particular for $t=1$ we obtain $\frac14\leq\|x_M\|_1$. The
$2$-concavity of the function $M$ implies that $M(2)\leq 4M(1)$ (see e.g. \cite{KPS}) Hence, for the function $\varphi_M$ given by \eqref{lem_phi} we have
$$\varphi_M(\infty)=\frac1p M'(1-0)\leq M(2)-M(1)\leq 3M(1)=3.$$
Therefore, for $t\geq 1$ we infer that
$$\int_0^1\min\{x_M(s),tx_M^d(s)\}ds=\int_0^1x_M(s)ds\geq\frac14\geq\frac1{12}\varphi_M(\infty)\geq\frac1{12}\varphi_M(t).$$
On the other hand, definition of the function $x_M$ (see \eqref{x main def}) implies that $x_M\geq \mathrm{supp}x_M$. In addition by Lemma \ref{bk lemma} $\|x_M\|_1\leq 1$. Consequently, we have
$$\int_0^1\min\{x_M(s),tx_M^d(s)\}ds=\int_0^1x_M(s)ds\leq1=\varphi_M(1)\leq\frac52\varphi_M(t),t\geq 1.$$
Thus, combining the last two inequalities with \eqref{lem_two_sided} we obtain that
\begin{equation}\label{lem_all_t}
\frac1{12}\varphi_M(t)\leq\int_0^1\min\{x_M(s),tx_M^d(s)\}ds\leq\frac52\varphi_M(t)
\end{equation}
for all $t\geq 0$.

Next, for the function $M_p(t)$ defined by \eqref{lp+l2 orlicz} we have
$$\min\{t^2,t^p\}\leq M_p(t)\leq \frac2p\min\{t^2,t^p\}, t\geq 0.$$
Consequently, from the one hand for all $t\geq 0$ we have
\begin{gather*}
\begin{split}
\int_0^1M_p(tx_M^{1/p}(s))ds&\leq\frac{2}{p}\int_0^1\min\{t^2x_M^{2/p}(s), t^px_M(s)\}ds\\
&\leq2t^p\int_0^1\min\{t^{2-p}x_M^d(s),x_M(s)\}ds\stackrel{\eqref{lem_all_t}}{\leq}5 t^p\varphi_M(t^{2-p})\\
&=5N_M(t) \leq N_M(5t).
\end{split}
\end{gather*}
And similarly, from another hand
\begin{gather*}
\begin{split}
\int_0^1M_p(tx_M^{1/p}(s))ds&\geq\int_0^1\min\{t^{2-p}x_M^d(s),x_M(s)\}ds\stackrel{\eqref{lem_all_t}}{\geq}\frac{1}{12}t^p\varphi_M(t^{2-p})\\
&=\frac{1}{12}N_M(t) \geq N_M(\frac{1}{12}t),
\end{split}
\end{gather*}
That is for every finitely supported $x\in l_\infty$
$$\frac1{12}\|x\|_{L_{N_M}}\leq \|x\|_{L_{M'}}\leq 5\|x\|_{L_{N_M}},$$
 where the Orlicz function $M'$ is defined by the equality $M'(t)=\int_0^1M_p(tx_M^{1/p}(s))ds$.
By Lemma \ref{ms} we have that $\|x_M^{1/p}\otimes x\|_{L_p+L_2}=\|x\|_{L_{M'}}.$ Consequently, due to \eqref{lem_norms} we obtain  that
\begin{equation*}
\frac1{12}\|x\|_{l_M}\leq\|x_M^{1/p}\otimes x\|_{L_p+L_2}\leq5\|x\|_{l_M}
\end{equation*}
for every finitely supported $x\in l_{\infty}.$
The last inequality immediately implies \eqref{crucial estimate} for $\xi(M)$ defined by $\xi(M)=(i(\eta(M)))^{1/p}=(i(x_M))^{1/p}$.
\end{proof}

We are now ready to prove the main result of the paper.

\begin{proof}[Proof of Theorem \ref{main theorem}] Let us use means provided by Theorem \ref{rs theorem} to represent the norm in $I$. That is, we write
\begin{equation}\label{th_I-sim}
\|x\|_I\sim(\int_{M\in\mathcal{O}_{p,2}}\|x\|_{l_M}^pd\nu(M))^{1/p}
\end{equation}
for every finitely supported $x\in l_{\infty}.$

Let us show firstly that condition \eqref{assu} holds in fact for almost every $M\in\mathcal{O}_{p,2}$.
For every $M\in\mathcal{O}_{p,2},$ it follows from $p-$convexity of $M$ that $M(t)\leq t^p$ for $t\in[0,1].$ Similarly,  $2-$concavity of $M$ implies that $M(t)\leq t^2$ for $t\geq 1.$ Denote the Orlicz function $t\to\max\{t^2,t^p\}$ by $F.$ We have $M\leq F$ and, therefore, $\|x\|_{l_M}\leq\|x\|_{l_F}$ for every finitely supported $x\in l_{\infty}.$ However, for Orlicz spaces $l_F$ and $l_p$ we have $l_F=l_p$ and $\|\cdot\|_{F}\sim\|\cdot\|_p$. Hence, we have
$$\|x\|_I\stackrel{\eqref{th_I-sim}}{\leq} \mathrm{const}\|x\|_p$$ for every finitely supported $x\in l_{\infty}$.

Define measurable functions $\alpha_t:\mathcal{O}_{p,2}\to\mathbb{R}$ by setting
$$\alpha_t(M):=\frac1{t^p}M(t), t\in(0,1).$$ Since $M(t)\leq t^p$, the limit $\alpha(M):=\lim_{t\rightarrow 0}\alpha_t(M)$ exists. Observe that $\alpha(M)$  is a measurable function as a pointwise limit of measurable functions. Again using  $p-$convexity of $M\in\mathcal{O}_{p,2}$ we infer that $M(t)\geq \alpha(M)t^p$ for all $t\geq0.$ Therefore, we have $\|x\|_{l_M}\geq \alpha(M)^{1/p}\|x\|_p.$ If $\nu(\{M: \alpha(M)>0\})>0,$ then again using \eqref{th_I-sim} we have $\|x\|_I\geq \mathrm{const} \|x\|_p$ for every finitely supported $x\in l_{\infty}.$ Since we already established that $\|x\|_I\leq\mathrm{const}\|x\|_p,$ it follows that $l_p=I,$ that contradicts to the assumption $\mathcal{I}\neq\mathcal{L}_p(H)$. Hence, $\nu(\{M: \alpha(M)>0\})=0,$ and therefore  equality \eqref{assu} holds for almost every $M.$

Let $\xi:\mathcal{O}_{p,2}\to\mathcal{L}_p(\mathcal{R}),$ be the Bochner measurable function constructed in Lemma \ref{bochner}. By the preceding argument inequality \eqref{crucial estimate} holds for almost every $M$ in $\mathcal{O}_{p,2}$. Therefore, combining \eqref{th_I-sim} and \eqref{crucial estimate} we obtain that
$$\|A\|_{\mathcal{I}}\sim \left(\int_{M\in\mathcal{O}_{p,2}}\|\xi(M)\otimes A\|_{\mathcal{L}_p+\mathcal{L}_2}^pd\nu(M)\right)^{1/p}$$
for every finite rank operator $A\in\mathcal{L}(H).$ Furthermore, Lemma \ref{lpl2 embed} implies that
\begin{equation}\label{kmain1}
\|A\|_{\mathcal{I}}\sim \left(\int_{M\in\mathcal{O}_{p,2}}\|\mathscr{K}( \xi(M)\otimes A\otimes r)\|_p^pd\nu(M)\right)^{1/p}
\end{equation}
for every finite rank operator $A\in\mathcal{L}(H).$

Since every finite rank operator is contained in $\mathcal{L}_p(H)$ we have that $B\otimes A\in(\mathcal{L}_p+\mathcal{L}_2)(\mathcal{R}\overline{\otimes}\mathcal{L}(H))$ for every $B\in\mathcal{L}_p(\mathcal{R})$. Hence by Lemma  \ref{lpl2 embed}, for a fixed finite rank operator $A,$ the mapping $B\to\mathscr{K}( B\otimes A\otimes r)$ is continuous on $\mathcal{L}_p(\mathcal{R}).$  Since $\xi:\mathcal{O}_{p,2}\to\mathcal{L}_p(\mathcal{R})$ is Bochner measurable, we infer that the function $M\to\mathscr{K}(\xi(M)\otimes A\otimes r)$ from $\mathcal{O}_{p,2}$ into $\mathcal{L}_p(\mathcal{R}),$ is also Bochner measurable. Define an operator $\mathscr{J}$ from the ideal of finite rank operators into the set of Bochner measurable $\mathcal{L}_p(\mathcal{R})-$valued functions on $\mathcal{O}_{p,2}$ by setting
$$(\mathscr{J}(A))(M)=\mathscr{K}(\xi(M)\otimes A\otimes r),\quad M\in\mathcal{O}_{p,2}.$$

It is proved in \cite{Hal_Neu} (see Theorem 2 there and the preceding paragraph) that every metrizable compact equipped with atomless Radon probability measure is isomorphic to a unit interval equipped with Lebesgue measure. Hence, every metrizable compact equipped with (not necessarily atomless) Radon probability measure is a standard probability space in the sense of Rokhlin. That is, (up to an isomorphism) it is an interval equipped with Lebesgue measure plus at most countably many atoms. Therefore, by Theorem \ref{rs theorem} $(\mathcal{O}_{p,2},\nu)$ is (isomorphic to) an interval equipped with Lebesgue measure plus at most countably many atoms.

It follows from \cite[Lemma 6.2]{BGM} that Bochner space $L_p(\mathcal{O}_{p,2},\nu,\mathcal{L}_p(\mathcal{R}))$ is naturally isomorphic to the space $\mathcal{L}_p(\mathcal{R}\overline{\otimes }L_{\infty}(\mathcal{O}_{p,2},\nu),\tau\otimes d\nu)$ and
\begin{equation}\label{kmain2}
\left(\int_{M\in\mathcal{O}_{p,2}}\|\mathscr{K}(\xi(M)\otimes A\otimes r)\|_p^pd\nu(M)\right)^{1/p}=\|\mathscr{J}(A)\|_{\mathcal{L}_p(\mathcal{R}\otimes L_{\infty}(\mathcal{O}_{p,2},\nu))}.
\end{equation}

Combining \eqref{kmain1} and \eqref{kmain2}, we infer that $\|A\|_{\mathcal{I}}\sim\|\mathscr{J}(A)\|_p$ for every finite rank operator $A.$ Since $\mathcal{I}$ is separable, the set of all finite rank operators is dense in $\mathcal{I}$, and therefore, $\mathcal{I}$ isomorphically embeds into $\mathcal{L}_p(\mathcal{R}\otimes L_{\infty}(\mathcal{O}_{p,2},\nu))$.

Since there exists a trace preserving $*-$homomorphism from $L_{\infty}(\mathcal{O}_{p,2},\nu)$ into $L_{\infty}(0,1)$ and a trace preserving $*-$homomorphism from $L_{\infty}(0,1)$ into $\mathcal{R}$, we have that there exists a trace preserving $*-$homomorphism from $\mathcal{R}\otimes L_{\infty}(\mathcal{O}_{p,2},\nu)$ into $\mathcal{R}\otimes\mathcal{R}.$ The latter von Neumann algebra is $*-$isomorphic to $\mathcal{R}.$ Thus, $\mathcal{L}_p(\mathcal{R}\otimes L_{\infty}(\mathcal{O}_{p,2},\nu))$ isomorphically embeds into $\mathcal{L}_p(\mathcal{R}).$ Consequently, since $\mathcal{I}$ is isomorphically embeds into $\mathcal{L}_p(\mathcal{R}\otimes L_{\infty}(\mathcal{O}_{p,2},\nu))$ we infer  that $\mathcal{I}$ isomorphically embeds into $\mathcal{L}_p(\mathcal{R}).$
\end{proof}

\section{Proof of Theorem \ref{second main theorem}}

Throughout this section, let $\mathcal{N}$ be a finite von Neumann algebra equipped with a faithful normal tracial state $\tau.$

\subsection{Definitions of auxiliary algebras and homomorphisms}  Here we define finite von Neumann algebras $(\mathcal{N}_m,\tau_m)$ by means of the GNS construction of the algebra $l_{\infty}(\mathbb{Z}_+^m)\otimes\mathcal{N}$ with respect to a certain tracial state $\tau_m.$ After that, we discuss certain properties of trace preserving $*-$isomorphisms $\varrho_k,$ $0\leq k<m.$

 Fix a free ultrafilter $\omega$ on $\mathbb{Z}_+$ and denote by $\mathcal{N}_m$ the von Neumann algebra obtained via the GNS construction from the algebra $l_{\infty}(\mathbb{Z}_+^m)\otimes\mathcal{N}$ (identified with the algebra of all bounded $\mathcal{N}-$valued functions on $\mathbb{Z}_+^m$) and the tracial state
\footnote{Here, $\lim_{n\to\omega}a(n)$ denotes the limit of a sequence $a=\{a(n)\}_{n\geq0}$ along the ultrafilter $\omega.$}
$$\tau_m:x\to\lim\limits_{n_0\to\omega}\cdots\lim\limits_{n_{m-1}\to\omega}\tau(x(n_0,\cdots,n_{m-1})),\quad x\in l_{\infty}(\mathbb{Z}_+^m)\otimes\mathcal{N}.$$
That is, $\mathcal{N}_m$ is the quotient of the algebra $l_{\infty}(\mathbb{Z}_+^m)\otimes\mathcal{N}$ by the ideal
$$\Big\{x\in l_{\infty}(\mathbb{Z}_+^m)\otimes\mathcal{N}:\ \lim\limits_{n_0\to\omega}\cdots\lim\limits_{n_{m-1}\to\omega}\tau(|x^2|(n_0,\cdots,n_{m-1}))=0\Big\}.$$
The state $\tau_m$ generates a faithful normal tracial state on $\mathcal{N}_m$ (also denoted by $\tau_m$).

% Consider now a unital $*-$homomorphism $i_1:\mathcal{N}\to \mathcal{N}_m$ defined by setting
% $$(i_1(x))(n_0,\cdots,n_{m-1})=x,\quad (n_0,\cdots,n_{m-1})\in\mathbb{Z}_+^m.$$
% It is clear that
% $$\tau_m(i_1(x))=\lim\limits_{n_0\to\omega}\cdots\lim\limits_{n_{m-1}\to\omega}\tau(x)=\tau(x),\quad x\in\mathcal{N},$$
% so that $i_1$ is trace preserving. To lighten the notation, we identify $\mathcal{N}$ with the subalgebra $i_1(\mathcal{N})$ in $\mathcal{N}_m$ consisting of all constant $\mathcal{N}-$valued functions on $\mathbb{Z}_+^m$ and write $x$ instead of $i_1(x).$
%
% For $k<m,$ consider a unital $*-$homomorphism $i_2:\mathcal{N}_k\to\mathcal{N}_m$ defined by setting
% $$(i_2(x))(n_0,\cdots,n_{m-1})=x(n_0,\cdots,n_{k-1}),\quad (n_0,\cdots,n_{m-1})\in\mathbb{Z}_+^m.$$
% Similarly, $i_2$ is trace preserving. To lighten the notation, we identify $\mathcal{N}_k$ with the subalgebra $i_2(\mathcal{N}_k)$ in $\mathcal{N}_m$ and write $x$ instead of $i_2(x).$

Let $\varrho_k:{\mathcal{N}}_1\to {\mathcal{N}}_m,$ $0\leq k<m,$ be the unital $*-$homomorphism defined by the evaluation at the $k-$th coordinate. That is,
$$(\varrho_k(x))(n_0,\cdots,n_{m-1})=x(n_k),\quad (n_0,\cdots,n_{m-1})\in\mathbb{Z}_+^m.$$
We have
$$\tau_m(\varrho_k(x))=\lim\limits_{n_0\to\omega}\cdots\lim\limits_{n_{m-1}\to\omega}\tau(x(n_k))=\lim\limits_{n_k\to\omega}\tau(x(n_k))=\lim\limits_{n_0\to\omega}\tau(x(n_0))=\tau_1(x).$$
That is, $\varrho_k$ is trace preserving for every $0\leq k<m.$

In what follows, $M_m(\mathbb{C})$, $m\ge 1$ is equipped with the normalised trace $\frac1m{\rm Tr}.$ We identify the tensor product algebra $(M_m(\mathbb{C}),\frac1m{\rm Tr})^{\otimes m}$ with the algebra $M_{m^m}(\mathbb{C})$ equipped with a normalised trace. Let $\theta_k:M_m(\mathbb{C})\to M_{m^m}(\mathbb{C}),$ $0\leq k<m,$ be a unital $*-$homomorphism defined by the setting
$$\theta_k(A)=1^{\otimes k}\otimes x\otimes 1^{\otimes (m-1-k)},\quad x\in M_m(\mathbb{C}).$$
In what follows, $\Omega_m=\{0,1,\cdots,m-1\}$ (respectively, $\Omega_{m^m}=\{0,\cdots,m^m-1\}$) equipped with the normalised measure and the algebra $L_{\infty}(\Omega_m)$ (respectively, $L_{\infty}(\Omega_{m^m})$) is identified with the diagonal subalgebra of $M_m(\mathbb{C})$ (respectively, of $M_{m^m}(\mathbb{C})$). Each $\theta_k,$ $0\leq k<m,$ maps a diagonal subalgebra $L_{\infty}(\Omega_m)$ of $M_m(\mathbb{C})$ into the diagonal subalgebra $L_{\infty}(\Omega_{m^m})$ of $M_{m^m}(\mathbb{C}).$

Finally, let us define the probability space $$(\mathcal{M}_m,\sigma_m):=(M_{m^m}(\mathbb{C}),\frac1{m^m}{\rm Tr})\otimes (\mathcal{N}_m,\tau_m).$$

\subsection{Equivalent representation for the norm $\|\cdot\|_{\mathcal{I}}$}

In this subsection, we assume that the commutative core $I$ of a Banach ideal $\mathcal{I}$ admits an isomorphic embedding $T:I\to \mathcal{L}_p(\mathcal{N},\tau),$ $1\leq p<2$. Our main objective in this subsection is Lemma \ref{prelim embedding lemma} providing an equivalent representation of the norm $\|\cdot\|_{\mathcal{I}}$.

\begin{lem}\label{equi1} Let $I\neq l_p$ be a symmetric sequence space and let $T:I\to \mathcal{L}_p(\mathcal{N},\tau),$ $1\leq p<2,$ be an isomorphic embedding. If $\{e_n\}_{n\geq0}$ is a standard basis of $I,$ then the sequence $\{|T(e_n)|^p\}_{n\geq0}$ is uniformly integrable.
\end{lem}
\begin{proof} Assume the contrary. By  \cite[Corollary 2.11]{hrs}, there exists a strictly increasing sequence $n_k,$ $k\geq0,$ such that the sequence $\{T(e_{n_k})\}_{k\geq0}$ is equivalent to the standard basis of $l_p.$ Thus, the sequence $\{e_{n_k}\}_{k\geq0}$ is equivalent to the standard basis of $l_p.$ Since $\{e_n\}_{n\geq0}$ is a symmetric basis of $I,$ it follows that $\{e_n\}_{n\geq0}$ is equivalent to the standard basis of $l_p.$ In particular, we have $I=l_p$, which is a contradiction with the assumption.
\end{proof}

The following lemma provides an equivalent representation for the norm $\|\cdot\|_{{I}}$.

\begin{lem}\label{junge pos lemma} Let $I$ be a symmetric sequence space. For every $\alpha\in\mathbb{C}^m,$ we have\footnote{Here, the notation $f\otimes e_k,$ $f\in L_p(\Omega_{m^m})$ stands for the element of $L_p(\Omega_{m^m},I)$ acting by $u\to f(u)e_k$ for every $u\in\Omega_{m^m}.$}
$$\Big\|\sum_{k=0}^{m-1}\theta_k(\alpha)\otimes e_k\Big\|_{L_p(\Omega_{m^m},I)}\stackrel{C_{abs}}{\sim}\|\alpha\|_I,\quad 1\leq p\leq 2,$$
Here, $C_{abs}$ is the absolute constant.
\end{lem}
\begin{proof} For convenience of the reader, we  restate here the main result of \cite{Junge-pos}. If $\{f_k\}_{k=0}^{m-1}$ is a sequence of independent random variables on a probability space $\Omega_{m^m}$, then
\begin{align*}\Big\|\sum_{k=0}^{m-1}f_k\otimes e_k\Big\|_{L_p(\Omega_{m^m},I)}\stackrel{C_{abs}}{\sim}\Big(m\int_0^{1/m}\mu^p(s,h)ds\Big)^{\frac1p}
 +\Big\|\sum_{i=1}^m\Big(m\int_{(i-1)/m}^{i/m}\mu^p(s,h)ds\Big)^{\frac1p}e_i\Big\|_I,\end{align*}
where the function $h\in L_{\infty}(\Omega_{m^m}\times (0,1))$ is given by the formula
$$h(u,t)=\sum_{i=1}^m\chi_{(\frac{i-1}{m},\frac{i}{m})}(t)f_i(u),\quad u\in\Omega_{m^m},t\in[0,1].$$
Here, $\chi_{(\frac{i-1}{m},\frac{i}{m})}$ is the indicator function of the interval $(\frac{i-1}{m},\frac{i}{m}).$ Setting $f_k=\theta_k(\alpha),$ we see that $h$ is equimeasurable with $\alpha$ (considered as an element of $L_{\infty}(\Omega_m)$ equipped with a normalised measure). Therefore,
$$m\int_{(i-1)/m}^{i/m}\mu^p(s,h)ds=\mu^p(i-1,\alpha),\quad 1\leq i\leq m.$$
In the latter formula, $\mu(\alpha)$ is computed with respect to the counting measure. In this setting, the above result reads as follows
$$\Big\|\sum_{k=0}^{m-1}\theta_k(\alpha)\otimes e_k\Big\|_{L_p(\Omega_{m^m},I)}\stackrel{C_{abs}}{\sim}\|\alpha\|_{\infty}+\|\alpha\|_I.$$
Since $\|\alpha\|_{\infty}\leq \|\alpha\|_I$, the assertion follows.
\end{proof}

% {\color{blue} IT IS USELESS. TO BE DELETED. The following lemma may be viewed as a simple case of noncommutative Fubini Theorem.
%
% \begin{lem}\label{treqlpeq} For every $z\in\mathcal{L}_1(\mathcal{M}_m,\sigma_m),$ we have
% \begin{equation}\label{treq}
% (m^{-m}{\rm Tr}\otimes\tau_m)(z)=\lim\limits_{n_0\to\omega}\cdots\lim\limits_{n_{m-1}\to\omega}(m^{-m}{\rm Tr}\otimes\tau)(z(n_0,\cdots,n_{m-1})).
% \end{equation}
% Consequently, for every $u\in\mathcal{L}_p(\mathcal{M}_m,\sigma_m),$ we have
% \begin{equation}\label{lpeq}
% \|u\|_p=\lim\limits_{n_0\to\omega}\cdots\lim\limits_{n_{m-1}\to\omega}\|u(n_0,\cdots,n_{m-1})\|_{\mathcal{L}_p(M_{m^m}(\mathbb{C})\otimes\mathcal{N})}.
% \end{equation}
% \end{lem}
% \begin{proof} Every $z\in\mathcal{L}_1(\mathcal{M}_m,\sigma_m)$ admits a representation of the form
% $$z=\sum_{i,j=0}^{m^m-1}U_{ij}\otimes z_{ij},$$
% where $U_{ij},$ $0\leq i,j<m^m$ are matrix units and $z_{ij}\in\mathcal{L}_1(\mathcal{N}_1,\tau_1)$ for $0\leq i,j<m^m.$ It suffices to verify \eqref{treq} for elements of the form $U_{ij}\otimes z_{ij}.$ For such an element,
% $$(m^{-m}{\rm Tr}\otimes\tau_m)(z)=m^{-m}{\rm Tr}(U_{ij})\tau_m(z_{ij})$$
% $$\stackrel{def}{=}m^{-m}{\rm Tr}(U_{ij})\lim\limits_{n_0\to\omega}\cdots\lim\limits_{n_{m-1}\to\omega}\tau(z_{ij}(n_0,\cdots,n_{m-1}))$$
% $$=\lim\limits_{n_0\to\omega}\cdots\lim\limits_{n_{m-1}\to\omega}(m^{-m}{\rm Tr}\otimes\tau)(z(n_0,\cdots,n_{m-1})).$$
%
% The second assertion follows from the first one by putting $z=|u|^p$ and taking into account that every ultrafilter is a homomorphism on $l_\infty(\mathbb{Z}_+)$.
% \end{proof}
%
% }

In the next lemma, we provide an equivalent representation for the norm $\|\cdot\|_{\mathcal{I}}$ of the Banach ideal $\mathcal{I}$ under the assumption that the commutative core $I$ admits an isomorphic embedding into $\mathcal{L}_p(\mathcal{N},\tau).$ We recall that trace preserving $*-$homomorphisms $\{\varrho_k\}_{k=0}^{m-1}$ from $\mathcal{N}_1$ to $\mathcal{N}_m$ were introduced in the preceding subsection.

\begin{lem}\label{prelim embedding lemma} Let $I\neq l_p$ be a symmetric sequence space and let $T:I\to \mathcal{L}_p(\mathcal{N},\tau)$ be an isomorphic embedding. There exists\footnote{Note that $F$ does not depend on the choice of $m.$} $F\in \mathcal{L}_p(\mathcal{N}_1,\tau_1)$ such that
$$\|A\|_{\mathcal{I}}\stackrel{C(I)}{\sim}\Big\|\sum_{k=0}^{m-1}\theta_k(A)\otimes\varrho_k(F)\Big\|_{\mathcal{L}_p(\mathcal{M}_m,\sigma_m)},\quad A=A^*\in M_m(\mathbb{C}).$$
Here, $C(I)$ is the constant which only depends on $I$ (but not on the integer $m\ge 1$) and whose value may change from line to line.
\end{lem}
\begin{proof} Set $F=\{T(e_k)\}_{k\geq0}\in\mathcal{L}_p(\mathcal{N},\tau)^{\omega}.$ It follows from Lemma \ref{equi1} that $|F|^p$ is a uniformly integrable sequence. Using Lemma \ref{equi2}, we infer that
$$F\in\mathcal{L}_p(\mathcal{N}^{\omega},\tau^{\omega})\stackrel{def}{=}\mathcal{L}_p(\mathcal{N}_1,\tau_1).$$

Fix a self-adjoint operator $A\in M_m(\mathbb{C})$ and select a unitary operator $U\in M_m(\mathbb{C})$ such that\footnote{Here, $\lambda(A)\in L_{\infty}(\Omega_m)$ denotes the eigenvalue sequence of the operator $A.$} $A=U^{-1}\lambda(A)U.$ Define a unitary operator $V\in M_{m^m}(\mathbb{C})\otimes\mathcal{N}$ by setting
$$V=\Big(\prod_{k=0}^{m-1}\theta_k(U)\Big)\otimes 1.$$
We have
$$V^{-1}\Big(\sum_{k=0}^{m-1}\theta_k(\lambda(A))\otimes T(e_{n_k})\Big)V=\sum_{k=0}^{m-1}\theta_k(A)\otimes T(e_{n_k}).$$
Here, $(n_0,\cdots,n_{m-1})\in\mathbb{Z}_+^m$ is arbitrary. Therefore, for the algebra $M_{m^m}(\mathbb{C})\otimes\mathcal{N}$ equipped with a tensor product trace, we have
$$\Big\|\sum_{k=0}^{m-1}\theta_k(A)\otimes T(e_{n_k})\Big\|_{\mathcal{L}_p(M_{m^m}(\mathbb{C})\otimes\mathcal{N})}=\Big\|\sum_{k=0}^{m-1}\theta_k(\lambda(A))\otimes T(e_{n_k})\Big\|_{\mathcal{L}_p(M_{m^m}(\mathbb{C})\otimes\mathcal{N})}.$$
Since each $\theta_k$ maps the diagonal subalgebra $L_{\infty}(\Omega_m)$ of $M_m(\mathbb{C})$ into the diagonal subalgebra $L_{\infty}(\Omega_{m^m})$ of $M_{m^m}(\mathbb{C}),$ it follows that
$$\Big\|\sum_{k=0}^{m-1}\theta_k(A)\otimes T(e_{n_k})\Big\|_{\mathcal{L}_p(M_{m^m}(\mathbb{C})\otimes\mathcal{N})}=\Big\|\sum_{k=0}^{m-1}\theta_k(\lambda(A))\otimes T(e_{n_k})\Big\|_{L_p(\Omega_{m^m},L_p({\mathcal{N}},\tau))},$$
where on the right hand side we use the Bochner space norm. Taking $C(I)\geq\|T\|,\|T^{-1}\|,$ we obtain
$$\Big\|\sum_{k=0}^{m-1}\theta_k(\lambda(A))\otimes T(e_{n_k})\Big\|_{L_p(\Omega_{m^m},L_p({\mathcal{N}},\tau))}\stackrel{C(I)}{\sim}\Big\|\sum_{k=0}^{m-1}\theta_k(\lambda(A))\otimes e_{n_k}\Big\|_{L_p(\Omega_{m^m},I)}.$$
Therefore, we have
$$\Big\|\sum_{k=0}^{m-1}\theta_k(A)\otimes T(e_{n_k})\Big\|_{\mathcal{L}_p(M_{m^m}(\mathbb{C})\otimes\mathcal{N})}
\stackrel{C(I)}{\sim}\Big\|\sum_{k=0}^{m-1}\theta_k(\lambda(A))\otimes e_{n_k}\Big\|_{L_p(\Omega_{m^m},I)}.$$
Since $\{e_k\}_{k\geq0}$ is a symmetric basis in $I,$ it follows that
$$\Big\|\sum_{k=0}^{m-1}\theta_k(A)\otimes T(e_{n_k})\Big\|_{\mathcal{L}_p(M_{m^m}(\mathbb{C})\otimes\mathcal{N})}\stackrel{C(I)}{\sim}\Big\|\sum_{k=0}^{m-1}\theta_k(\lambda(A))\otimes e_k\Big\|_{L_p(\Omega_{m^m},I)}$$
whenever $n_0<n_1<\cdots<n_{m-1}.$ It follows now from Lemma \ref{junge pos lemma} that
\begin{equation}\label{topi}
\Big\|\sum_{k=0}^{m-1}\theta_k(A)\otimes T(e_{n_k})\Big\|_{\mathcal{L}_p(M_{m^m}(\mathbb{C})\otimes\mathcal{N})}\stackrel{C(I)}{\sim}\|\lambda(A)\|_I=\|A\|_{\mathcal{I}}
\end{equation}
whenever $n_0<n_1<\cdots<n_{m-1}.$ By definitions of $F$ and $\varrho_k,$ $0\leq k<m,$ we have that
$$(\varrho_k(F))(n_0,\cdots,n_{m-1})=T(e_{n_k}),\quad (n_0,\cdots,n_{m-1})\in\mathbb{Z}_+^m.$$
Therefore, we can restate \eqref{topi} as follows:
$$\Big\|\sum_{k=0}^{m-1}\theta_k(A)\otimes (\varrho_k(F))(n_0,\cdots,n_{m-1})\Big\|_{\mathcal{L}_p(M_{m^m}(\mathbb{C})\otimes\mathcal{N})}\stackrel{C(I)}{\sim}\|A\|_{\mathcal{I}}$$
whenever $n_0<n_1<\cdots<n_{m-1}.$ Since $F\in\mathcal{L}_p(\mathcal{N}_1,\tau_1),$ it follows that $\varrho_k(F)\in\mathcal{L}_p(\mathcal{N}_m,\tau_m).$  Hence, $\sum_{k=0}^{m-1}\theta_k(A)\otimes\varrho_k(F)\in\mathcal{L}_p(M_{m^m}(\mathbb{C})\otimes\mathcal{N}_m).$ Clearly, $M_{m^m}(\mathbb{C})\otimes\mathcal{N}_m$ is an ultrapower of $M_{m^m}(\mathbb{C})\otimes\mathcal{N}.$ Thus, we have
\begin{align*}
&\Big\|\sum_{k=0}^{m-1}\theta_k(A)\otimes \varrho_k(F)\Big\|_{\mathcal{L}_p(M_{m^m}(\mathbb{C})\otimes\mathcal{N}_m)}\\
&=\lim\limits_{n_0\to\omega}\cdots\lim\limits_{n_{m-1}\to\omega}\Big\|\sum_{k=0}^{m-1}\theta_k(A)\otimes (\varrho_k(F))(n_0,\cdots,n_{m-1})\Big\|_{\mathcal{L}_p(M_{m^m}(\mathbb{C})\otimes\mathcal{N})}
 \stackrel{C(I)}{\sim}\|A\|_{\mathcal{I}}.
\end{align*}
This proves the assertion.
\end{proof}

\subsection{Proof of uniform integrability}

In this subsection, we show that, for every rank one projection $p,$ there exists an element $z\in L_p(0,1)$ such that
$$\mu(\sum_{k=0}^{m-1}\theta_k(p)\otimes\varrho_k(F))\leq\mu(z),\quad m\geq 1,$$
where $F$ is defined in Lemma \ref{prelim embedding lemma} above. The following lemma shows that $\{\varrho_k\}_{k\geq0}$ are subsymmetric copies in the sense of Koestler.

\begin{lem}\label{sum equi} For every $\mathcal{A}\subset\{0,\cdots,m-1\},$ we have
$$\mu\Big(\sum_{k\in\mathcal{A}}\varrho_k(x)\Big)=\mu\Big(\sum_{k=0}^{|\mathcal{A}|-1}\varrho_k(x)\Big),\quad x\in\mathcal{S}(\mathcal{N}_1,\tau_1).$$
\end{lem}
\begin{proof} Without loss of generality, $x$ is bounded. It suffices\footnote{Indeed, if all moments of bounded random variables $x$ and $y$ coincide, then so do their characteristic functions. Hence, their distributions also coincide.} to show that
$$\tau_m\Big(\Big(\sum_{k\in\mathcal{A}}\varrho_k(x)\Big)^N\Big)=\tau_m\Big(\Big(\sum_{k=0}^{M-1}\varrho_k(x)\Big)^N\Big),\quad N\in\mathbb{N},$$
where we denoted for brevity, $M=|\mathcal{A}|.$ Let $\mathcal{A}=\{k_0,\cdots,k_{M-1}\}.$ It is clear that
$$\tau_m\Big(\Big(\sum_{k\in\mathcal{A}}\varrho_k(x)\Big)^N\Big)=\lim\limits_{n_0\to\omega}\cdots\lim\limits_{n_{m-1}\to\omega}\tau\Big(\Big(\sum_{i=0}^{M-1}x(n_{k_i})\Big)^N\Big)=$$
$$=\lim\limits_{n_{k_0}\to\omega}\cdots\lim\limits_{n_{k_{M-1}}\to\omega}\tau\Big(\Big(\sum_{i=0}^{M-1}x(n_{k_i})\Big)^N\Big).$$
Renaming the index $n_{k_0}$ with $n_0,$ we obtain
$$\tau_m\Big(\Big(\sum_{k\in\mathcal{A}}\varrho_k(x)\Big)^N\Big)=\lim\limits_{n_0\to\omega}\lim\limits_{n_{k_1}\to\omega}\cdots\lim\limits_{n_{k_{M-1}}\to\omega}\tau\Big(\Big(x(n_0)+\sum_{i=1}^{M-1}x(n_{k_i})\Big)^N\Big).$$
Renaming the index $n_{k_1}$ with $n_1,$ we obtain
$$\tau_m\Big(\Big(\sum_{k\in\mathcal{A}}\varrho_k(x)\Big)^N\Big)=\lim\limits_{n_0\to\omega}\lim\limits_{n_1\to\omega}\lim\limits_{n_{k_2}\to\omega}\cdots\lim\limits_{n_{k_{M-1}}\to\omega}\tau\Big(\Big(x(n_0)+x(n_1)+\sum_{i=2}^{M-1}x(n_{k_i})\Big)^N\Big).$$
Repeating the process, we conclude the proof.
\end{proof}

Consider the probability space $(\mathcal{N}_{{\rm ext}},\tau_{{\rm ext}})$
$$\mathcal{N}_{{\rm ext}}=\mathbb{C}\bigoplus\Big(\bigoplus_{m\geq1}\mathcal{N}_m\Big),\quad \tau_{{\rm ext}}=\frac1e\bigoplus\Big(\bigoplus_{k=0}^{\infty}\frac1{e\cdot m!}\tau_m\Big).$$
Define the mapping $\mathscr{L}:\mathcal{L}_1(\mathcal{N}_1,\tau_1)\to\mathcal{L}_1(\mathcal{N}_{{\rm ext}},\tau_{{\rm ext}})$ defined by the setting
$$\mathscr{L}(x)=0\bigoplus\Big(\bigoplus_{m\geq 1}\Big(\sum_{k=0}^{m-1}\varrho_k(x)\Big)\Big),\quad x\in\mathcal{L}_1(\mathcal{N}_1,\tau_1).$$
The expression for the mapping $\mathscr{L}$ could be understood as a subsymmetric version of the Kruglov operator replacing tensor copies by the copies given by the $\{\rho_k\}_{k\geq0}.$ However, we do not need these properties and, hence, do not state and prove them.

\begin{lem}\label{l bound} $\mathscr{L}:\mathcal{L}_p(\mathcal{N}_1,\tau_1)\to\mathcal{L}_p(\mathcal{N}_{{\rm ext}},\tau_{{\rm ext}})$ for all $1\leq p\leq 2.$
\end{lem}
\begin{proof} Let $0\leq x\in\mathcal{L}_2(\mathcal{N}_1,\tau_1).$ We have
$$\tau_{{\rm ext}}\big((\mathscr{L}x)^2\big)=\sum_{m\geq 1}\frac1{e\cdot m!}\tau_m\Big(\Big(\sum_{k=0}^{m-1}\varrho_k(x)\Big)^2\Big)=\sum_{m\geq 1}\frac1{e\cdot m!}\sum_{k_1,k_2=0}^{m-1}\tau_m(\varrho_{k_1}(x)\varrho_{k_2}(x)).$$
Since the $*-$homomorphisms $\varrho_k,$ $0\leq k<m,$ are trace preserving, it follows that
$$0\leq\tau_m(\varrho_{k_1}(x)\varrho_{k_2}(x))\leq\left(\tau_m((\varrho_{k_1}(x))^2)\tau_m((\varrho_{k_2}(x))^2)\right)^{1/2}=\tau_1(x^2).$$
Therefore, we have
$$\tau_{{\rm ext}}((\mathscr{L}x)^2)\leq\sum_{m\geq 1}\frac{m^2}{e\cdot m!}\tau(x^2)=2\tau(x^2).$$
It is immediate that $\|\mathscr{L}\|_{\mathcal{L}_2\to\mathcal{L}_2}\leq\sqrt{2}.$ Since $\tau_{{\rm ext}}(\mathscr{L}x)=\tau_1(x)$ for every $x\in\mathcal{L}_1(\mathcal{N}_1,\tau_1),$ it follows that $\|\mathscr{L}\|_{\mathcal{L}_1\to\mathcal{L}_1}\leq 1.$ The assertion follows by interpolation.
\end{proof}

\begin{lem}\label{bad to good} For every rank one projection $p\in M_m(\mathbb{C}),$ we have
$$\mu\Big(\sum_{k=0}^{m-1}\theta_k(p)\otimes\varrho_k(x)\Big)\leq D_e\mu(\mathscr{L}(x)),\quad x\in\mathcal{S}(\mathcal{N}_1,\tau_1).$$
Here, the left hand side is computed in $(\mathcal{M}_m,\sigma_m)$ and the right hand side is computed in $(\mathcal{N}_{{\rm ext}},\tau_{{\rm ext}}).$
\end{lem}
\begin{proof} Consider the family $\{\theta_k(p)\}_{k=0}^{m-1}$ of commuting projections with trace $\frac1m$. Define the family of $2^m$ pairwise orthogonal projections
$$p_{\mathcal{A}}=\Big(\prod_{k\in\mathcal{A}}\theta_k(p)\Big)\cdot\Big(\prod_{k\notin\mathcal{A}}(1-\theta_k(p))\Big),\quad \mathcal{A}\subset\{0,1,\cdots,m-1\}.$$
A moment's reflection shows that $\sum_{\mathcal{A}}p_{\mathcal{A}}=1.$ We have
$$p_{\mathcal{A}}\theta_k(p)=
\begin{cases}
p_{\mathcal{A}},\quad k\in\mathcal{A}\\
0,\quad k\notin\mathcal{A}.
\end{cases}
$$
Therefore,
$$\sum_{k=0}^{m-1}\theta_k(p)\otimes\varrho_k(x)=\sum_{\mathcal{A}}\sum_{k=0}^{m-1}p_{\mathcal{A}}\theta_k(p)\otimes\varrho_k(x)= \sum_{\mathcal{A}}p_{\mathcal{A}}\otimes\Big(\sum_{k\in\mathcal{A}}\varrho_k(x)\Big).$$
Since the projections $p_{\mathcal{A}},$ $\mathcal{A}\subset\{0,1,\cdots,m-1\},$ are pairwise orthogonal, it follows from Lemma \ref{sum equi} that
$$\mu\Big(\sum_{k=0}^{m-1}\theta_k(p)\otimes\varrho_k(x)\Big)=\mu\Big(\sum_{\mathcal{A}}p_{\mathcal{A}}\otimes\Big(\sum_{k=0}^{|\mathcal{A}|-1}\varrho_k(x)\Big)\Big).$$
Set
$$p_l=\sum_{|\mathcal{A}|=l}p_{\mathcal{A}},\quad 0\leq l\leq m.$$
These projections are pairwise orthogonal and $\sum_{l=0}^mp_l=1.$ We have
\begin{equation}\label{esqqq}
\mu\Big(\sum_{k=0}^{m-1}\theta_k(p)\otimes\varrho_k(x)\Big)=\mu\Big(\sum_{l=0}^mp_l\otimes\Big(\sum_{k=0}^{l-1}\varrho_k(x)\Big)\Big).
\end{equation}

Since the projections $\theta_k(p),$ $0\leq k<m,$ are independent (see Subsection \ref{indep subsect}), it is immediate from the definition of $p_{\mathcal{A}}$ that
$$\frac1{m^m}{\rm Tr}(p_{\mathcal{A}})=\Big(\frac1m\Big)^{|\mathcal{A}|}\Big(1-\frac1m\Big)^{m-|\mathcal{A}|},\quad\mathcal{A}\subset\{0,1,\cdots,m-1\}.$$
Therefore, we have
$$\frac1{m^m}{\rm Tr}(p_l)=\binom{m}{l}\Big(\frac1m\Big)^l\Big(1-\frac1m\Big)^{m-l}\leq\frac1{l!}=e\cdot\frac1{e\cdot l!},\quad 0\leq l\leq m.$$
The assertion follows by combining \eqref{esqqq} with the definition of $\mathscr{L}(x)$ and the definition of the dilation operator.
\end{proof}

\subsection{Proof of Theorem \ref{second main theorem}}
We are now ready to prove our second main result.

\begin{proof}[Proof of Theorem \ref{second main theorem}] The proof of the implication $\eqref{ji}\to\eqref{jii}$ is immediate. We now proceed with the proof of the implication $\eqref{jii}\to\eqref{ji}.$ Suppose that there exists an isomorphic embedding from the commutative core $I$ of the Banach ideal $\mathcal{I}$ to $\mathcal{L}_p(\mathcal{N},\tau)$ for some probability space $(\mathcal{N},\tau).$ Let $\{p_m\}_{m\geq0}$ be an increasing sequence such that ${\rm Tr}(p_m)=m.$ We identify $p_m\mathcal{L}(H)p_m=M_m(\mathbb{C}).$ Take an element $F\in\mathcal{L}_p(\mathcal{N}_1,\tau_1)$ constructed in Lemma \ref{prelim embedding lemma}. Define the mapping $T_m:\mathcal{C}_{00}\to\mathcal{L}_p(\mathcal{M}_m,\sigma_m)$ by setting
$$T_m(A)\stackrel{def}{=}\sum_{k=0}^{m-1}\theta_k(p_mAp_m)\otimes\varrho_k(F),\quad m\geq 0,$$
where $\mathcal{C}_{00}$ denotes the ideal of finite rank
operators in $\mathcal{L}(H).$ Set
$$T(A)\stackrel{def}{=}\{T_m(A)\}_{m\geq0}\in\Big(\mathcal{L}_p(\mathcal{M}_m,\sigma_m)\Big)_{m\geq0}^{\omega}.$$
If $A=\sum_{k=0}^{N-1}\lambda(k,A)q_k,$
where each $q_k,$ $0\leq k<N,$ is a rank one projection, then (see \eqref{kps dilate ineq})
$$\mu(T_m(A))=\mu\Big(\sum_{k=0}^{N-1}\lambda(k,A)T_m(q_k)\Big)\leq\|A\|_{\infty} N D_N\mu(T_m(q_0)).$$
Using Lemma \ref{bad to good}, we infer that
$$\mu(T_m(A))\leq N\|A\|_{\infty} D_{eN}\mu(\mathscr{L}(F)).$$
Since $\mathscr{L}(F)\in\mathcal{L}_p(\mathcal{N}_{{\rm ext}},\tau_{{\rm ext}}),$ it follows that the sequence $\{|T_m(A)|^p\}_{m\geq 0}$ is uniformly integrable. By Lemma \ref{equi2}, we have
$$T(A)\in\mathcal{L}_p(\Big(\mathcal{M}_m\Big)_{m\geq0}^{\omega},\Big(\sigma_m\Big)_{m\geq0}^{\omega}),\quad A\in\mathcal{C}_{00}.$$
It follows immediately from Lemma \ref{prelim embedding lemma} that
$$\|T(A)\|_p=\lim_{m\to\omega}\|T_m(A)\|_p\sim\lim_{m\to\omega}\|p_mAp_m\|_{\mathcal{I}}=\|A\|_{\mathcal{I}},\quad A\in\mathcal{C}_{00}.$$
Thus, $T$ extends to an isomorphic embedding
$$T:\mathcal{I}\to\mathcal{L}_p(\Big(\mathcal{M}_m\Big)_{m\geq0}^{\omega},\Big(\sigma_m\Big)_{m\geq0}^{\omega}).$$
\end{proof}

\appendix

\section{}
 The
following result shows that for the Orlicz ideal the notion of
$p$-convexity (resp., $q$-concavity) coincides with upper
$p$-estimate (resp., lower $q$-estimate).

 \begin{thm}\label{prop_O} Let $1\le p\le 2$ and $2\le q<\infty.$ The following conditions are equivalent
\begin{enumerate}[{\rm (i)}]
\item\label{pri} The Orlicz ideal $\mathcal L_M(H)$ is $p$-convex and $q$-concave.
\item \label{prii} The Orlicz ideal $\mathcal L_M(H)$ satisfies an upper $p$-estimate and a lower $q$-estimate.
\item \label{priii} The Orlicz space  $l_M$ is $p$-convex and $q$-concave.
    \item\label{priv} The Orlicz space  $l_M$  satisfies an upper $p$-estimate and a lower $q$-estimate.
    \item\label{prv} The Orlicz function $M$ on $(0,\infty)$ is equivalent on $[0,1]$ to some $p$-convex and $q$-concave Orlicz function.
\end{enumerate}
  \end{thm}
  \begin{proof} The implications \eqref{pri}$\Rightarrow$\eqref{priii} and \eqref{prii}$\Rightarrow$\eqref{priv} are obvious.
  The equivalence of \eqref{priii} and \eqref{priv} is proved in \cite[p. 121,124]{Mal}.
  The implications  \eqref{priii}$\Rightarrow$\eqref{pri} and \eqref{priii}$\Rightarrow$\eqref{prii} are shown in \cite[Theorems 3.8 and 4.7]{dds} and
  \cite[Proposition 5.1]{dds}, respectively.

  \eqref{priv}$\Rightarrow$\eqref{prv}.
Let $l_M$ satisfies upper $p$-estimate and lower $q$-estimate.
Without loss of generality, we may assume that $M(1)=1.$ Applying
\eqref{eq__up_p_est} and \eqref{eq_low_q_est} to the sequence
$$\sum_{j=0}^{nm-1}e_j=\sum_{k=0}^{n-1}\sum_{j=(k-1)m}^{km-1}e_j\in
l_M,$$ where $\{e_j\}_{j=0}^\infty$ is the standard vector basis
in $l_M$ and observing
$$\Big\|\sum_{j=(k-1)m}^{km-1}e_j\Big\|_{l_M}=\Big\|\sum_{j=0}^{m-1}e_j\Big\|_{l_M}
\ \ \ \mbox{for any} \ \ \ 1\le k\le n,$$ we obtain that
\begin{equation}\label{eq_(ii)to(iii)1}{\rm const}\cdot n^{1/q}\Big\|\sum_{j=0}^{m-1}e_j\Big\|_{l_M}\leq\Big\|\sum_{j=0}^{nm-1}e_j\Big\|_{l_M}\leq{\rm const}\cdot n^{1/p}\Big\|\sum_{j=0}^{m-1}e_j\Big\|_{l_M}\end{equation}
for all $n,m\in\mathbb{N}.$

Let $0< s\le t\le 1.$ Set $u:=\frac1{M(s)}$ and $v:=\frac1{M(t)}.$
Then $u\ge v\geq 1.$ Let $m:=[v]$ and
$n:=\Big[\frac{u}{m}\Big]+1$. Observe that since $v<2m$ and $v\le
u$, we have
\begin{equation}\label{eq_(ii)to(iii)2}
n\le \frac{u}{m}+1< \frac{2u}{v}+\frac{u}{v}=\frac{3u}{v} \ \
\mbox{and} \ \ n=\Big[\frac{u}{m}\Big]+1\ge \frac{u}m.
\end{equation}
Observe also that (see e.g. \cite[I, 1.2, Example 9]{R-R})
\begin{equation}\label{eq_(ii)to(iii)3}\Big\|\sum_{j=0}^{m-1}e_j\Big\|_{l_M}=\frac1{M^{-1}\Big(\frac{1}{m}\Big)}
\ \ \ \mbox{for any} \ \ \ m\in\mathbb N.\end{equation} Now using
the fact that $M^{-1}$ is increasing and applying the right hand
side inequality of \eqref{eq_(ii)to(iii)1} together with
\eqref{eq_(ii)to(iii)3}, we obtain
$$t=M^{-1}\Big(\frac1v\Big)\le M^{-1}\Big(\frac1m\Big)\leq$$
$$\stackrel{\eqref{eq_(ii)to(iii)1}}{\le} {\rm const} \cdot n^{\frac1p} M^{-1}\Big(\frac1{mn}\Big)
\stackrel{\eqref{eq_(ii)to(iii)2}}{\le} {\rm const} \cdot
\Big(\frac{u}{v}\Big)^{\frac1p} M^{-1}\Big(\frac1u\Big)={\rm
const} \cdot\Big(\frac{M(t)}{M(s)}\Big)^{\frac1p} \cdot s.$$
 Raising to the power $p$ both sides of the inequality above, we obtain \eqref{pconv crit}.  The left hand side inequality of \eqref{eq_(ii)to(iii)1} similarly implies \eqref{2conc crit}. By Lemma \ref{as lemma}, there exists a $p$-convex and $q$-concave Orlicz function equivalent to $M.$

\eqref{prv}$\Rightarrow$\eqref{priii}. Suppose that an Orlicz
function $M$ satisfies \eqref{prv}.  Without loss of generality we
may assume that
 $M$ is $p$-convex. Define a convex Orlicz function $M_1$ by setting $$M_1(t):=M(t^{1/p}), \ \ \ t>0.$$
  Since for any finite
sequence $\{x_k\}_{k=1}^n\subseteq l_M,$
   \begin{equation*} \Big\Vert \Big(\sum_{k=1}^n\vert x_k\vert^p\Big)^{1/p}
\Big\Vert_{l_M}=\Big\Vert \sum_{k=1}^n\vert x_k\vert ^p
\Big\Vert_{l_{M_1}}^{1/p}\le \Big(\sum_{k=1}^n\Vert x_k\Vert
_{l_{M_1}}^p\Big)^{1/p} \leq \Big(\sum_{k=1}^n\Vert
x_k\Vert_{l_M}^p\Big)^{1/p},
  \end{equation*}
  it follows that $l_M$ is $p$-convex. Similarly, one can deduce that $l_M$ is $q$-concave from $q$-concavity of $M.$
 \end{proof}

\end{document}